\documentclass{article}
\usepackage[utf8]{inputenc}
\usepackage{amsmath,amsthm,amsfonts,amssymb,bm, mathdots}
\usepackage{xcolor}
\usepackage{geometry}
\usepackage{tikz}
\usepackage{ytableau}
\usepackage{authblk}
\usepackage{enumitem}
\usepackage{hyperref}
\usepackage{indentfirst}
\usepackage{url}
\usepackage{fdsymbol}
\usepackage{adjustbox}

\usetikzlibrary{decorations.pathreplacing}

\newcommand{\vertex}[1]
{
\foreach \i in {1,...,#1}
{\draw[fill=violet] (0,2*\i-2) rectangle (2*#1-1,2*\i-1);
\draw[fill=gray] (0,2*\i-1) rectangle (2*#1-1,2*\i);
}
\draw[step=1.0] (0,0) grid (2*#1-1,2*#1);
\foreach \i in {1,...,#1}
{
\draw[very thick, blue, fill=white] (0,2*\i-2+0.5) circle (0.1);
\draw[very thick, blue,fill=white] (0,2*\i-1+0.5) circle (0.1);
\draw[very thick, blue,fill=white] (2*#1-1,2*\i-2+0.5) circle (0.1);
\draw[very thick, blue,fill=blue] (2*#1-1,2*\i-1+0.5) circle (0.1);
}
\foreach \i in {1,...,#1}
{
\draw[very thick, blue,fill=blue] (\i-0.5,0) circle (0.1);
\draw[very thick, blue, fill=white] (\i-0.5,2*#1) circle (0.1);
}
\foreach \i in {2,...,#1}
{
\draw[very thick,blue, fill=white] (#1+\i-1.5,0) circle (0.1);
\draw[very thick,blue,fill=white] (#1+\i-1.5,2*#1) circle (0.1);
}
}

\newcommand{\vertexb}[1]
{
\foreach \i in {1,...,#1}
{
\draw[fill=white] (0,2*\i-2) rectangle (2*#1-1,2*\i-1);
\draw[fill=pink] (0,2*\i-1) rectangle (2*#1-1,2*\i);
}
\draw[step=1.0] (0,0) grid (2*#1-1,2*#1);
\foreach \i in {1,...,#1}
{
\draw[very thick,blue,fill=white] (0,2*\i-2+0.5) circle (0.1);
\draw[very thick,blue,fill=blue] (0,2*\i-1+0.5) circle (0.1);
\draw[very thick,blue,fill=white] (2*#1-1,2*\i-2+0.5) circle (0.1);
\draw[very thick,blue,fill=white] (2*#1-1,2*\i-1+0.5) circle (0.1);
}
\foreach \i in {1,...,#1}
{
\draw[very thick,blue,fill=white] (#1+\i-1.5,0) circle (0.1);
\draw[very thick,blue,fill=blue] (\i-0.5,2*#1) circle (0.1);
}
\foreach \i in {2,...,#1}
{
\draw[very thick,blue,fill=blue] (\i-1.5,0) circle (0.1);
\draw[very thick,blue,fill=blue] (#1+\i-1.5,2*#1) circle (0.1);
}
}

\newcommand{\checkerboard}[1]
{
\foreach \i in {0,...,#1}
\foreach \j in {0,...,\i}
{
\draw[lightgray, fill={\ifodd\numexpr\i+\j\relax white\else lightgray\fi}] (\i,\j) rectangle (\i+1,\j+1);
\draw[lightgray, fill={\ifodd\numexpr\i+\j\relax lightgray\else white\fi}] (\i,-\j) rectangle (\i+1,-\j-1);
\draw[lightgray, fill={\ifodd\numexpr\i+\j\relax lightgray\else white\fi}] (2*#1-\i+2,\j) rectangle (2*#1-\i+1,\j+1);
\draw[lightgray, fill={\ifodd\numexpr\i+\j\relax white\else lightgray\fi}] (2*#1-\i+2,-\j) rectangle (2*#1-\i+1,-\j-1);
}
}

\newcommand{\vdom}[3][gray]
{
\draw[ultra thick, fill = #1] (#2,#3) rectangle (#2+1,#3+2); 
}

\newcommand{\hdom}[3][gray]
{
\draw[ultra thick, fill = #1] (#2,#3) rectangle (#2+2,#3+1); 
}

\newcommand{\M}{M}
\colorlet{pink}{violet!40!white}
\DeclareMathOperator{\weight}{weight}

\newcommand{\C}{\mathbb{C}}

\newcommand{\I}{\mathbf{I}}
\newcommand{\J}{\mathbf{J}}
\newcommand{\K}{\mathbf{K}}
\renewcommand{\L}{\mathbf{L}}

\theoremstyle{plain}
\newtheorem{thm}{Theorem}[section]
\newtheorem*{thm*}{Theorem}
\newtheorem{prop}[thm]{Proposition}
\newtheorem*{prop*}{Proposition}
\newtheorem{cor}[thm]{Corollary}
\newtheorem{lem}[thm]{Lemma}
\newtheorem*{lem*}{Lemma}
\newtheorem{example}[thm]{Example}

\newtheorem{remark}[thm]{Remark}
\newtheorem{problem}[thm]{Open problem}

\title{Colored vertex models and $k$-tilings of the Aztec diamond}
\author{Sylvie Corteel} 
\affil{Department of Mathematics, UC Berkeley USA\\

CNRS, IRIF, Universit\'e de Paris, France\\
{\small\ttfamily{corteel@berkeley.edu}}}
\author{Andrew Gitlin}
\affil{Department of Mathematics, UC Berkeley USA\\
{\small\ttfamily{andrew\_gitlin@berkeley.edu}}}
\author{David Keating}
\affil{Department of Mathematics, UW Madison USA\\
{\small\ttfamily{dkeating3@wisc.edu}}}

\date{\today}

\begin{document}

\maketitle

\begin{abstract}
We study $k$-tilings ($k$-tuples of domino tilings) of the Aztec diamond of rank $m$.  We assign a weight to each $k$-tiling, depending on the number of dominos of certain types and the number of ``coupled dominos" between the tilings.  Employing the colored vertex models introduced in \cite{GKsuper} to study supersymmetric LLT polynomials, we compute the generating polynomials of the $k$-tilings.  We then prove some combinatorial results about $k$-tilings, including a bijection between $k$-tilings with no coupled dominos and $1$-tilings, and we compute the arctic curves of the tilings for $t=0$ and $t\rightarrow\infty$.  We also present some lozenge $k$-tilings of the hexagon and compute the arctic curves of the tilings for $t=0$.
\end{abstract}

\section{Introduction}

Domino tilings of the Aztec diamond were first studied in 1992 in \cite{aztec0, aztec1}.  The Aztec diamond of rank $m$ is the set of lattice squares inside the diamond-shaped region
\[ 
AD_m=\{(x, y)\in \mathbb{R}^2 : |x| + |y| \leq m + 1\}.\]
The Aztec diamond of rank 3 is drawn on the left of Figure \ref{ex0}.
We tile this region with dominos. For example a tiling of the Aztec diamond of rank 3 is drawn on the right of Figure \ref{ex0}.
and the domino tilings of the Aztec diamond of rank 2 are displayed in Figure \ref{rank2}.
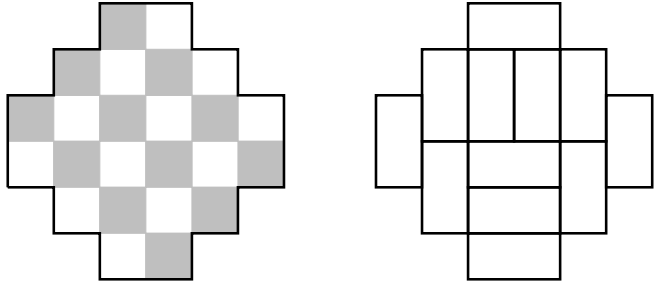
\begin{figure}[ht]
\hspace{3cm}
\resizebox{0.25\textwidth}{!}{
\begin{tikzpicture}[baseline = (current bounding box).center]
\checkerboard{2}
\draw[ultra thick](0,-1)--(0,1)--(1,1)--(1,2)--(2,2)--(2,3)--(4,3)--(4,2)--(5,2)--(5,1)--(6,1)--(6,-1)--(5,-1)--(5,-2)--(4,-2)--(4,-3)--(2,-3)--(2,-2)--(1,-2)--(1,-1)--(0,-1);
\end{tikzpicture}}\hspace{1cm}
\resizebox{0.25\textwidth}{!}{
\begin{tikzpicture}[baseline = (current bounding box).center]
\draw[ultra thick] (0,-1) rectangle (1,1); \draw[ultra thick] (1,-2) rectangle (2,0); \draw[ultra thick] (1,0) rectangle (2,2);
\draw[ultra thick] (2,-3) rectangle (4,-2); \draw[ultra thick] (2,-1) rectangle (4,0);
\draw[ultra thick] (2,2) rectangle (4,3);
\draw[ultra thick] (2,0) rectangle (3,2); \draw[ultra thick] (2,-2) rectangle (4,-1); \draw[ultra thick] (3,0) rectangle (4,2); 
\draw[ultra thick] (4,-2) rectangle (5,0); \draw[ultra thick] (4,0) rectangle (5,2); 
\draw[ultra thick] (5,-1) rectangle (6,1);       
\end{tikzpicture}}
\caption{The Aztec diamond of rank 3 (left) and an example of a domino tiling (right)}
\label{ex0}
\end{figure}
One surprising result is the following:
\begin{thm}
The number of domino tilings of the Aztec diamond of rank $m$ is $2^{m+1\choose 2}$.
\end{thm}

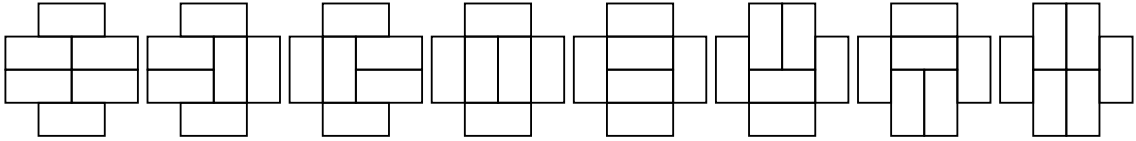
\begin{figure}[ht]\resizebox{\textwidth}{!}{
\begin{tikzpicture}[baseline = (current bounding box).center]
\draw[ultra thick](0,0) rectangle (2,1);
\draw[ultra thick](4,0) rectangle (2,1);
\draw[ultra thick](0,-1) rectangle (2,0);
\draw[ultra thick](4,-1) rectangle (2,0);
\draw[ultra thick](1,-2) rectangle (3,-1);
\draw[ultra thick](1,1) rectangle (3,2);
\end{tikzpicture}\ \ 
\begin{tikzpicture}[baseline = (current bounding box).center]
\draw[ultra thick](0,0) rectangle (2,1);
\draw[ultra thick](3,1) rectangle (2,-1);
\draw[ultra thick](0,-1) rectangle (2,0);
\draw[ultra thick](4,1) rectangle (3,-1);
\draw[ultra thick](1,-2) rectangle (3,-1);
\draw[ultra thick](1,1) rectangle (3,2);
\end{tikzpicture}\ \ 
\begin{tikzpicture}[baseline = (current bounding box).center]
\draw[ultra thick](0,-1) rectangle (1,1);
\draw[ultra thick](1,-1) rectangle (2,1);
\draw[ultra thick](2,1) rectangle (4,0);
\draw[ultra thick](4,-1) rectangle (2,0);
\draw[ultra thick](1,-2) rectangle (3,-1);
\draw[ultra thick](1,1) rectangle (3,2);
\end{tikzpicture}\ \ 
\begin{tikzpicture}[baseline = (current bounding box).center]
\draw[ultra thick](0,-1) rectangle (1,1);
\draw[ultra thick](1,-1) rectangle (2,1);
\draw[ultra thick](3,1) rectangle (2,-1);
\draw[ultra thick](4,1) rectangle (3,-1);
\draw[ultra thick](1,-2) rectangle (3,-1);
\draw[ultra thick](1,1) rectangle (3,2);
\end{tikzpicture}\ \ 
\begin{tikzpicture}[baseline = (current bounding box).center]
\draw[ultra thick](0,-1) rectangle (1,1);
\draw[ultra thick](1,-1) rectangle (3,0);
\draw[ultra thick](3,1) rectangle (1,0);
\draw[ultra thick](4,1) rectangle (3,-1);
\draw[ultra thick](1,-2) rectangle (3,-1);
\draw[ultra thick](1,1) rectangle (3,2);
\end{tikzpicture}\ \ 
\begin{tikzpicture}[baseline = (current bounding box).center]
\draw[ultra thick](0,-1) rectangle (1,1);
\draw[ultra thick](1,-1) rectangle (3,0);
\draw[ultra thick](2,0) rectangle (3,2);
\draw[ultra thick](4,1) rectangle (3,-1);
\draw[ultra thick](1,-2) rectangle (3,-1);
\draw[ultra thick](1,0) rectangle (2,2);
\end{tikzpicture}\ \ 
\begin{tikzpicture}[baseline = (current bounding box).center]
\draw[ultra thick](0,-1) rectangle (1,1);
\draw[ultra thick](3,1) rectangle (1,0);
\draw[ultra thick](4,1) rectangle (3,-1);
\draw[ultra thick](2,-2) rectangle (3,0);
\draw[ultra thick](1,-2) rectangle (2,0);
\draw[ultra thick](1,1) rectangle (3,2);
\end{tikzpicture}\ \ 
\begin{tikzpicture}[baseline = (current bounding box).center]
\draw[ultra thick](0,-1) rectangle (1,1);
\draw[ultra thick](1,-2) rectangle (2,0);
\draw[ultra thick](2,0) rectangle (3,2);
\draw[ultra thick](4,1) rectangle (3,-1);
\draw[ultra thick](2,-2) rectangle (3,0);
\draw[ultra thick](1,0) rectangle (2,2);
\end{tikzpicture}
}
\caption{The 8 domino tilings of the Aztec diamond of rank 2}
\label{rank2}
\end{figure}

\noindent There are numerous proofs of this theorem using many interesting combinatorial techniques, including random generation algorithms \cite{aztec0, aztec1}, non-intersecting paths \cite{aztecschroder}, and sequences of interlacing partitions \cite{steep,RYG}.  Thanks to these techniques, we know a lot about the asymptotic behavior of domino tilings of the Aztec diamond of rank $m$ when $m\rightarrow \infty$.  An important result is the arctic circle theorem of Jockusch, Propp, and Shor \cite{aztec3}, which roughly states that a uniformly random tiling behaves in a brickwork pattern in four regions (called frozen regions or polar regions), one adjacent to each corner of the Aztec diamond, whose union is the complement of the largest circle (called the arctic circle) that can be inscribed in the Aztec diamond.  A beautiful picture for a typical tiling for $m=1000$, as well as a description of the brickwork pattern in the frozen regions, can be found 
in Figure \ref{fig:large}, which was taken from \href{https://sites.uclouvain.be/aztecdiamond/examples/}{this site}.

\begin{figure}[]
    \centering
    \begin{tabular}{c}
    \vspace{-2cm}
    \includegraphics[width=.5\textwidth]{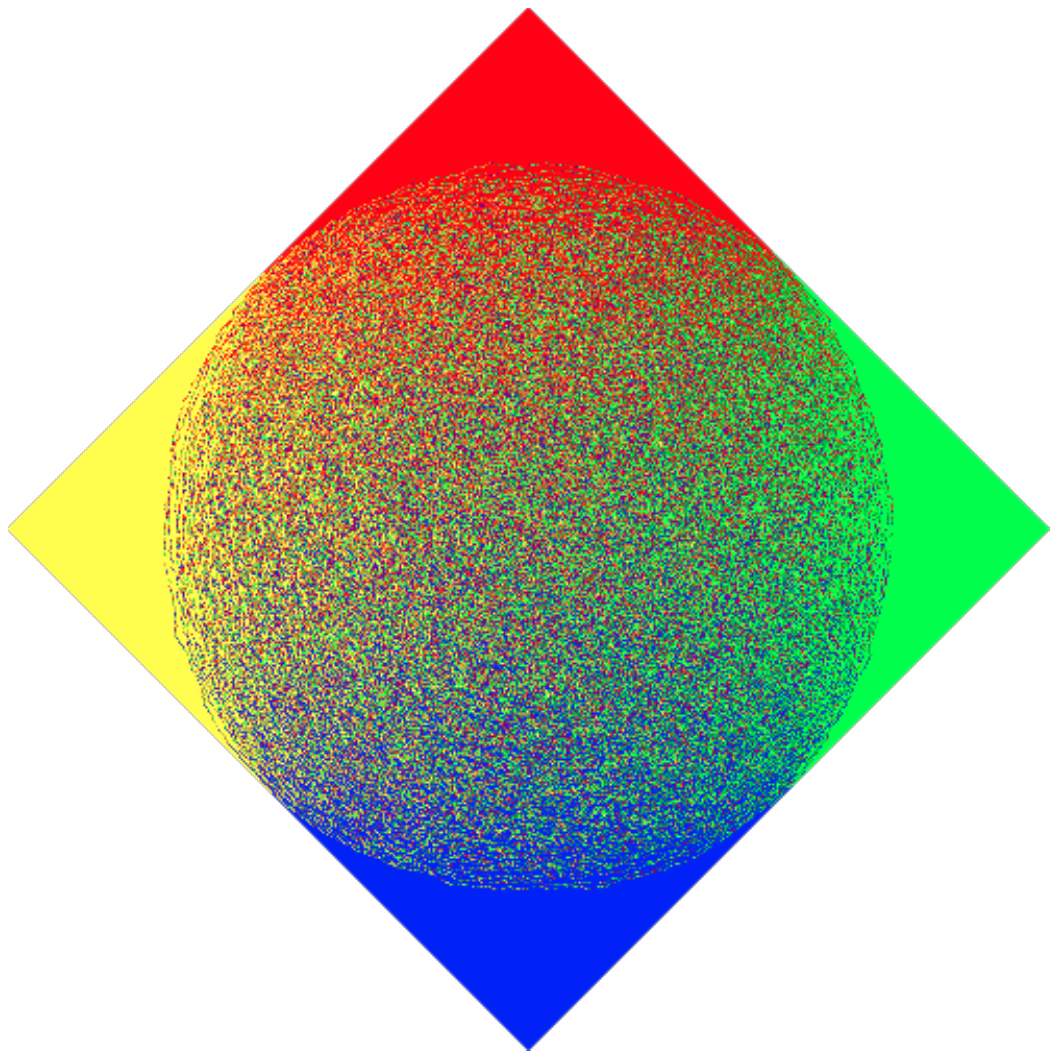} 
    \\
    \begin{tabular}{cccc}
    \begin{tikzpicture}[baseline = (current bounding box).center, scale = 0.3]
        \hdom[red]{0}{0}; \hdom[red]{2}{0};
        \hdom[red]{-1}{1}; \hdom[red]{1}{1}; \hdom[red]{3}{1};
        \hdom[red]{0}{2}; \hdom[red]{2}{2};
        \hdom[red]{-1}{3}; \hdom[red]{1}{3}; \hdom[red]{3}{3};
        \hdom[red]{0}{4}; \hdom[red]{2}{4};
    \end{tikzpicture}
    &
    \begin{tikzpicture}[baseline = (current bounding box).center, scale = 0.3]
        \hdom[blue]{0}{0}; \hdom[blue]{2}{0};
        \hdom[blue]{-1}{1}; \hdom[blue]{1}{1}; \hdom[blue]{3}{1};
        \hdom[blue]{0}{2}; \hdom[blue]{2}{2};
        \hdom[blue]{-1}{3}; \hdom[blue]{1}{3}; \hdom[blue]{3}{3};
        \hdom[blue]{0}{4}; \hdom[blue]{2}{4};
    \end{tikzpicture}
    &
    \begin{tikzpicture}[baseline = (current bounding box).center, scale = 0.3]
        \vdom[green]{0}{0}; \vdom[green]{0}{2};
        \vdom[green]{1}{-1}; \vdom[green]{1}{1}; \vdom[green]{1}{3};
        \vdom[green]{2}{0}; \vdom[green]{2}{2};
        \vdom[green]{3}{-1}; \vdom[green]{3}{1}; \vdom[green]{3}{3};
        \vdom[green]{4}{0}; \vdom[green]{4}{2};
    \end{tikzpicture}
    &
    \begin{tikzpicture}[baseline = (current bounding box).center, scale = 0.3]
        \vdom[yellow]{0}{0}; \vdom[yellow]{0}{2};
        \vdom[yellow]{1}{-1}; \vdom[yellow]{1}{1}; \vdom[yellow]{1}{3};
        \vdom[yellow]{2}{0}; \vdom[yellow]{2}{2};
        \vdom[yellow]{3}{-1}; \vdom[yellow]{3}{1}; \vdom[yellow]{3}{3};
        \vdom[yellow]{4}{0}; \vdom[yellow]{4}{2};
    \end{tikzpicture}
    \end{tabular}
    \end{tabular}
    \caption[]{On top is a typical tiling of the Aztec diamond of rank 1000. Here the four different types of dominos: \resizebox{1cm}{!}{
    \begin{tikzpicture}[baseline = (current bounding box).center]
    \draw[lightgray, fill=lightgray] (0,0) rectangle (1,1);
    \draw[very thick] (0,0) rectangle (2,1);
    \end{tikzpicture}
    }, \resizebox{1cm}{!}{
    \begin{tikzpicture}[baseline = (current bounding box).center]
    \draw[lightgray, fill=lightgray] (0,0) rectangle (1,1);
    \draw[very thick] (-1,0) rectangle (1,1);
    \end{tikzpicture}
    },  \resizebox{0.5cm}{!}{
    \begin{tikzpicture}[baseline = (current bounding box).center]
    \draw[lightgray, fill=lightgray] (0,0) rectangle (1,1);
    \draw[very thick] (0,0) rectangle (1,2);
    \end{tikzpicture}
    }, and 
    \resizebox{0.5cm}{!}{
    \begin{tikzpicture}[baseline = (current bounding box).center]
    \draw[lightgray, fill=lightgray] (0,0) rectangle (1,1);
    \draw[very thick] (0,-1) rectangle (1,1);
    \end{tikzpicture}
    }, are colored red, blue, green, and yellow, respectively. Note that the corners of the Aztec diamond contain large regions in which the tiling exhibits the same brickwork pattern, shown in the bottom image for the North, South, East, and West corner, respectively. These large regions are known as the frozen regions. The center of the tiling does not display a regular pattern and is known as the disordered region. }
    \label{fig:large}
\end{figure}

\begin{thm}\cite{aztec3,aztec2}
Fix $\epsilon > 0$. For each $m$,
consider a uniformly random domino tiling of the Aztec diamond of rank $m$ scaled by a factor $1/m$
in each axis to fit into the limiting diamond
$$
AD_\infty=\{|x| + |y| \le 1\}.
$$
and let $P_m$ be the image of the polar regions of the random tiling
under this scaling transformation. 
Then as $m\rightarrow \infty$
$$
\left \{(x, y)\in AD_\infty: x^2 + y^2 > \frac{1}{2}+\epsilon \right \} \cap \left(\frac{1}{m}AD_m\right)\subset
 P_m \subset \left \{(x, y)\in AD_\infty: x^2 + y^2 > \frac{1}{2}-\epsilon \right \}
$$
holds with probability that tends to 1.
\label{arcticthm}
\end{thm}
\noindent In the rest of this paper, we say that the arctic curve  of the Aztec diamond is the circle
\[
x^2+y^2=\frac{1}{2}.
\]

Cohn, Elkies, and Propp \cite{aztec2} later proved some results about the behavior of the tiling inside the arctic circle, specifically regarding the probability of observing a given domino in a given position and regarding the ``height function" of the tiling. Many more papers were written on domino tiling of the Aztec diamond, for example \cite{romik, johansson, johansson1}.

The main goal of this paper is to study superpositions of domino tilings of the Aztec diamond, by using recent ideas coming from the study of colored vertex models.  It has recently been shown \cite{color1,LLT,REU} that the LLT polynomials can be realized as a certain class of partition functions constructed from an integrable vertex model, and more recently the same was shown for the supersymmetric LLT polynomials (also called super ribbon functions) \cite{CFsuper,GKsuper}.  We will use the vertex models introduced in \cite{GKsuper}, which can also be realized as degenerations of a vertex model introduced in \cite{color1}.

For $k\ge 1$, we define a $k$-tiling of the Aztec diamond of rank $m$ to be the superposition of $k$ tilings, colored from 1 to $k$.  These tilings are not independent; we define an interaction between the tilings of colors $i < j$ to be a pair of dominos, one of color $i$ and one of color $j$, of a certain form. We call call such pairs ``coupled dominos". By relating these tilings to the colored vertex models introduced in \cite{GKsuper}, we prove that 
\begin{thm}\label{thm:partitionfunctionktiling}
The generating polynomial of the $k$-tilings of the Aztec diamond of rank $m$ is
\begin{equation}
\prod_{\ell=0}^{k-1}\prod_{1\le i\le j\le m}(1+t^{\ell}x_iy_j).    
\end{equation}
\end{thm}
\noindent Here $t$ follows the number of coupled dominos, and $x_i$ and $y_j$ follow the numbers of dominos of certain types as defined in Section \ref{two-models-section}. In fact, we introduce two different models, each with a different choice of coupled dominos, we call these the ``white-pink" and ``purple-gray" models. Both models have the same generating polynomial given in Thm. \ref{thm:partitionfunctionktiling}.

If $k=1$, we recover known results for the Aztec diamond \cite{steep}. If $t=0$, we construct a bijection between $k$-tilings with no coupled dominos and $1$-tilings.  This bijection allows us to compute the arctic curves for $t=0$ and $t\rightarrow\infty$. Note that setting $t=0$ disallows the presence of coupled dominos, greatly restricting the possible configurations in the 2-tilings. In fact, the brickwork pattern in the West corner frozen region of the usual Aztec diamond is entirely made up of coupled dominos in the purple-gray model, and cannot exist when $t=0$. We will see that it is replaced, in part, by a new ``herringbone" pattern of dominos. See Figure \ref{fig:2tilingt0} for an example.

\begin{figure}
    \centering
    \begin{tabular}{c}
    \resizebox{0.9\textwidth}{!}{
    \begin{tabular}{cc}
    \includegraphics[width=0.5\textwidth]{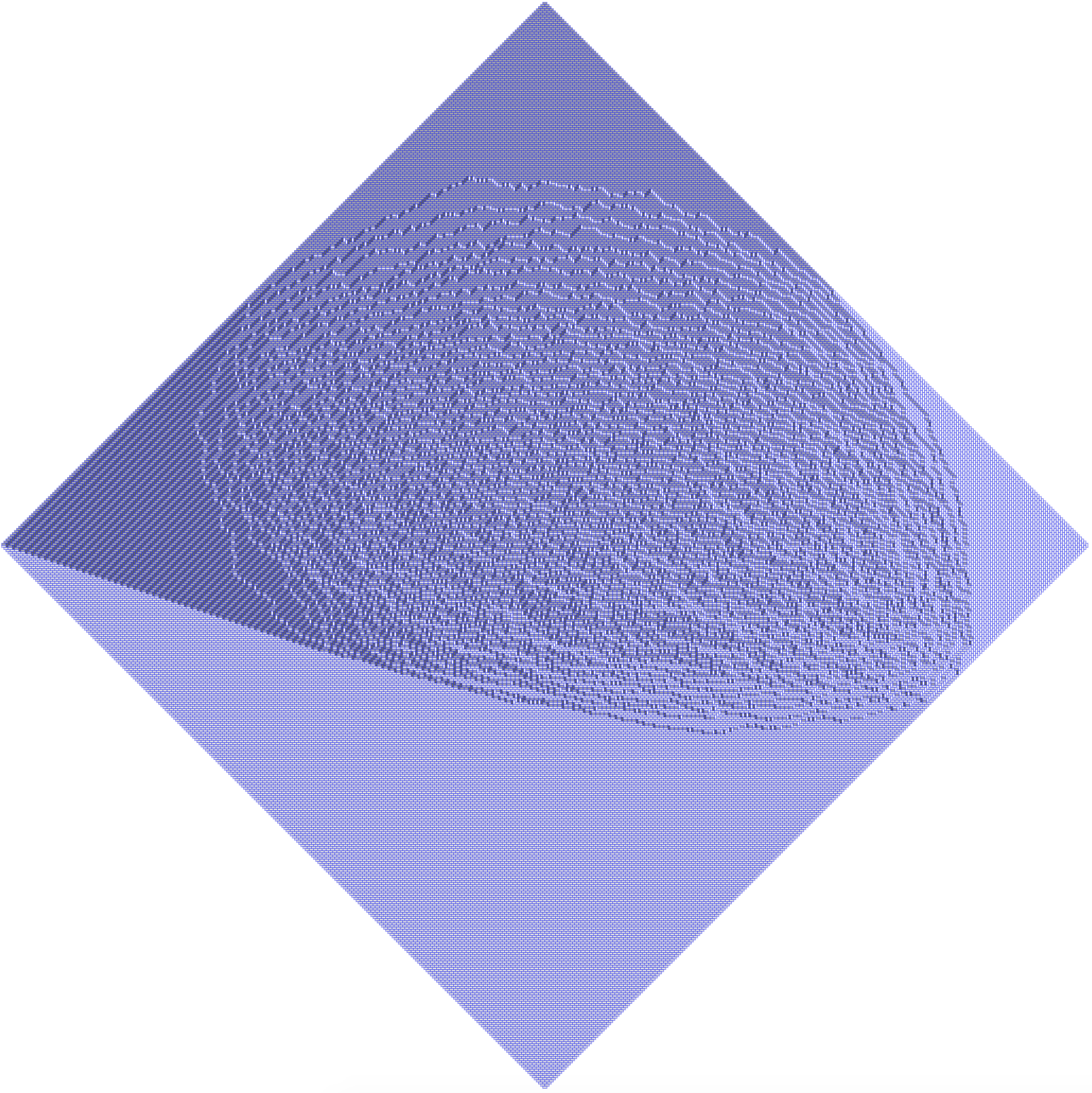} & \includegraphics[width=0.5\textwidth]{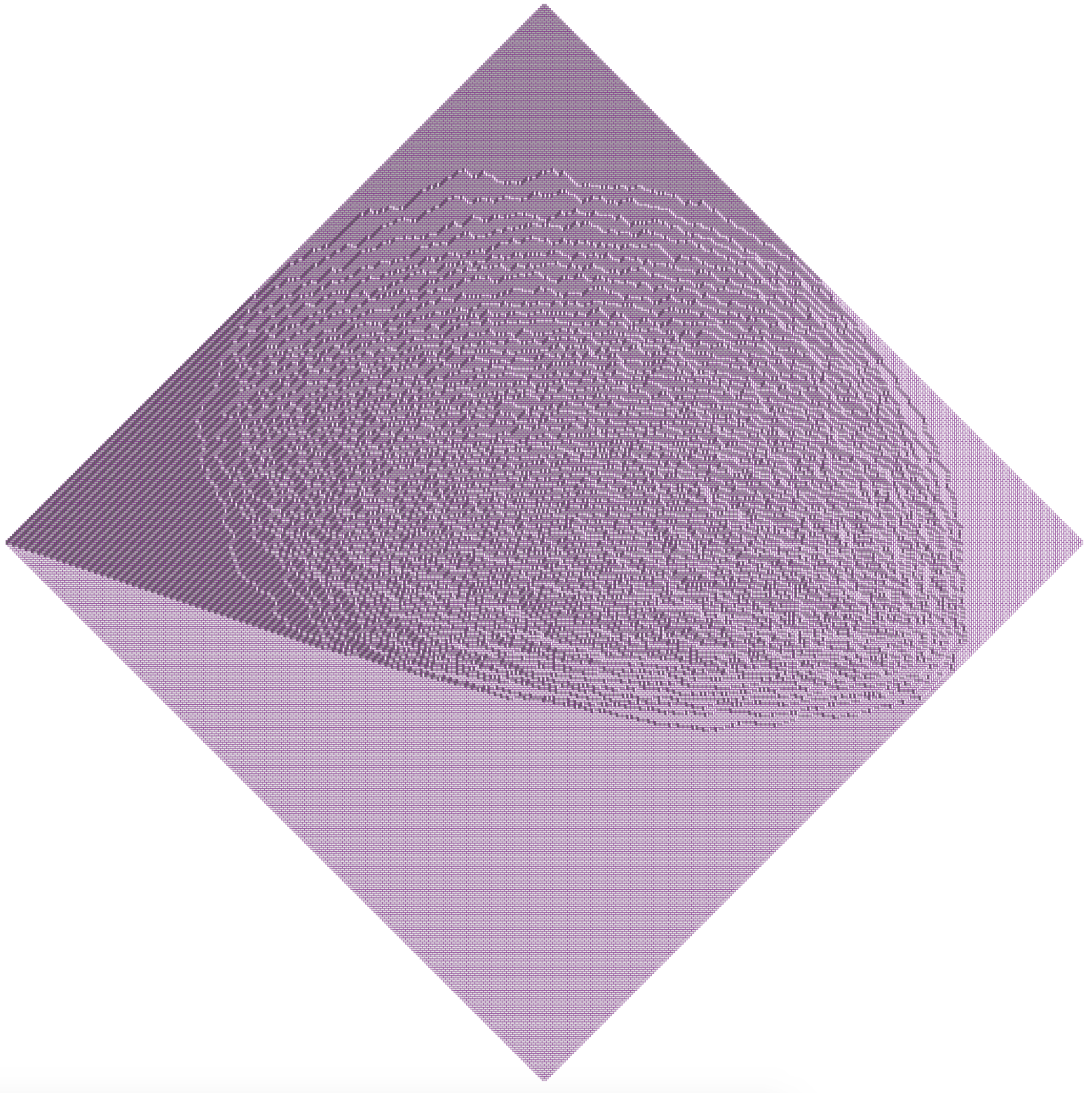} \\ \\
    \begin{tikzpicture}[baseline = (current bounding box).center, scale = 0.3]
        \draw[very thick, blue, fill=gray, fill opacity = 0.7] (0,0) rectangle (2,1); \draw[very thick, blue, fill=gray, fill opacity = 1.0] (2,-1) rectangle (3,1);
        \draw[very thick, blue, fill=gray, fill opacity = 0.7] (1,1) rectangle (3,2); \draw[very thick, blue, fill=gray, fill opacity = 1.0] (3,0) rectangle (4,2);
        \draw[very thick, blue, fill=gray, fill opacity = 0.7] (2,2) rectangle (4,3); \draw[very thick, blue, fill=gray, fill opacity = 1.0] (4,1) rectangle (5,3);
        \draw[very thick, blue, fill=gray, fill opacity = 0.7] (3,3) rectangle (5,4); \draw[very thick, blue, fill=gray, fill opacity = 1.0] (5,2) rectangle (6,4);

        \draw[very thick, blue, fill=gray, fill opacity = 0.7] (3,-1) rectangle (5,0); \draw[very thick, blue, fill=gray, fill opacity = 1.0] (5,-2) rectangle (6,0);
        \draw[very thick, blue, fill=gray, fill opacity = 0.7] (4,0) rectangle (6,1); \draw[very thick, blue, fill=gray, fill opacity = 1.0] (6,-1) rectangle (7,1);
        \draw[very thick, blue, fill=gray, fill opacity = 0.7] (5,1) rectangle (7,2); \draw[very thick, blue, fill=gray, fill opacity = 1.0] (7,0) rectangle (8,2); 
        \draw[very thick, blue, fill=gray, fill opacity = 0.7] (6,2) rectangle (8,3); \draw[very thick, blue, fill=gray, fill opacity = 1.0] (8,1) rectangle (9,3);
        \end{tikzpicture}
        &
         \begin{tikzpicture}[baseline = (current bounding box).center, scale = 0.3]
          \draw[very thick, red, fill=gray, fill opacity = 1.0] (0,-1) rectangle (1,1); \draw[very thick, red, fill=gray, fill opacity = 0.7] (1,-1) rectangle (3,0);
          \draw[very thick, red, fill=gray, fill opacity = 1.0] (1,0) rectangle (2,2); \draw[very thick, red, fill=gray, fill opacity = 0.7] (2,0) rectangle (4,1);
          \draw[very thick, red, fill=gray, fill opacity = 1.0] (2,1) rectangle (3,3); \draw[very thick, red, fill=gray, fill opacity = 0.7] (3,1) rectangle (5,2);
          \draw[very thick, red, fill=gray, fill opacity = 1.0] (3,2) rectangle (4,4); \draw[very thick, red, fill=gray, fill opacity = 0.7] (4,2) rectangle (6,3);

          \draw[very thick, red, fill=gray, fill opacity = 1.0] (3,-2) rectangle (4,0); \draw[very thick, red, fill=gray, fill opacity = 0.7] (4,-2) rectangle (6,-1);
          \draw[very thick, red, fill=gray, fill opacity = 1.0] (4,-1) rectangle (5,1); \draw[very thick, red, fill=gray, fill opacity = 0.7] (5,-1) rectangle (7,0);
          \draw[very thick, red, fill=gray, fill opacity = 1.0] (5,0) rectangle (6,2); \draw[very thick, red, fill=gray, fill opacity = 0.7] (6,0) rectangle (8,1);
          \draw[very thick, red, fill=gray, fill opacity = 1.0] (6,1) rectangle (7,3); \draw[very thick, red, fill=gray, fill opacity = 0.7] (7,1) rectangle (9,2);
         \end{tikzpicture}
    \end{tabular}
    }
    \end{tabular}
    \caption[]{Above is a simulation of a 2-tiling of the Aztec diamond of rank 256 for $t=0$. Here we draw one tiling in blue and the other in red to match the notation of Section \ref{kcolors}. We use the shading of the dominos to distinguish the four types of domino:
    $$
    \begin{tabular}{cccc}
    \resizebox{1cm}{!}{
    \begin{tikzpicture}[baseline = (current bounding box).center]
    \draw[lightgray, fill=lightgray] (0,0) rectangle (1,1);
    \draw[very thick] (-1,0) rectangle (1,1);
    \end{tikzpicture}
    }
    & 
    \resizebox{0.5cm}{!}{
    \begin{tikzpicture}[baseline = (current bounding box).center]
    \draw[lightgray, fill=lightgray] (0,0) rectangle (1,1);
    \draw[very thick] (0,-1) rectangle (1,1);
    \end{tikzpicture}
    }
    & 
    \resizebox{1cm}{!}{
    \begin{tikzpicture}[baseline = (current bounding box).center]
    \draw[lightgray, fill=lightgray] (0,0) rectangle (1,1);
    \draw[very thick] (0,0) rectangle (2,1);
    \end{tikzpicture}
    }
    & 
    \resizebox{0.5cm}{!}{
    \begin{tikzpicture}[baseline = (current bounding box).center]
    \draw[lightgray, fill=lightgray] (0,0) rectangle (1,1);
    \draw[very thick] (0,0) rectangle (1,2);
    \end{tikzpicture}
    } \\ \\
    \resizebox{1cm}{!}{
    \begin{tikzpicture}[baseline = (current bounding box).center]
    \draw[very thick, fill=gray, fill opacity = 0.4] (-1,0) rectangle (1,1);
    \end{tikzpicture}
    }
    & 
    \resizebox{0.5cm}{!}{
    \begin{tikzpicture}[baseline = (current bounding box).center]
    \draw[very thick, fill=gray, fill opacity = 1.0] (0,-1) rectangle (1,1);
    \end{tikzpicture}
    }
    & 
    \resizebox{1cm}{!}{
    \begin{tikzpicture}[baseline = (current bounding box).center]
    \draw[very thick, fill=gray, fill opacity = 0.7] (0,0) rectangle (2,1);
    \end{tikzpicture}
    }
    & 
    \resizebox{0.5cm}{!}{
    \begin{tikzpicture}[baseline = (current bounding box).center]
    \draw[very thick, fill=gray, fill opacity = 0.1] (0,0) rectangle (1,2);
    \end{tikzpicture}
    }
    \end{tabular}
    $$
    Note that in the western corner of both tilings we no longer see a brickwork pattern as shown in Figure \ref{fig:large}. Instead, the brickwork pattern from the southern corner extends further north and west, and new herringbone patterns appear. The herringbone pattern for each color is shown below the corresponding tiling. Note that the pattern for red is shifted by one diagonal from the pattern for blue.
    }
    \label{fig:2tilingt0}
\end{figure}

For $k$-tilings and $t=0$, we have
\begin{thm}
For the purple-gray model, the arctic curve for the $k$-tilings of the Aztec diamond when $t=0$ is given by
\[
\begin{cases}
x^2+y^2=\frac{1}{2},& x\in[-1/2,1/2],\; y>1/2 \\
(x+(k-1)y)^2+(k y)^2=\frac{1}{2},& x\in[-1/(2k),1 - 1/(2k)],\; y<-1/(2k) \\
\left(\frac{2x+(k-1)(x+y-1)}{2}\right)^2+\left(\frac{2y+(k-1)(x+y-1)}{2}\right)^2=\frac{1}{2},& y\in[-1/(2k),1/2],\; x>-\frac{k-1}{k+1}y+\frac{k}{k+1} \\
\left(\frac{2x+(k-1)(x+y-1)}{2k}\right)^2+\left(\frac{2y+(k-1)(3y-x-1)}{2k}\right)^2=\frac{1}{2},& y\in[-1/(2k),1/2],\; x<-\frac{k-1}{k+1}y-\frac{1}{k+1}
\end{cases}
\]
for each color.
\label{karctic}
\end{thm}
Figure \ref{fig:ArcticCurveDomino} shows a simulation of a 2-tiling when $t=0$, with the arctic curve superimposed.

The paper is organized as follows.  
\begin{itemize}
\item We introduce domino tilings of the Aztec diamond in Section \ref{aztec-section} and some combinatorial notions related to partitions in Section \ref{partitions-section}.  In Sections \ref{two-models-section} and \ref{interactions-subsection}, we define two different models (purple-gray and white-pink) for relating a $k$-tiling to a sequence of $k$-tuples of partitions and for assigning a weight to a $k$-tiling.  
\item In Section \ref{k-color-vertex-models-section}, we discuss the $k$-color vertex models that were introduced in \cite{LLT,GKsuper} to study (super) LLT polynomials.  In Section \ref{1-color-vertex-models-section}, we relate the rows of vertices to (co-)interlacing $k$-tuples of partitions.  
\item In Section \ref{kcolors}, we relate the vertex models to the purple-gray and white-pink models, and we use the vertex models to compute the generating polynomials of the $k$-tilings.  
\item We relate domino tilings to collections of non-intersecting paths in Section \ref{schroder-section}, which allows us to compute the arctic curve of the tilings for $t=0$ in Section \ref{t-0-section}.  
\item In Section \ref{t-infinity-section}, we relate the $t=0$ and $t\rightarrow\infty$ cases, which allows us to compute the arctic curve of the tilings for $t \rightarrow \infty$.  
\item We propose some open problems in Section \ref{open-section}.  We present some lozenge $k$-tilings of the hexagon and compute their arctic curves for $t=0$ in Appendix \ref{ap:a}.
\end{itemize}

\noindent{\bf Acknowledgements.} This material is based upon work supported by the National Science Foundation under Grant No. 1440140, while SC was in residence at the Mathematical Sciences Research Institute in Berkeley, California, during the Fall 2021. All authors are partially funded by NSF grant DMS-2054482. SC is partially funded by ANR grant ANR COMBIN\'e ANR-19-CE48-0011.

\section{Domino tilings of the Aztec diamond}
\label{aztec}

\subsection{The Aztec diamond} \label{aztec-section}

A lattice square is a $1 \times 1$ square $[a,a+1] \times [b,b+1]$ in $\mathbb{R}^2$, for some $a,b \in \mathbb{Z}$.  A horizontal domino 
$ \resizebox{!}{0.125cm}{{\begin{tikzpicture}[baseline = (current bounding box).center]
\draw (0,0) grid (2,1);
\end{tikzpicture}}} $
is a $2 \times 1$ rectangle consisting of two lattice squares.  A vertical domino
$ \resizebox{!}{0.25cm}{\begin{tikzpicture}[baseline = (current bounding box).center]
\draw (0,0) grid (1,2);
\end{tikzpicture} }$
is a $1 \times 2$ rectangle consisting of two lattice squares.  Given a region $R \subseteq \mathbb{R}^2$ consisting of lattice squares, a domino tiling of $R$ is a partitioning of $R$ into horizontal and vertical dominos.  A $k$-tiling $\bm{T} = (T_1,\ldots,T_k)$ is a $k$-tuple of domino tilings; we say the dominos in the $i$-th tiling $T_i$ are colored $i$.

The \textbf{Aztec diamond} of rank $m$ is the region in $\mathbb{R}^2$ which consists of all lattice squares lying completely inside the diamond-shaped region 
\[ \{(x, y) : |x| + |y| \leq m + 1\}. \]
We will be interested in domino tilings of the Aztec diamond.  As we can draw $\mathbb{R}^2$ as a checkerboard, where the lattice square $[a,a+1] \times [b,b+1]$ is shaded white if $a+b+m$ is odd and gray if $a+b+m$ is even, we have four types of dominos with which to tile the Aztec diamond of rank $m$:
\begin{center}
\begin{tabular}{cccc}
\resizebox{1.5cm}{!}{
\begin{tikzpicture}[baseline = (current bounding box).center]
\draw[lightgray, fill=lightgray] (0,0) rectangle (1,1);
\draw[very thick] (-1,0) rectangle (1,1);
\end{tikzpicture}
}
& 
\resizebox{0.75cm}{!}{
\begin{tikzpicture}[baseline = (current bounding box).center]
\draw[lightgray, fill=lightgray] (0,0) rectangle (1,1);
\draw[very thick] (0,-1) rectangle (1,1);
\end{tikzpicture}
}
& 
\resizebox{1.5cm}{!}{
\begin{tikzpicture}[baseline = (current bounding box).center]
\draw[lightgray, fill=lightgray] (0,0) rectangle (1,1);
\draw[very thick] (0,0) rectangle (2,1);
\end{tikzpicture}
}
& 
\resizebox{0.75cm}{!}{
\begin{tikzpicture}[baseline = (current bounding box).center]
\draw[lightgray, fill=lightgray] (0,0) rectangle (1,1);
\draw[very thick] (0,0) rectangle (1,2);
\end{tikzpicture}
}
\\
type I & type II & type III & type IV
\end{tabular}
\end{center}
For example, here is one possible domino tiling of the Aztec diamond of rank $m=3$.
\begin{center}
\resizebox{0.2\textwidth}{!}{
\begin{tikzpicture}[baseline = (current bounding box).center]
\checkerboard{2}
\draw[ultra thick] (0,-1) rectangle (1,1); \draw[ultra thick] (1,-2) rectangle (2,0); \draw[ultra thick] (1,0) rectangle (2,2);
\draw[ultra thick] (2,-3) rectangle (4,-2); \draw[ultra thick] (2,-1) rectangle (4,0);
\draw[ultra thick] (2,2) rectangle (4,3);
\draw[ultra thick] (2,0) rectangle (3,2); \draw[ultra thick] (2,-2) rectangle (4,-1); \draw[ultra thick] (3,0) rectangle (4,2); 
\draw[ultra thick] (4,-2) rectangle (5,0); \draw[ultra thick] (4,0) rectangle (5,2); 
\draw[ultra thick] (5,-1) rectangle (6,1);  
\end{tikzpicture}
}
\end{center}

\subsection{Partitions} \label{partitions-section}

A \textbf{partition} $\lambda$ is a non-increasing sequence of non-negative integers. We will usually write
$\lambda=(\lambda_1,\lambda_2,\ldots )$ such that $\lambda_1\ge \lambda_2\ge \ldots \ge 0$.
As it turns out, specifying a domino tiling of the Aztec diamond of rank $m$ is equivalent to specifying $2m+1$ partitions that satisfy certain conditions.  We will actually show this equivalence with two different constructions, but first, we must make some further definitions.  

The \textbf{Young diagram} of a partition $\lambda$ is the set of cells $(i,j)$ such that $0\le i\le \lambda_j$. Here we draw the diagram in French notation.
The \textbf{content} of a cell $u = (i,j)$ in the $i$-th row and $j$-th column of the Young diagram of a partition $\lambda$ is $c(u) = i-j$.  Notice that the \textbf{$c$ content line} $y = x-c$ goes through the center of each cell with content $c$.  The \textbf{Maya diagram} of a partition $\lambda = (\lambda_1,\lambda_2,\ldots)$ is a doubly infinite sequence  
\[ \ldots,a_{-\frac52},a_{-\frac32},a_{-\frac12},a_{\frac12},a_{\frac32},a_{\frac52},\ldots \]
of symbols in the alphabet $\{ \circ = 0, \bullet = 1 \}$, starting with infinitely many $\bullet$'s and ending with infinitely many $\circ$'s, where 
\[ a_{i+\frac{1}{2}} = \left \{ \begin{array}{ll} \bullet & \text{if $\lambda_j - j = i$ for some $j$} \\ \circ & \text{otherwise} \end{array} \right.. \]
The Maya diagram of $\lambda$ can be found by following the NE border of the Young diagram of $\lambda$ from NW to SE, where each $E$ step corresponds to a $\circ$ and each $S$ step corresponds to a $\bullet$.  The steps corresponding to $a_{c-\frac12}$ and $a_{c+\frac12}$ lie on the left and right side of the $c$ content line, respectively.  It is easy to see that the 0 content line is the unique content line such that the number of $\circ$'s to its left equals the number of $\bullet$'s to its right.

\begin{example}
The Maya diagram of (4,3,2,2,1) is $\ldots \bullet \bullet \enspace \circ\bullet\circ\bullet\bullet\circ\bullet\circ\bullet \enspace \circ \circ \ldots$.
\[ \resizebox{3cm}{!}{
\begin{tikzpicture}[baseline=(current bounding box.center)]
\draw (0,0) grid (2,4); \draw (0,4) grid (1,5); \draw (2,0) grid (3,2); \draw (3,0) grid (4,1);
\node[scale=2] at (0.5,5) {$\circ$}; \node[scale=2] at (1,4.5) {$\bullet$}; \node[scale=2] at (1.5,4) {$\circ$}; \node[scale=2] at (2,3.5) {$\bullet$}; \node[scale=2] at (2,2.5) {$\bullet$}; \node[scale=2] at (2.5,2) {$\circ$}; \node[scale=2] at (3,1.5) {$\bullet$}; \node[scale=2] at (3.5,1) {$\circ$}; \node[scale=2] at (4,0.5) {$\bullet$}; 
\draw (4,0) grid (7,0); \node[scale=2] at (4.5,0) {$\circ$}; \node[scale=2] at (5.5,0) {$\circ$}; 
\draw (0,5) grid (0,8); \node[scale=2] at (0,5.5) {$\bullet$}; \node[scale=2] at (0,6.5) {$\bullet$}; 
\draw[ultra thick, red] (0,0)--(4,4);
\end{tikzpicture} } \]
In red, we have indicated the 0 content line; there are two $\circ$'s to its left and two $\bullet$'s to its right.
\end{example}

We say two partitions $\lambda$ and $\mu$ \textbf{interlace}, and we write $\lambda \succeq \mu$, if 
\[ \lambda_1 \geq \mu_1 \geq \lambda_2 \geq \mu_2 \geq \ldots. \]
The conjugate of a partition $\lambda$ denoted by $\lambda'$ is the partition whose diagram is
the set of lattice squares $(i,j)$ such that $0\le j\le \lambda_i$.
We say two partitions $\lambda$ and $\mu$ \textbf{co-interlace}, and we write $\lambda \succeq' \mu$, if and only if $\lambda' \succeq \mu'$. 

A $k$-tuple of partitions $\bm{\lambda}$ is simply
a sequence of $k$ partitions 
$(\lambda^{(1)},\lambda^{(2)}, \ldots , \lambda^{(k)})$.  
We define the conjugate of $\bm{\lambda} = (\lambda^{(1)},\lambda^{(2)}, \ldots , \lambda^{(k)})$ to be $\bm{\lambda}' =({\lambda^{(k)}}',{\lambda^{(k-1)}}', \ldots , {\lambda^{(1)}}')$. We have to reverse the order of the colors 
as we are keeping with the notation in \cite{GKsuper} where it is used to prove certain identities between LLT polynomials (see \cite[Thm. 4.3]{GKsuper}).
We extend the definition of interlacing to $k$-tuples of partitions and we say that
$\bm{\lambda}\succeq \bm{\mu}$ if and only if $\lambda^{(i)}\succeq \mu^{(i)}$ for all $i\in[k]$, and similarly for co-interlacing.

We define the empty partition $\emptyset$ to be the partition with 0 parts, and we define $\bm{0}$ to be the $k$-tuple of empty partitions.  The length $\ell(\lambda)$ of a partition $\lambda$ is the number of non-zero parts.

\subsection{Sequences of partitions and tilings of the Aztec diamond} \label{two-models-section}

In this subsection we present a bijection betweem domino tilings and sequences of partitions that was first defined in \cite{steep} and generalized in \cite{RYG}.

\indent We will actually use two different constructions to go from a sequence of interlacing partitions to a tiling of the Aztec diamond.  In addition, for each construction, we will define the weight of a tiling as a polynomial in the variables $x_1,\ldots,x_m,y_1,\ldots,y_m$.

\begin{figure}[ht]
\centering\resizebox{12cm}{!}{
\begin{tikzpicture}
\checkerboard{2}
\draw[dashed] (2,-4)--(7,1); \node[below, left] at (2,-4) {$\lambda^0$};
\draw[dashed] (1.5,-3.5)--(6.5,1.5); \node[below, left] at (1.5,-3.5) {$\lambda^1$};
\draw[dashed] (1,-3)--(6,2); \node[below, left] at (1,-3) {$\lambda^2$};
\draw[dashed] (0.5,-2.5)--(5.5,2.5); \node[below, left] at (0.5,-2.5) {$\ldots$};
\draw[dashed] (0,-2)--(5,3); \node[below, left] at (0,-2) {$\lambda^{2m-2}$};
\draw[dashed] (-0.5,-1.5)--(4.5,3.5); \node[below, left] at (-0.5,-1.5) {$\lambda^{2m-1}$};
\draw[dashed] (-1,-1)--(4,4); \node[below, left] at (-1,-1) {$\lambda^{2m}$};
\draw[red,thick] (-0.5,0)--(6.5,0); \node[left] at (-0.5,0) {$0$};
\end{tikzpicture}
\begin{tikzpicture}
\checkerboard{2}
\draw[dashed] (2,-4)--(7,1); \node[below, left] at (2,-4) {$\lambda^0$};
\draw[dashed] (1.5,-3.5)--(6.5,1.5); \node[below, left] at (1.5,-3.5) {$\lambda^1$};
\draw[dashed] (1,-3)--(6,2); \node[below, left] at (1,-3) {$\lambda^2$};
\draw[dashed] (0.5,-2.5)--(5.5,2.5); \node[below, left] at (0.5,-2.5) {$\ldots$};
\draw[dashed] (0,-2)--(5,3); \node[below, left] at (0,-2) {$\lambda^{2m-2}$};
\draw[dashed] (-0.5,-1.5)--(4.5,3.5); \node[below, left] at (-0.5,-1.5) {$\lambda^{2m-1}$};
\draw[dashed] (-1,-1)--(4,4); \node[below, left] at (-1,-1) {$\lambda^{2m}$};
\draw[red,thick] (3,-3.5)--(3,3.5); \node[below] at (3,-3.5) {$0$};
\end{tikzpicture}}
\caption{Two ways of slicing the Aztec diamond. Here the superscript of each $\lambda$ indicates the number of the diagonal slice that it lies on.}
\label{fig:slicing}
\end{figure}

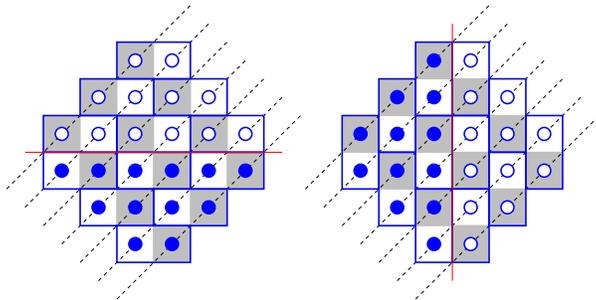
\begin{figure}[ht]
\centering\resizebox{8cm}{!}{
\begin{tikzpicture}
\checkerboard{2}
\draw[dashed] (2,-4)--(7,1);
\draw[dashed] (1.5,-3.5)--(6.5,1.5);
\draw[dashed] (1,-3)--(6,2); 
\draw[dashed] (0.5,-2.5)--(5.5,2.5);
\draw[dashed] (0,-2)--(5,3);
\draw[dashed] (-0.5,-1.5)--(4.5,3.5); 
\draw[dashed] (-1,-1)--(4,4);
\draw[very thick, blue,fill=white] (0.5,0.5) circle (5pt); 
\draw[very thick, blue,fill=blue] (0.5,-0.5) circle (5pt); 
\draw[very thick, blue,fill=white] (1.5,0.5) circle (5pt); \draw[very thick, blue,fill=white] (2.5,0.5) circle (5pt); \draw[very thick, blue,fill=white] (3.5,0.5) circle (5pt); \draw[very thick, blue,fill=white] (4.5,0.5) circle (5pt); \draw[very thick, blue,fill=white] (5.5,0.5) circle (5pt);
\draw[very thick, blue,fill=white] (1.5,1.5) circle (5pt); \draw[very thick, blue,fill=white] (2.5,1.5) circle (5pt); \draw[very thick, blue,fill=white] (3.5,1.5) circle (5pt); \draw[very thick, blue,fill=white] (4.5,1.5) circle (5pt);
\draw[very thick, blue,fill=white] (2.5,2.5) circle (5pt); \draw[very thick, blue,fill=white] (3.5,2.5) circle (5pt);
\draw[very thick, blue,fill=blue] (1.5,-0.5) circle (5pt); \draw[very thick, blue,fill=blue] (2.5,-0.5) circle (5pt); \draw[very thick, blue,fill=blue] (3.5,-0.5) circle (5pt); \draw[very thick, blue,fill=blue] (4.5,-0.5) circle (5pt); \draw[very thick, blue,fill=blue] (5.5,-0.5) circle (5pt);
\draw[very thick, blue,fill=blue] (1.5,-1.5) circle (5pt); \draw[very thick, blue,fill=blue] (2.5,-1.5) circle (5pt); \draw[very thick, blue,fill=blue] (3.5,-1.5) circle (5pt); \draw[very thick, blue,fill=blue] (4.5,-1.5) circle (5pt);
\draw[very thick, blue,fill=blue] (2.5,-2.5) circle (5pt); \draw[very thick, blue,fill=blue] (3.5,-2.5) circle (5pt);
\draw[very thick, blue] (0,-1) rectangle (2,0); \draw[very thick, blue] (2,-1) rectangle (4,0); \draw[very thick, blue] (4,-1) rectangle (6,0);
\draw[very thick, blue] (1,-2) rectangle (3,-1); \draw[very thick, blue] (3,-2) rectangle (5,-1);
\draw[very thick, blue] (2,-3) rectangle (4,-2);
\draw[very thick, blue] (0,0) rectangle (2,1); \draw[very thick, blue] (2,0) rectangle (4,1); \draw[very thick, blue] (4,0) rectangle (6,1); 
\draw[very thick, blue] (1,1) rectangle (3,2); \draw[very thick, blue] (3,1) rectangle (5,2); 
\draw[very thick, blue] (2,2) rectangle (4,3);
\draw[red,thick] (-0.5,0)--(6.5,0); 
\end{tikzpicture}
\begin{tikzpicture}
\checkerboard{2}
\draw[dashed] (2,-4)--(7,1);
\draw[dashed] (1.5,-3.5)--(6.5,1.5);
\draw[dashed] (1,-3)--(6,2); 
\draw[dashed] (0.5,-2.5)--(5.5,2.5);
\draw[dashed] (0,-2)--(5,3);
\draw[dashed] (-0.5,-1.5)--(4.5,3.5); 
\draw[dashed] (-1,-1)--(4,4);
\draw[very thick, blue,fill=blue] (0.5,0.5) circle (5pt); 
\draw[very thick, blue,fill=blue] (0.5,-0.5) circle (5pt); 
\draw[very thick, blue,fill=blue] (1.5,0.5) circle (5pt); \draw[very thick, blue,fill=blue] (2.5,0.5) circle (5pt); \draw[very thick, blue,fill=white] (3.5,0.5) circle (5pt); \draw[very thick, blue,fill=white] (4.5,0.5) circle (5pt); \draw[very thick, blue,fill=white] (5.5,0.5) circle (5pt);
\draw[very thick, blue,fill=blue] (1.5,1.5) circle (5pt); \draw[very thick, blue,fill=blue] (2.5,1.5) circle (5pt); \draw[very thick, blue,fill=white] (3.5,1.5) circle (5pt); \draw[very thick, blue,fill=white] (4.5,1.5) circle (5pt);
\draw[very thick, blue,fill=blue] (2.5,2.5) circle (5pt); \draw[very thick, blue,fill=white] (3.5,2.5) circle (5pt);
\draw[very thick, blue,fill=blue] (1.5,-0.5) circle (5pt); \draw[very thick, blue,fill=blue] (2.5,-0.5) circle (5pt); \draw[very thick, blue,fill=white] (3.5,-0.5) circle (5pt); \draw[very thick, blue,fill=white] (4.5,-0.5) circle (5pt); \draw[very thick, blue,fill=white] (5.5,-0.5) circle (5pt);
\draw[very thick, blue,fill=blue] (1.5,-1.5) circle (5pt); \draw[very thick, blue,fill=blue] (2.5,-1.5) circle (5pt); \draw[very thick, blue,fill=white] (3.5,-1.5) circle (5pt); \draw[very thick, blue,fill=white] (4.5,-1.5) circle (5pt);
\draw[very thick, blue,fill=blue] (2.5,-2.5) circle (5pt); \draw[very thick, blue,fill=white] (3.5,-2.5) circle (5pt);
\draw[very thick, blue] (0,-1) rectangle (1,1); \draw[very thick, blue] (1,0) rectangle (2,2); \draw[very thick, blue] (2,1) rectangle (3,3);
\draw[very thick, blue] (3,1) rectangle (4,3); \draw[very thick, blue] (4,0) rectangle (5,2);
\draw[very thick, blue] (5,-1) rectangle (6,1);
\draw[very thick, blue] (1,-2) rectangle (2,0); \draw[very thick, blue] (2,-3) rectangle (3,-1); \draw[very thick, blue] (3,-3) rectangle (4,-1); 
\draw[very thick, blue] (4,-2) rectangle (5,0); \draw[very thick, blue] (2,-1) rectangle (3,1); 
\draw[very thick, blue] (3,-1) rectangle (4,1);
\draw[red,thick] (3,-3.5)--(3,3.5); 
\end{tikzpicture}
}
\caption{The empty tilings for the purple-gray model (left) and the white-pink model (right)}
\label{fig:empty}
\end{figure}

The first construction, which we will call the \textbf{purple-gray model}, is as follows.  Specifying a domino tiling of the Aztec diamond of rank $m$ is equivalent to specifying $2m+1$ Maya diagrams on the slices going from SW to NE (drawn as dashed lines), as in the left of Figure \ref{fig:slicing}, where the 0 content line for the diagrams is drawn in red.  We impose the condition that 
\[
\emptyset = \lambda^0 \preceq ' \lambda^1 \succeq \ldots \preceq' \lambda^{2m-1} \succeq \lambda^{2m}=\emptyset
\]
and $\lambda^{i}_1 + \ell(\lambda^{i}) \leq m$ for all $i$.  (The Maya diagrams of these partitions are truncated to fit inside the Aztec diamond, with the 0 content line positioned as specified in the left of Figure \ref{fig:slicing}.  To recover the untruncated Maya diagram from the truncated one, we can pre-pend infinitely many $\bullet$'s and post-pend infinitely many $\circ$'s.)  From these Maya diagrams, we can draw dominos according to the following rules.
\[
\centering\resizebox{1cm}{!}{
\begin{tikzpicture}
\draw[lightgray] (0,0) rectangle (1,1);
\draw[lightgray,fill=lightgray] (1,0) rectangle (2,1);
\draw[very thick, blue,fill=blue] (0.5,0.5) circle (5pt);
\draw[very thick, blue,fill=blue] (1.5,0.5) circle (5pt);
\draw[very thick, blue] (0,0) rectangle (2,1);
\end{tikzpicture}},\;\;\;
\centering
\resizebox{.5cm}{!}{
\begin{tikzpicture}
\draw[lightgray] (0,0) rectangle (1,1);
\draw[lightgray,fill=lightgray] (0,1) rectangle (1,2);
\draw[very thick, blue] (0,0) rectangle (1,2);
\draw[very thick, blue,fill=white] (0.5,0.5) circle (5pt);
\draw[very thick, blue,fill=white] (0.5,1.5) circle (5pt);
\end{tikzpicture}},\;\;\;
\centering\resizebox{1cm}{!}{ \begin{tikzpicture}
\draw[lightgray,fill=lightgray] (0,0) rectangle (1,1);
\draw[lightgray] (1,0) rectangle (2,1);
\draw[very thick, blue] (0,0) rectangle (2,1);
\draw[very thick, blue,fill=white] (0.5,0.5) circle (5pt);
\draw[very thick, blue,fill=white] (1.5,0.5) circle (5pt);
\end{tikzpicture}},\;\;\;
\centering\resizebox{.5cm}{!}{
\begin{tikzpicture}
\draw[lightgray, fill=lightgray] (0,0) rectangle (1,1);
\draw[lightgray] (0,1) rectangle (1,2);
\draw[very thick, blue,fill=blue] (0.5,0.5) circle (5pt);
\draw[very thick, blue,fill=blue] (0.5,1.5) circle (5pt);
\draw[very thick, blue] (0,0) rectangle (1,2);
\end{tikzpicture}}
\]

\noindent (There is a unique way to do this.)  For example, if each $\lambda^\ell = \emptyset$, then the corresponding Maya diagrams and domino tiling are given on the left of Figure \ref{fig:empty}.  We define the weight of a domino tiling of the Aztec diamond to be the product of the weights of the dominos, where we assign weights to the dominos according to the following rules.
\begin{itemize}
\item A domino of the form 
\resizebox{!}{0.3cm}{
\begin{tikzpicture}[baseline = (current bounding box).center]
\draw[lightgray, fill=lightgray] (0,0) rectangle (1,1);
\draw[lightgray] (0,1) rectangle (1,2);
\draw[very thick,blue,fill=blue] (0.5,0.5) circle (5pt);
\draw[very thick,blue,fill=blue] (0.5,1.5) circle (5pt);
\draw[very thick, blue] (0,0) rectangle (1,2);
\end{tikzpicture} }
whose top square is on slice $2i-1$ gets a weight of $x_i$.
\item A domino of the form
\resizebox{!}{0.3cm}{
\begin{tikzpicture}[baseline = (current bounding box).center]
\draw[lightgray] (0,0) rectangle (1,1);
\draw[lightgray,fill=lightgray] (0,1) rectangle (1,2);
\draw[very thick,blue,fill=white] (0.5,0.5) circle (5pt);
\draw[very thick,blue,fill=white] (0.5,1.5) circle (5pt);
\draw[very thick, blue] (0,0) rectangle (1,2);
\end{tikzpicture} }
whose bottom square is on slice $2i-1$ gets a weight of $y_i$.
\item All other dominos get a weight of 1.
\end{itemize}

The second construction, which we will call the \textbf{white-pink model}, is as follows. Specifying a domino tiling of the Aztec diamond of rank $m$ is equivalent to specifying $2m+1$ Maya diagrams on the slices going from SW to NE (drawn as dashed lines), as in the right of Figure \ref{fig:slicing}, where the 0 content line for the diagrams is drawn in red.  We impose the condition that 
\[
\emptyset = \lambda^0 \preceq  \lambda^1 \succeq' \ldots \preceq \lambda^{2m-1} \succeq' \lambda^{2m}=\emptyset
\]
and $\lambda^{i}_1 + \ell(\lambda^{i}) \leq m$ for all $i$. Note that the two models are the same up  to swapping interlacing and cointerlacing. (Similarly as in the purple-gray model, we truncate the Maya diagrams to fit inside the Aztec diamond.)  From these Maya diagrams, we can draw dominos according to the following rules.
\[
\centering\resizebox{!}{.5cm}{
\begin{tikzpicture}
\draw[lightgray] (0,0) rectangle (1,1);
\draw[lightgray,fill=lightgray] (1,0) rectangle (2,1);
\draw[very thick, blue, fill=white] (0.5,0.5) circle (5pt);
\draw[very thick, blue, fill=white] (1.5,0.5) circle (5pt);
\draw[very thick, blue] (0,0) rectangle (2,1);
\end{tikzpicture}},\;\;\;
\centering\resizebox{!}{1cm}{
\begin{tikzpicture}
\draw[lightgray] (0,0) rectangle (1,1);
\draw[lightgray,fill=lightgray] (0,1) rectangle (1,2);
\draw[very thick, blue] (0,0) rectangle (1,2);
\draw[very thick, blue,fill=blue] (0.5,0.5) circle (5pt);
\draw[very thick, blue,fill=blue] (0.5,1.5) circle (5pt);
\end{tikzpicture}},\;\;\;
\centering\resizebox{!}{.5cm}{
\begin{tikzpicture}
\draw[lightgray,fill=lightgray] (0,0) rectangle (1,1);
\draw[lightgray] (1,0) rectangle (2,1);
\draw[very thick, blue] (0,0) rectangle (2,1);
\draw[very thick, blue,fill=blue] (0.5,0.5) circle (5pt);
\draw[very thick, blue,fill=blue] (1.5,0.5) circle (5pt);
\end{tikzpicture}},\;\;\;
\centering\resizebox{!}{1cm}{
\begin{tikzpicture}
\draw[lightgray, fill=lightgray] (0,0) rectangle (1,1);
\draw[lightgray] (0,1) rectangle (1,2);
\draw[very thick, blue, fill=white] (0.5,0.5) circle (5pt);
\draw[very thick, blue, fill=white] (0.5,1.5) circle (5pt);
\draw[very thick, blue] (0,0) rectangle (1,2);
\end{tikzpicture}}
\]

\noindent (There is a unique way to do this.)  For example, if each $\lambda^\ell = \emptyset$, then the corresponding Maya diagrams and domino tiling are given on the right of Figure \ref{fig:empty}.  We define the weight of a domino tiling of the Aztec diamond to be the product of the weights of the dominos, where we assign weights to the dominos according to the following rules.
\begin{itemize}
\item A domino of the form 
\resizebox{!}{0.4cm}{
\begin{tikzpicture}
\draw[lightgray] (0,0) rectangle (1,1);
\draw[lightgray,fill=lightgray] (1,0) rectangle (2,1);
\draw[very thick, blue, fill=white] (0.5,0.5) circle (5pt);
\draw[very thick, blue, fill=white] (1.5,0.5) circle (5pt);
\draw[very thick, blue] (0,0) rectangle (2,1);
\end{tikzpicture} }
whose left square is on slice $2i-1$ gets a weight of $x_i$.
\item A domino of the form
\resizebox{!}{0.4cm}{
\begin{tikzpicture}
\draw[lightgray,fill=lightgray] (0,0) rectangle (1,1);
\draw[lightgray] (1,0) rectangle (2,1);
\draw[very thick, blue] (0,0) rectangle (2,1);
\draw[very thick, blue,fill=blue] (0.5,0.5) circle (5pt);
\draw[very thick, blue,fill=blue] (1.5,0.5) circle (5pt);

\end{tikzpicture} }
whose right square is on slice $2i-1$ gets a weight of $y_i$.
\item All other dominos get a weight of 1.
\end{itemize}

\begin{figure}[tb]
\centering\resizebox{3cm}{!}{
\begin{tikzpicture}[baseline = (current bounding box).center]
\checkerboard{2}
\draw[ultra thick] (0,-1) rectangle (1,1); \draw[ultra thick] (1,-2) rectangle (2,0); \draw[ultra thick] (1,0) rectangle (2,2);
\draw[ultra thick] (2,-3) rectangle (4,-2); \draw[ultra thick] (2,-1) rectangle (4,0);
\draw[ultra thick] (2,2) rectangle (4,3);
\draw[ultra thick] (2,0) rectangle (3,2); \draw[ultra thick] (2,-2) rectangle (4,-1); \draw[ultra thick] (3,0) rectangle (4,2); 
\draw[ultra thick] (4,-2) rectangle (5,0); \draw[ultra thick] (4,0) rectangle (5,2); 
\draw[ultra thick] (5,-1) rectangle (6,1);       
\draw[ultra thick, blue, fill=white] (0.5,0.5) circle (5pt);
\draw[ultra thick, blue, fill=white] (0.5,-.5) circle (5pt);
\draw[ultra thick, blue, fill=white] (1.5,-1.5) circle (5pt);
\draw[ultra thick, blue, fill=white] (1.5,-.5) circle (5pt);
\draw[ultra thick, blue, fill=white] (1.5,1.5) circle (5pt);
\draw[ultra thick, blue, fill=white] (1.5,.5) circle (5pt);
\draw[ultra thick, blue, fill=white] (3.5,1.5) circle (5pt);
\draw[ultra thick, blue, fill=white] (3.5,.5) circle (5pt);
\draw[ultra thick, blue, fill=white] (3.5,2.5) circle (5pt);\draw[ultra thick, blue, fill=white] (2.5,2.5) circle (5pt);
\draw[ultra thick, blue, fill=white] (3.5,-1.5) circle (5pt);
\draw[ultra thick, blue, fill=white] (2.5,-1.5) circle (5pt);
\draw[ultra thick, blue, fill=blue] (3.5,-2.5) circle (5pt);
\draw[ultra thick, blue, fill=blue] (2.5,-2.5) circle (5pt);
\draw[ultra thick, blue, fill=blue] (3.5,-.5) circle (5pt);
\draw[ultra thick, blue, fill=blue] (2.5,-.5) circle (5pt);
\draw[ultra thick, blue, fill=blue] (4.5,-.5) circle (5pt);
\draw[ultra thick, blue, fill=blue] (4.5,-1.5) circle (5pt);
\draw[ultra thick, blue, fill=blue] (4.5,.5) circle (5pt);
\draw[ultra thick, blue, fill=blue] (4.5,1.5) circle (5pt);
\draw[ultra thick, blue, fill=blue] (2.5,.5) circle (5pt);
\draw[ultra thick, blue, fill=blue] (2.5,1.5) circle (5pt);
\draw[ultra thick, blue, fill=blue] (5.5,-.5) circle (5pt);
\draw[ultra thick, blue, fill=blue] (5.5,.5) circle (5pt);
\end{tikzpicture} }
\resizebox{3cm}{!}{
\begin{tikzpicture}[baseline = (current bounding box).center]
\checkerboard{2}
\draw[ultra thick] (0,-1) rectangle (1,1); \draw[ultra thick] (1,-2) rectangle (2,0); \draw[ultra thick] (1,0) rectangle (2,2);
\draw[ultra thick] (2,-3) rectangle (4,-2); \draw[ultra thick] (2,-1) rectangle (4,0);
\draw[ultra thick] (2,2) rectangle (4,3);
\draw[ultra thick] (2,0) rectangle (3,2); \draw[ultra thick] (2,-2) rectangle (4,-1); \draw[ultra thick] (3,0) rectangle (4,2); 
\draw[ultra thick] (4,-2) rectangle (5,0); \draw[ultra thick] (4,0) rectangle (5,2); 
\draw[ultra thick] (5,-1) rectangle (6,1);       
\draw[ultra thick, blue, fill=blue] (0.5,0.5) circle (5pt);
\draw[ultra thick, blue, fill=blue] (0.5,-.5) circle (5pt);
\draw[ultra thick, blue, fill=blue] (1.5,-1.5) circle (5pt);
\draw[ultra thick, blue, fill=blue] (1.5,-.5) circle (5pt);
\draw[ultra thick, blue, fill=blue] (1.5,1.5) circle (5pt);
\draw[ultra thick, blue, fill=blue] (1.5,.5) circle (5pt);
\draw[ultra thick, blue, fill=blue] (3.5,1.5) circle (5pt);
\draw[ultra thick, blue, fill=blue] (3.5,.5) circle (5pt);
\draw[ultra thick, blue, fill=blue] (3.5,2.5) circle (5pt);\draw[ultra thick, blue, fill=blue] (2.5,2.5) circle (5pt);
\draw[ultra thick, blue, fill=blue] (3.5,-1.5) circle (5pt);
\draw[ultra thick, blue, fill=blue] (2.5,-1.5) circle (5pt);
\draw[ultra thick, blue, fill=white] (3.5,-2.5) circle (5pt);
\draw[ultra thick, blue, fill=white] (2.5,-2.5) circle (5pt);
\draw[ultra thick, blue, fill=white] (3.5,-.5) circle (5pt);
\draw[ultra thick, blue, fill=white] (2.5,-.5) circle (5pt);
\draw[ultra thick, blue, fill=white] (4.5,-.5) circle (5pt);
\draw[ultra thick, blue, fill=white] (4.5,-1.5) circle (5pt);
\draw[ultra thick, blue, fill=white] (4.5,.5) circle (5pt);
\draw[ultra thick, blue, fill=white] (4.5,1.5) circle (5pt);
\draw[ultra thick, blue, fill=white] (2.5,.5) circle (5pt);
\draw[ultra thick, blue, fill=white] (2.5,1.5) circle (5pt);
\draw[ultra thick, blue, fill=white] (5.5,-.5) circle (5pt);
\draw[ultra thick, blue, fill=white] (5.5,.5) circle (5pt);
\end{tikzpicture} }
\caption{A domino tiling and the corresponding Maya diagrams in the purple-gray (left) and white-pink (right) models}
\label{fig:DominoTiling}
\end{figure}
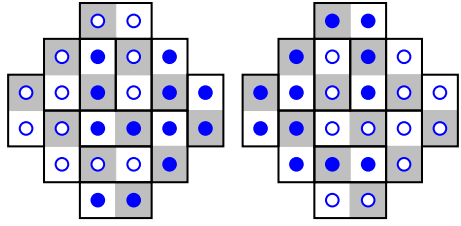

For example, consider the domino tiling in Figure \ref{fig:DominoTiling}.  In the purple-gray model, this tiling gives the sequence of partitions
\[ \emptyset \preceq ' (1,1) \succeq (1,1)\preceq '  (2,1)\succeq (1)\preceq '  (2)\succeq \emptyset \]
and has weight $x_1^2x_2x_3y_2y_3^2$.  In the white-pink model, this tiling gives the sequence of partitions
\[ \emptyset \preceq (1) \succeq' \emptyset \preceq  (1) \succeq' (1) \preceq  (1) \succeq' \emptyset \]
and has weight $x_1x_2y_1y_3$.

Both models were studied previously in \cite{steep,RYG}:
\begin{prop}\cite{steep}
The generating polynomial of both models is the same and is
\[ \prod_{1 \leq i \leq j \leq m}(1+x_iy_j). \]
\end{prop}
We will generalize these two models to $k$-tilings, and recover this result in the case $k=1$.

\subsection{Extending the models to $k$-tilings}
\label{interactions-subsection}

\indent The two models in the previous section can be extended to $k$-tilings.  For the purple-gray model, specifying a $k$-tiling of the Aztec diamond of rank $m$ is equivalent to specifying a sequence of $2m+1$ $k$-tuples of partitions satisfying
\[
\bm{0} = \bm{\lambda}^0 \preceq ' \bm{\lambda}^1 \succeq \ldots \preceq' \bm{\lambda}^{2m-1} \succeq \bm{\lambda}^{2m}=\bm{0}.
\]
For the white-pink model, specifying a $k$-tiling of the Aztec diamond of rank $m$ is equivalent to specifying a sequence of $2m+1$ $k$-tuples of partitions satisfying
\[
\bm{0} = \bm{\lambda}^0 \preceq  \bm{\lambda}^1 \succeq ' \ldots \preceq \bm{\lambda}^{2m-1} \succeq ' \bm{\lambda}^{2m}= \bm{0}.
\]
For both models, letting the $k$-tiling be $\bm{T} = (T_1,\ldots,T_k)$ and letting the $j$-th $k$-tuple of partitions be $\bm{\lambda}^j=(\lambda^{(j,1)},\ldots ,\lambda^{(j,k)})$ for all $j$, the $i$-th tiling $T_i$ corresponds to the sequence of partitions $(\lambda^{(0,i)},\ldots ,\lambda^{(2m,i)})$ for all $i$.

We define the weight of a $k$-tiling $\bm{T}$ as a polynomial in the variables $x_1,\ldots,x_m,y_1,\ldots,y_m,t$ by the equation
\[ 
w(\bm{T})=w(T_1,\ldots,T_k) = t^{\# {\rm coupled dominos}} \prod_{i=1}^k w(T_i). 
\]
In other words, the weight of a $k$-tiling is the product of the weights of the individual tilings, times an additional factor of $t$ for every \textbf{interaction} between two of the tilings.  In the purple-gray model, an interaction is a pair of \textbf{coupled dominos} of the form
\begin{equation*}
\resizebox{1cm}{!}{
\begin{tikzpicture}[baseline = (current bounding box).center]
\draw[lightgray] (0,1) rectangle (1,2);
\draw[lightgray,fill=lightgray] (0,0) rectangle (1,1);
\draw[lightgray] (-1,0) rectangle (0,1);
\draw[line width=1mm, blue] (0,0) rectangle (1,2);
\draw[line width=1mm, red] (-1.01,-0.01) rectangle (0.99,0.99);
\end{tikzpicture}}\ ,
\;\;\;
\resizebox{.5cm}{!}{
 \begin{tikzpicture}[baseline = (current bounding box).center]
\draw[lightgray,fill=lightgray] (0,1) rectangle (1,2);
\draw[lightgray] (0,0) rectangle (1,1);
\draw[line width=1mm, blue] (0,0) rectangle (1,2);
\draw[line width=1mm, red] (-0.1,-0.1) rectangle (0.9,1.9);
\end{tikzpicture}}\ ,
\;\;\;
\resizebox{1cm}{!}{
 \begin{tikzpicture}[baseline = (current bounding box).center]
\draw[lightgray,fill=lightgray] (0,1) rectangle (1,2);
\draw[lightgray] (0,0) rectangle (1,1);
\draw[line width=1mm, blue] (-1.01,0.99) rectangle (0.99,1.99);
\draw[line width=1mm, red] (0,0) rectangle (1,2);
\end{tikzpicture}}\ , {\rm or}
\;\;\;
\resizebox{.5cm}{!}{
 \begin{tikzpicture}[baseline = (current bounding box).center]
\draw[lightgray,fill=lightgray] (0,1) rectangle (1,2);
\draw[lightgray] (0,0) rectangle (1,1);
\draw[line width=1mm, blue] (-0.01,0.99) rectangle (.99,2.99);
\draw[line width=1mm, red] (0,0) rectangle (1,2);
\end{tikzpicture}}.
\end{equation*}
In the white-pink model, an interaction is a pair of coupled dominos of the form
\begin{equation*}
\resizebox{1cm}{!}{
\begin{tikzpicture}[baseline = (current bounding box).center]
\draw[lightgray,fill=lightgray] (1,-1) rectangle (2,0);
\draw[lightgray,fill=lightgray] (0,0) rectangle (1,1);
\draw[line width=1mm, blue] (0,-1) rectangle (2,0);
\draw[line width=1mm, red] (-0.01,-1.01) rectangle (0.99,0.99);
\end{tikzpicture}}\ ,
\;\;\;
\resizebox{1cm}{!}{ \begin{tikzpicture}[baseline = (current bounding box).center]
\draw[lightgray,fill=lightgray] (1,0) rectangle (2,1);
\draw[line width=1mm, blue] (0,0) rectangle (2,1);
\draw[line width=1mm, red] (-0.1,-0.1) rectangle (1.9,0.9);
\end{tikzpicture}}\ ,
\;\;\;
\resizebox{1cm}{!}{
 \begin{tikzpicture}[baseline = (current bounding box).center]
\draw[lightgray,fill=lightgray] (0,0) rectangle (1,1);
\draw[lightgray,fill=lightgray] (1,1) rectangle (2,2);
\draw[line width=1mm, blue] (0.99,-0.01) rectangle (1.99,1.99);
\draw[line width=1mm, red] (0,0) rectangle (2,1);
\end{tikzpicture}}\ , {\rm or}
\;\;\;
\resizebox{1.5cm}{!}{
 \begin{tikzpicture}[baseline = (current bounding box).center]
\draw[lightgray] (1,0) rectangle (2,1);
\draw[lightgray,fill=lightgray] (0,0) rectangle (1,1);
\draw[lightgray,fill=lightgray] (2,0) rectangle (3,1);
\draw[line width=1mm, blue] (0.99,-0.01) rectangle (2.99,0.99);
\draw[line width=1mm, red] (0,0) rectangle (2,1);
\end{tikzpicture}}.
\end{equation*}
Here blue is a smaller color than red.  

\begin{figure}[ht]
\begin{center}
\resizebox{\textwidth}{!}{
\begin{tabular}{ccc}
\resizebox{0.2\textwidth}{!}{
\begin{tikzpicture}[baseline = (current bounding box).center]
\checkerboard{2}
\draw[very thick, blue] (0,-1) rectangle (1,1); \draw[very thick, blue] (1,-2) rectangle (2,0); \draw[very thick, blue] (1,0) rectangle (2,2);
\draw[very thick, blue] (2,-3) rectangle (4,-2); \draw[very thick, blue] (2,-1) rectangle (4,0);
\draw[very thick, blue] (2,2) rectangle (4,3);
\draw[very thick, blue] (2,0) rectangle (3,2); \draw[very thick, blue] (2,-2) rectangle (4,-1); \draw[very thick, blue] (3,0) rectangle (4,2); 
\draw[very thick, blue] (4,-2) rectangle (5,0); \draw[very thick, blue] (4,0) rectangle (5,2); 
\draw[very thick, blue] (5,-1) rectangle (6,1);
\end{tikzpicture}
}
&
\resizebox{0.2\textwidth}{!}{
\begin{tikzpicture}[baseline = (current bounding box).center]
\checkerboard{2}
\draw[very thick, red] (0,-1) rectangle (1,1); \draw[very thick, red] (1,-2) rectangle (2,0); 
\draw[very thick, red] (1,0) rectangle (3,1);
\draw[very thick, red] (2,-3) rectangle (3,-1); 
\draw[very thick, red] (2,-1) rectangle (4,0);
\draw[very thick, red] (1,1) rectangle (3,2);
\draw[very thick, red] (3,-3) rectangle (4,-1);
\draw[very thick, red] (3,0) rectangle (5,1); 
\draw[very thick, red] (3,1) rectangle (5,2); 
\draw[very thick, red] (4,-2) rectangle (5,0); 
\draw[very thick, red] (2,2) rectangle (4,3); 
\draw[very thick, red] (5,-1) rectangle (6,1);
\end{tikzpicture}
}
&
\resizebox{0.2\textwidth}{!}{
\begin{tikzpicture}[baseline = (current bounding box).center]
\checkerboard{2}
\draw[very thick, green] (0,-1) rectangle (2,0); 
\draw[very thick, green] (0,0) rectangle (2,1);
\draw[very thick, green] (1,-2) rectangle (3,-1);
\draw[very thick, green] (2,-3) rectangle (4,-2);
\draw[very thick, green] (3,-2) rectangle (5,-1);
\draw[very thick, green] (1,1) rectangle (3,2);
\draw[very thick, green] (3,1) rectangle (5,2);
\draw[very thick, green] (2,2) rectangle (4,3);
\draw[very thick, green] (2,-1) rectangle (3,1);
\draw[very thick, green] (3,-1) rectangle (4,1);
\draw[very thick, green] (4,-1) rectangle (5,1);
\draw[very thick, green] (5,-1) rectangle (6,1);
\end{tikzpicture}
}
\\
$T_1$ & $T_2$ & $T_3$
\end{tabular}
}
\end{center}
\caption{An example of a 3-tiling of the rank 3 Aztec diamond with 11 coupled dominos in the purple-gray model} 
\label{ex:3til}
\end{figure}

In Figure \ref{ex:3til}, we have an example of a 3-tiling of the rank 3 Aztec diamond.  In the purple-gray model, the first tiling has weight $x_1^2 x_2y_2^2x_3y_3^2$, the second has weight $x_1^3 y_1  y_2  y_3$, the third has weight $x_1 y_1 x_2 y_2$, and there are 11 pairs of coupled dominos - 4 between blue and red,  3 between blue and green, and 4 between red and green. 

The reason for these choices of interactions is two-fold: first, as we will see in Sections \ref{vertexsection} and \ref{kcolors}, these interactions can be obtain from certain vertex models whose integrability is essential in our analysis. Secondly, these tilings are related to (super) LLT polynomials \cite{GKsuper,KN23} and are an example of an LLT process, generalizing the well-studied Schur process.

\section{Colored vertex models} \label{vertexsection}

\subsection{Defining the vertex models} 
\label{k-color-vertex-models-section}

Next we discuss several families of vertex models, which generalize the vertices introduced in Section \ref{aztec}.  These vertex models were originally defined and studied in \cite{LLT,GKsuper}, although they can be realized as degenerations of vertex models introduced in \cite{color1} (see \cite[Lemma 3.1]{GKsuper}).  

We begin with some notation.  For a vector $\I = (I_1,\ldots,I_k) \in \mathbb{R}^k$, we define \[ |\I| = \sum_{m=1}^k I_i. \]
For vectors $\I = (I_1,\ldots,I_k), \J = (J_1,\ldots,J_k) \in \mathbb{R}^k$, we define
\[ \varphi(\I,\J) = \sum_{1 \leq i < j \leq k} I_iJ_j. \] 
For variables $x$ and $t$ and an integer $n \geq 0$, we define the $t$-Pochhammer symbol
\[ (x;t)_n = \prod_{m=0}^{n-1} (1-xt^m). \]

We first define our vertices algebraically.  The $L$, $L'$, $\M$, and $\M'$ matrices are families of functions $(\{0,1\}^k)^4 \rightarrow \C[x,t]$, one for each integer $k \geq 0$, whereas the  $R'$ matrix is a family of functions $(\{0,1\}^k)^4 \rightarrow \C(x,y,t)$, one for each integer $k \geq 0$.  In other words, each vertex associates a \textbf{weight} (either a polynomial in $x,t$ or a rational function in $x,y,t$) to every 4-tuple of vectors in $\{0,1\}^k$ for each integer $k \geq 0$.
\begin{equation}\label{tab:algdef}
\resizebox{.9\textwidth}{!}{
\bgroup
\def\arraystretch{2.5}
\begin{tabular}{||c|c||}
\hline
Type of vertex & Algebraic definition \\ \hline \hline 
$L$ & $\displaystyle L_x^{(k)}(\I,\J,\K,\L) = \textbf{1}_{\I+\J=\K+\L} \prod_{i=1}^k \textbf{1}_{I_i+J_i \neq 2} \cdot x^{|\L|} t^{\varphi(\L,\I+\J)}$ \\[5pt] \hline
$L'$ & $\displaystyle L_x^{'(k)}(\I,\J,\K,\L) = \textbf{1}_{\I+\J=\K+\L} \prod_{i=1}^k \textbf{1}_{K_i \geq J_i} \cdot x^{|\L|} t^{\varphi(\L,\K-\J)}$ \\[5pt] \hline
$\M$ & $\displaystyle \M_x^{(k)}(\I,\J,\K,\L) = x^kt^{\binom{k}{2}}L^{(k)}_{1/(xt^{k-1})}(\I,\J,\K,\L)$  \\[5pt] \hline
$\M'$ & $\displaystyle \M_x^{'(k)}(\I,\J,\K,\L) = x^k L_{1/x}^{'(k)}(\I,\J,\K,\L)$  \\[5pt] \hline
$R'$ &  $\displaystyle R_{y/x}^{'(k)}(\I,\J,\K,\L) = \textbf{1}_{\I+\J=\K+\L} \prod_{i=1}^k \textbf{1}_{I_i + J_i \neq 2} \cdot (x/y)^{|\L|} (-x/y;t)_{|\K|+|\L|}^{-1} t^{\varphi(\L,\K+\L)}$ \\[5pt] \hline
\end{tabular}
\egroup
}
\end{equation}

\indent However, it is often useful to think of a vertex graphically.  We can draw a vertex as a face with four incident edges, each labelled by an element of $\{0,1\}^k$.  A face takes one of two forms, 
\[ \begin{tikzpicture}[baseline=(current bounding box.center)]\draw[thin] (0,0) rectangle (1,1); \node[left] at (0,0.5) {$\J$}; \node[right] at (1,0.5) {$\L$};\node[above] at (0.5,1) {$\K$};\node[below] at (0.5,0) {$\I$}; \end{tikzpicture} \text{(a box) or } \begin{tikzpicture}[baseline=(current bounding box.center)] \draw[thin] (0,1)--(1,0);\draw[thin] (0,0)--(1,1); \node at (-0.1,-0.1) {$\I$}; \node at (-0.1,1.1) {$\J$}; \node at (1.1,1.1) {$\K$}; \node at (1.1,-0.1) {$\L$};\end{tikzpicture} \text{(a cross).} \]
The edge labels describe colored paths moving through the face SW-to-NE (in a box) or left-to-right (in a cross).  If an edge has the label $\I = (I_1,\ldots,I_k) \in \{0,1\}^k$, then for each $i \in [k]$, a path of color $i$ is incident at the edge if and only if $I_i = 1$.  For example, with $k=2$ (letting blue be color 1 and red be color 2), the path configuration associated to the edge labels
\[ \I = (0,1), \J = (1,0), \K = (0,1), \L = (1,0) \]
is 
\[
\begin{tikzpicture}[baseline=(current bounding box.center)] \draw[thin] (0,0) rectangle (1,1); \draw[blue,thick] (0,0.5)--(1,0.5); \draw[red, thick] (0.5,0)--(0.5,1); \end{tikzpicture} 
\text{ (for a box) or }
\begin{tikzpicture}[baseline=(current bounding box.center)] \draw[blue, ultra thick] (0,1)--(1,0); \draw[red, ultra thick] (0,0)--(1,1); \end{tikzpicture}
\text{ (for a cross).}
\]
The factor of $\textbf{1}_{\I+\J=\K+\L}$ that appears in the algebraic definitions of all five vertices imposes a \textbf{path conservation} restriction: in order for a vertex to have a non-zero weight, the paths entering the vertex and the paths exiting the vertex must be the same.  To define the vertex weights graphically, we start by defining the weights in the case $k=1$.
\begin{equation}\label{tab:1colorgraphical}
\resizebox{0.9\textwidth}{!}{
\begin{tabular}{||c|c||}
\hline
Type of vertex & One-color definition \\ \hline\hline
$L$ & \resizebox{0.8\textwidth}{!}{
\begin{tabular}{cccccc}
   \begin{tikzpicture}[baseline=(current bounding box.center)]\draw[thin] (0,0) rectangle (1,1); \node[left] at (0,0.5) {j}; \node[right] at (1,0.5) {l};\node[above] at (0.5,1) {k};\node[below] at (0.5,0) {i};\node at (0.5,0.5) {$x$}; \end{tikzpicture}: & \begin{tikzpicture}[baseline=(current bounding box.center)]\draw[thin] (0,0) rectangle (1,1);\end{tikzpicture} & \begin{tikzpicture}[baseline=(current bounding box.center)]\draw[thin] (0,0) rectangle (1,1);\draw[blue, thick] (0.5,0)--(0.5,0.5)--(1,0.5);\end{tikzpicture} & \begin{tikzpicture}[baseline=(current bounding box.center)] \draw[thin] (0,0) rectangle (1,1); \draw[blue, thick] (0,0.5)--(1,0.5); \end{tikzpicture} & \begin{tikzpicture}[baseline=(current bounding box.center)] \draw[thin] (0,0) rectangle (1,1); \draw[blue, thick] (0.5,0)--(0.5,1); \end{tikzpicture} & \begin{tikzpicture}[baseline=(current bounding box.center)] \draw[thin] (0,0) rectangle (1,1); \draw[blue, thick] (0,0.5)--(0.5,0.5)--(0.5,1); \end{tikzpicture} \\
$L_x^{(1)}(i,j,k,l)$: &  $1$ & $x$ & $x$ & $1$ & $1$ \\
\end{tabular}
} \\[40pt] \hline
$L'$ & \resizebox{0.8\textwidth}{!}{
\begin{tabular}{cccccc}
   \begin{tikzpicture}[baseline=(current bounding box.center)]\draw[thin,fill=violet] (0,0) rectangle (1,1); \node[left] at (0,0.5) {j}; \node[right] at (1,0.5) {l};\node[above] at (0.5,1) {k};\node[below] at (0.5,0) {i};\node at (0.5,0.5) {$x$};\end{tikzpicture}: & \begin{tikzpicture}[baseline=(current bounding box.center)] \draw[thin,fill=violet] (0,0) rectangle (1,1); \draw[blue, thick] (0.5,0)--(0.5,1); \end{tikzpicture} & \begin{tikzpicture}[baseline=(current bounding box.center)] \draw[thin,fill=violet] (0,0) rectangle (1,1); \draw[blue, thick] (0,0.5)--(0.5,0.5)--(0.5,1); \end{tikzpicture} & \begin{tikzpicture}[baseline=(current bounding box.center)] \draw[thin,fill=violet] (0,0) rectangle (1,1); \draw[blue, thick] (0,0.5)--(1,0.5); \draw[blue, thick] (0.5,0)--(0.5,1); \end{tikzpicture} & \begin{tikzpicture}[baseline=(current bounding box.center)]\draw[thin,fill=violet] (0,0) rectangle (1,1); \end{tikzpicture} & \begin{tikzpicture}[baseline=(current bounding box.center)]\draw[thin,fill=violet] (0,0) rectangle (1,1);\draw[blue, thick] (0.5,0)--(0.5,0.5)--(1,0.5);\end{tikzpicture} \\
$L_x^{'(1)}(i,j,k,l)$: &  $1$ & $1$ & $x$ & $1$ & $x$ \\
\end{tabular}
} \\[40pt] \hline
$\M$ & \resizebox{0.8\textwidth}{!}{
\begin{tabular}{cccccc}
   \begin{tikzpicture}[baseline=(current bounding box.center)]\draw[thin,fill=gray] (0,0) rectangle (1,1); \node[left] at (0,0.5) {j}; \node[right] at (1,0.5) {l};\node[above] at (0.5,1) {k};\node[below] at (0.5,0) {i};\node at (0.5,0.5) {$ x$}; \end{tikzpicture}: & \begin{tikzpicture}[baseline=(current bounding box.center)]\draw[thin,fill=gray] (0,0) rectangle (1,1);\end{tikzpicture} & \begin{tikzpicture}[baseline=(current bounding box.center)]\draw[thin,fill=gray] (0,0) rectangle (1,1);\draw[blue, thick] (0.5,0)--(0.5,0.5)--(1,0.5);\end{tikzpicture} & \begin{tikzpicture}[baseline=(current bounding box.center)] \draw[thin,fill=gray] (0,0) rectangle (1,1); \draw[blue, thick] (0,0.5)--(1,0.5); \end{tikzpicture} & \begin{tikzpicture}[baseline=(current bounding box.center)] \draw[thin,fill=gray] (0,0) rectangle (1,1); \draw[blue, thick] (0.5,0)--(0.5,1); \end{tikzpicture} & \begin{tikzpicture}[baseline=(current bounding box.center)] \draw[thin,fill=gray] (0,0) rectangle (1,1); \draw[blue, thick] (0,0.5)--(0.5,0.5)--(0.5,1); \end{tikzpicture} \\
$\M_x^{(1)}(i,j,k,l)$: &  $x$ & $1$ & $1$ & $x$ & $x$ \\
\end{tabular}
} \\[40pt] \hline
$\M'$ & \resizebox{0.8\textwidth}{!}{
\begin{tabular}{cccccc}
   \begin{tikzpicture}[baseline=(current bounding box.center)]\draw[thin,fill=pink] (0,0) rectangle (1,1); \node[left] at (0,0.5) {j}; \node[right] at (1,0.5) {l};\node[above] at (0.5,1) {k};\node[below] at (0.5,0) {i};\node at (0.5,0.5) {$ x$};\end{tikzpicture}: & \begin{tikzpicture}[baseline=(current bounding box.center)] \draw[thin,fill=pink] (0,0) rectangle (1,1); \draw[blue, thick] (0.5,0)--(0.5,1); \end{tikzpicture} & \begin{tikzpicture}[baseline=(current bounding box.center)] \draw[thin,fill=pink] (0,0) rectangle (1,1); \draw[blue, thick] (0,0.5)--(0.5,0.5)--(0.5,1); \end{tikzpicture} & \begin{tikzpicture}[baseline=(current bounding box.center)] \draw[thin,fill=pink] (0,0) rectangle (1,1); \draw[blue, thick] (0,0.5)--(1,0.5); \draw[blue, thick] (0.5,0)--(0.5,1); \end{tikzpicture} & \begin{tikzpicture}[baseline=(current bounding box.center)]\draw[thin,fill=pink] (0,0) rectangle (1,1); \end{tikzpicture} & \begin{tikzpicture}[baseline=(current bounding box.center)]\draw[thin,fill=pink] (0,0) rectangle (1,1);\draw[blue, thick] (0.5,0)--(0.5,0.5)--(1,0.5);\end{tikzpicture} \\
$\M_x^{'(1)}(i,j,k,l)$: &  $x$ & $x$ & $1$ & $x$ & $1$ \\
\end{tabular}
} \\[40pt] \hline
$R'$ & \resizebox{0.8\textwidth}{!}{
\begin{tabular}{cccccc}
\begin{tikzpicture}[baseline=(current bounding box.center)] \fill[yellow] (0,0) rectangle (1,1); \draw[thin] (0,1)--(1,0);\draw[thin] (0,0)--(1,1); \node at (-0.1,-0.1) {$i$}; \node at (-0.5,-0.1) {$y$}; \node at (-0.1,1.1) {$j$}; \node at (-0.5,1.1) {$x$}; \node at (1.1,1.1) {$k$}; \node at (1.1,-0.1) {$l$};\end{tikzpicture}: & \begin{tikzpicture}[baseline=(current bounding box.center)] \fill[yellow] (0,0) rectangle (1,1); \draw[blue, ultra thick] (0,1)--(1,0);\draw[thin] (0,0)--(1,1); \end{tikzpicture} & \begin{tikzpicture}[baseline=(current bounding box.center)] \fill[yellow] (0,0) rectangle (1,1); \draw[blue, ultra thick] (0,1)--(0.5,0.5)--(1,1);\draw[thin] (0,0)--(0.5,0.5)--(1,0); \end{tikzpicture} & \begin{tikzpicture}[baseline=(current bounding box.center)] \fill[yellow] (0,0) rectangle (1,1); \draw[thin] (0,1)--(0.5,0.5)--(1,1);\draw[blue, ultra thick] (0,0)--(0.5,0.5)--(1,0); \end{tikzpicture} & \begin{tikzpicture}[baseline=(current bounding box.center)] \fill[yellow] (0,0) rectangle (1,1); \draw[thin] (0,1)--(1,0);\draw[blue, ultra thick] (0,0)--(1,1); \end{tikzpicture} & \begin{tikzpicture}[baseline=(current bounding box.center)] \fill[yellow] (0,0) rectangle (1,1); \draw[thin] (0,1)--(1,0);\draw[thin] (0,0)--(1,1); \end{tikzpicture} \\
$R_{y/x}^{'(1)}(i,j,k,l)$:&  $\frac{1}{1+y/x}$ & $\frac{y/x}{1+y/x}$ & $\frac{1}{1+y/x}$ & $\frac{y/x}{1+y/x}$ & $1$ 
\end{tabular}
} \\[40pt] \hline
\end{tabular}
}
\end{equation}

The $k$-color weights are then defined in terms of the one-color weights.  
\begin{equation}\label{tab:kcolorgraphical}
\bgroup
\def\arraystretch{1.9}
\begin{tabular}{||c|c||}
\hline
Type of vertex & $k$-color definition \\ \hline \hline 
$L$ & \begin{tabular}{c} $\displaystyle L_x^{(k)}(\I,\J,\K,\L) =  \prod_{i=1}^k L_{xt^{\delta_i}}^{(1)}(I_i,J_i,K_i,L_i)$ \\ where $\delta_i = \text{\# colors greater than $i$ that are present}$ \end{tabular} \\ \hline
$L'$ & \begin{tabular}{c} $L_x^{'(k)}(\I,\J,\K,\L) = \prod_{i=1}^k L_{xt^{\delta'_i}}^{'(1)}(I_i,J_i,K_i,L_i)$ 
 \\ where $\delta'_i = \text{\# colors greater than $i$ of the form}$ \resizebox{0.04\textwidth}{!}{\begin{tikzpicture}[baseline=(current bounding box.center)] \draw[thin,fill=violet] (0,0) rectangle (1,1); \draw[blue, thick] (0.5,0)--(0.5,1); \end{tikzpicture}} \end{tabular} \\[18pt] \hline
$\M$ & \begin{tabular}{c} $\displaystyle \M_x^{(k)}(\I,\J,\K,\L) =  \prod_{i=1}^k t^{\beta_i} M_{xt^{\alpha_i-\beta_i}}^{(1)}(I_i,J_i,K_i,L_i)$ \\ where $\begin{array}{ll} \alpha_i = \text{\# colors greater than $i$ that don't exit right}  \\ \beta_i = \text{\# colors greater than $i$ that exit top} \end{array}$ \end{tabular} \\ \hline
$\M'$ & \begin{tabular}{c} $\M_x^{'(k)}(\I,\J,\K,\L) = \prod_{i=1}^k t^{\gamma_i} \M_{xt^{-\gamma_i}}^{'(1)}(I_i,J_i,K_i,L_i)$ \\ where $\gamma_i = \text{\# colors greater than $i$ of the form}$ \resizebox{0.04\textwidth}{!}{\begin{tikzpicture}[baseline=(current bounding box.center)] \draw[thin,fill=pink] (0,0) rectangle (1,1); \draw[blue, thick] (0.5,0)--(0.5,1); \end{tikzpicture}} \end{tabular} \\[18pt] \hline
$R'$ &  \begin{tabular}{c} $\displaystyle R_{y/x}^{'(k)}(\I,\J,\K,\L) = \prod_{i=1}^k R_{y/(xt^{\epsilon'_i})}^{'(1)}(I_i,J_i,K_i,L_i)$ \\ where $\epsilon'_i = \text{\# colors greater than $i$ that are present}$ \end{tabular}
\\ \hline
\end{tabular}
\egroup
\end{equation}

\noindent For the white $L$ and gray $M$ weights, the equivalence of the algebraic and graphical definitions is shown in Appendix \ref{ap:b}.  The equivalence for the other vertices can be shown similarly.

\indent A \textbf{lattice} is a rectangular grid of vertices, with the variables and the labels on the outer edges specified, but with the labels on the internal edges unspecified.  A \textbf{lattice configuration} is a lattice with the labels on the internal edges specified, such that the weight of each vertex is non-zero.  The weight of a lattice configuration is the product of the weights of the vertices.  Given a lattice $L$, the associated \textbf{partition function} is
\[ \sum_{C \in LC(L)} \weight(C) \]
where $LC(L)$ is the set of valid lattice configurations on $L$.  Often, when it is clear from context, we will abuse notation and let the drawing of the lattice be equal to the partition function of the vertex model on the lattice.

Our vertices satisfy several Yang-Baxter equations.  We state two of them here; both follow from the Yang-Baxter equation for purple and white vertices given in \cite[Proposition 3.3]{GKsuper} after an appropriate change of variables.


\begin{prop} \label{prop:YBEpurplegray}

The $L'$, $\M$, and $R'$ matrices satisfy the Yang-Baxter equation
\[
\sum_{\text{interior paths}}
w\left(
\resizebox{1.8cm}{!}{
  \begin{tikzpicture}[baseline=(current bounding box.center)] 
  \def\vS{1.8}
  \draw[thin,fill=violet] (0,1) rectangle (1,2);
  \draw[thin,fill=gray] (0,0) rectangle (1,1);
  \fill[yellow] (-1,0.5) rectangle (0,1.5);
 \draw[] (-1,0.5) -- (0,1.5); \draw[] (-1,1.5) -- (0,0.5); 
 \draw[step=1.0,black,thin] (0,0) grid (1,2); 
 \node[scale=\vS] at (0.5,0.5) {$x$}; \node[scale=\vS] at (0.5,1.5) {$y$};
 \node[left,scale=\vS] at (-1,1.5) {$\J_1$}; \node[left,scale=\vS] at (-1,0.5) {$\I_1$}; \node[below,scale=\vS] at (0.5,0) {$\K_1$};
 \node[right,scale=\vS] at (1,1.5) {$\I_3$}; \node[right,scale=\vS] at (1,0.5) {$\J_3$}; \node[above,scale=\vS] at (0.5,2) {$\K_3$};
 \end{tikzpicture} 
 }
 \right)
=
\sum_{\text{interior paths}}
w\left(
\resizebox{1.8cm}{!}{
 \begin{tikzpicture}[baseline=(current bounding box.center)] 
 \def\vS{1.8}
 \draw[thin,fill=gray] (0,1) rectangle (1,2);
 \draw[thin,fill=violet] (0,0) rectangle (1,1); \fill[yellow] (1,0.5) rectangle (2,1.5);
 \draw[] (2,0.5) -- (1,1.5); \draw[] (2,1.5) -- (1,0.5); 
 \draw[step=1.0,black,thin] (0,0) grid (1,2); 
 \node[scale=\vS] at (0.5,0.5) {$y$}; \node[scale=\vS] at (0.5,1.5) {$x$};
 \node[left,scale=\vS] at (0,1.5) {$\J_1$}; \node[left,scale=\vS] at (0,0.5) {$\I_1$}; \node[below,scale=\vS] at (0.5,0) {$\K_1$};
 \node[right,scale=\vS] at (2,0.5) {$\J_3$}; \node[right,scale=\vS] at (2,1.5) {$\I_3$}; \node[above,scale=\vS] at (0.5,2) {$\K_3$};
 \end{tikzpicture}
 }
 \right)
\]
for any choice of boundary condition $\I_1,\J_1,\K_1,\I_3,\J_3,\K_3$, where the yellow cross vertex has parameter $xy t^{k-1}$.

\end{prop}

\begin{prop} \label{prop:YBEpinkwhite}

The $L$, $\M'$, and $R'$ matrices satisfy the Yang-Baxter equation
\[
\sum_{\text{interior paths}}
w\left(
\resizebox{1.8cm}{!}{
  \begin{tikzpicture}[baseline=(current bounding box.center)] 
  \def\vS{1.8}
 \draw[thin,fill=pink] (0,1) rectangle (1,2); \fill[yellow] (-1,0.5) rectangle (0,1.5);
 \draw[] (-1,0.5) -- (0,1.5); \draw[] (-1,1.5) -- (0,0.5); 
 \draw[step=1.0,black,thin] (0,0) grid (1,2); 
 \node[scale=\vS] at (0.5,0.5) {$x$}; \node[scale=\vS] at (0.5,1.5) {$y$};
 \node[left,scale=\vS] at (-1,1.5) {$\J_1$}; \node[left,scale=\vS] at (-1,0.5) {$\I_1$}; \node[below,scale=\vS] at (0.5,0) {$\K_1$};
 \node[right,scale=\vS] at (1,1.5) {$\I_3$}; \node[right,scale=\vS] at (1,0.5) {$\J_3$}; \node[above,scale=\vS] at (0.5,2) {$\K_3$};
 \end{tikzpicture} 
 }
 \right)
=
\sum_{\text{interior paths}}
w\left(
\resizebox{1.8cm}{!}{
 \begin{tikzpicture}[baseline=(current bounding box.center)] 
 \def\vS{1.8}
 \draw[thin,fill=pink] (0,0) rectangle (1,1); \fill[yellow] (1,0.5) rectangle (2,1.5);
 \draw[] (2,0.5) -- (1,1.5); \draw[] (2,1.5) -- (1,0.5); 
 \draw[step=1.0,black,thin] (0,0) grid (1,2); 
 \node[scale=\vS] at (0.5,0.5) {$ y$}; \node[scale=\vS] at (0.5,1.5) {$x$};
 \node[left,scale=\vS] at (0,1.5) {$\J_1$}; \node[left,scale=\vS] at (0,0.5) {$\I_1$}; \node[below,scale=\vS] at (0.5,0) {$\K_1$};
 \node[right,scale=\vS] at (2,0.5) {$\J_3$}; \node[right,scale=\vS] at (2,1.5) {$\I_3$}; \node[above,scale=\vS] at (0.5,2) {$\K_3$};
 \end{tikzpicture}
 }
 \right)
\]
for any choice of boundary condition $\I_1,\J_1,\K_1,\I_3,\J_3,\K_3$, where the yellow cross vertex has parameter $\frac{1}{yx}$.

\end{prop}

\begin{proof}

Fix $\I_1,\J_1,\K_1,\I_3,\J_3,\K_3$.  Written algebraically, \cite[Proposition 3.3]{GKsuper} states
\[ \begin{aligned}
&\sum_{\I_2,\J_2,\K_2} R'_{y/x}(\I_1,\J_1;\I_2,\J_2) L_x(\K_1,\J_2;\K_2,\J_3) L'_y(\K_2,\I_2;\K_3,\I_3) \\ 
&= \sum_{\I_2,\J_2,\K_2} L'_y(\K_1,\I_1;\K_2,\I_2) L_x(\K_2,\J_1;\K_3,\J_2) R'_{y/x}(\I_2,\J_2;\I_3,\J_3). 
\end{aligned} \]
If we replace $x$ with $\frac{1}{xt^{k-1}}$ and multiply both sides by $x^k t^{k \choose 2}$, we get
\[ \begin{aligned}
&\sum_{\I_2,\J_2,\K_2} R'_{xyt^{k-1}}(\I_1,\J_1;\I_2,\J_2) \cdot x^k t^{k \choose 2} L_{1/(xt^{k-1})}(\K_1,\J_2;\K_2,\J_3) \cdot L'_y(\K_2,\I_2;\K_3,\I_3) \\ 
&= \sum_{\I_2,\J_2,\K_2} L'_y(\K_1,\I_1;\K_2,\I_2) \cdot x^k t^{k \choose 2} L_{1/(xt^{k-1})}(\K_2,\J_1;\K_3,\J_2) \cdot R'_{xyt^{k-1}}(\I_2,\J_2;\I_3,\J_3)
\end{aligned} \]
which can be rewritten using Table \ref{tab:algdef} as
\[ \begin{aligned}
&\sum_{\I_2,\J_2,\K_2} R'_{xyt^{k-1}}(\I_1,\J_1;\I_2,\J_2) M_x(\K_1,\J_2;\K_2,\J_3) L'_y(\K_2,\I_2;\K_3,\I_3) \\ 
&= \sum_{\I_2,\J_2,\K_2} L'_y(\K_1,\I_1;\K_2,\I_2) M_x(\K_2,\J_1;\K_3,\J_2) R'_{xyt^{k-1}}(\I_2,\J_2;\I_3,\J_3)
\end{aligned} \]
which is Proposition \ref{prop:YBEpurplegray} written algebraically. If instead we replace $y$ with $\frac{1}{y}$ and multiply both sides by $y^k$, we get 
\[ \begin{aligned}
&\sum_{\I_2,\J_2,\K_2} R'_{1/(yx)}(\I_1,\J_1;\I_2,\J_2) \cdot L_x(\K_1,\J_2;\K_2,\J_3) \cdot y^k L'_{1/y}(\K_2,\I_2;\K_3,\I_3) \\ 
&= \sum_{\I_2,\J_2,\K_2} y^k L'_{1/y}(\K_1,\I_1;\K_2,\I_2) \cdot L_x(\K_2,\J_1;\K_3,\J_2) \cdot R'_{1/(yx)}(\I_2,\J_2;\I_3,\J_3)
\end{aligned} \]
which can be rewritten using Table \ref{tab:algdef} as 
\[ \begin{aligned}
&\sum_{\I_2,\J_2,\K_2} R'_{1/(yx)}(\I_1,\J_1;\I_2,\J_2) L_x(\K_1,\J_2;\K_2,\J_3) M'_y(\K_2,\I_2;\K_3,\I_3) \\ 
&= \sum_{\I_2,\J_2,\K_2} M'_{y}(\K_1,\I_1;\K_2,\I_2) L_x(\K_2,\J_1;\K_3,\J_2) R'_{1/(yx)}(\I_2,\J_2;\I_3,\J_3)
\end{aligned} \]
which is Proposition \ref{prop:YBEpinkwhite} written algebraically.
    
\end{proof}

\subsection{The vertex models and (co-)interlacing partitions}
\label{1-color-vertex-models-section}

To better understand the colored vertex models defined in the previous subsection, we will begin by looking at the case $k=1$.  In fact, in the case $k=1$, there is a very natural interpretation of rows of vertices in terms of (co-)interlacing partitions.  Given two partitions $\mu$ and $\lambda$ such that $\lambda_1+\ell(\lambda)\le n$, we can draw rows of $n$ vertices with border conditions as follows.
\[
\resizebox{\textwidth}{!}{
\begin{tikzpicture}[baseline=(current bounding box.center)]
\def\vS{1.8}
\draw[step=1.0,black,thin] (0,0) grid (6,1); 
\node[below, scale=\vS] at (3,0) {$\mu$};
\node[above, scale=\vS] at (3,1) {$\lambda$};
\node at (3,0) {x}; 
\node at (3,1) {x}; 
\draw[fill=white] (0,0.5) circle (0.1);
\draw[fill=white] (6,0.5) circle (0.1);
\end{tikzpicture}\ \ \ 
\begin{tikzpicture}[baseline=(current bounding box.center)]
\def\vS{1.8}
\draw[fill=gray] (0,0) rectangle (6,1);
\draw[step=1.0,black,thin, fill=gray] (0,0) grid (6,1); 
\node[above, scale=\vS] at (3,1) {$\mu$};
\node[below, scale=\vS] at (3,0) {$\lambda$};
\node at (2,1) {x}; 
\node at (3,0) {x}; 
\draw[fill=white] (0,0.5) circle (0.1);
\draw[fill=blue] (6,0.5) circle (0.1);
\end{tikzpicture}\ \ \  
\begin{tikzpicture}[baseline=(current bounding box.center)]
\def\vS{1.8}
\draw[fill=violet] (0,0) rectangle (6,1);
\draw[step=1.0,black,thin] (0,0) grid (6,1); 
\node[below, scale=\vS] at (3,0) {$\mu'$};
\node[above, scale=\vS] at (3,1) {$\lambda'$};
\node at (3,1) {x}; 
\node at (3,0) {x}; 
\draw[fill=white] (0,0.5) circle (0.1);
\draw[fill=white] (6,0.5) circle (0.1);
\end{tikzpicture}\ \ \
\begin{tikzpicture}[baseline=(current bounding box.center)]
\def\vS{1.8}
\draw[fill=pink] (0,0) rectangle (6,1);
\draw[step=1.0,black,thin] (0,0) grid (6,1); 
\node[above, scale=\vS] at (3,1) {$\mu'$};
\node[below, scale=\vS] at (3,0) {$\lambda'$};
\node at (4,1) {x}; 
\node at (3,0) {x}; 
\draw[fill=blue] (0,0.5) circle (0.1);
\draw[fill=white] (6,0.5) circle (0.1);
\end{tikzpicture}
}
\]
Here a partition on the boundary of the array means that the border condition is given by the corresponding Maya diagram, and we mark the position of the 0 content line with an x.  We give an example with $\lambda=(3,3,2)$ and $\mu=(2,2,0)$ in Figure \ref{vertex}.  It is easy to see that each row has no valid configurations unless $\mu \preceq\lambda$, in which case it has one valid configuration with weight $x^{|\lambda|-|\mu|}$.  Thus the weight of each row is
\[ \left \{ \begin{array}{ll} x^{|\lambda|-|\mu|} & \text{if $\mu \preceq\lambda$} \\ 0 & \text{otherwise} \end{array} \right. . \]

\begin{figure}[ht]
\resizebox{\textwidth}{!}{
\begin{tikzpicture}[baseline=(current bounding box.center)]
\def\vS{1.8}
\draw[step=1.0,black,thin] (0,0) grid (6,1); 
\node[below, scale=\vS] at (3,0) {$\mu$};
\node[above, scale=\vS] at (3,1) {$\lambda$};

\node at (3,0) {x}; 
\node at (3,1) {x}; 
\draw[fill=white] (0,0.5) circle (0.1);
\draw[fill=white] (6,0.5) circle (0.1);

\draw[fill=blue] (0.5,0.0) circle (0.1);
\draw[fill=white] (1.5,0.0) circle (0.1);
\draw[fill=white] (2.5,0.0) circle (0.1);
\draw[fill=blue] (3.5,0.0) circle (0.1);
\draw[fill=blue] (4.5,0.0) circle (0.1);
\draw[fill=white] (5.5,0.0) circle (0.1);

\draw[fill=white] (0.5,1.0) circle (0.1);
\draw[fill=white] (1.5,1.0) circle (0.1);
\draw[fill=blue] (2.5,1.0) circle (0.1);
\draw[fill=blue] (3.5,1.0) circle (0.1);
\draw[fill=white] (4.5,1.0) circle (0.1);
\draw[fill=blue] (5.5,1.0) circle (0.1);
\draw[ultra thick,blue](0.5,0)--(0.5,0.5)--(2.5,0.5)--(2.5,1);
\draw[ultra thick,blue](3.5,0)--(3.5,1);
\draw[ultra thick,blue](4.5,0)--(4.5,0.5)--(5.5,0.5)--(5.5,1);
\end{tikzpicture}\ \ \
\begin{tikzpicture}[baseline=(current bounding box.center)]
\def\vS{1.8}
\draw[fill=gray] (0,0) rectangle (6,1);
\draw[step=1.0,black,thin, fill=gray] (0,0) grid (6,1); 

\node[above, scale=\vS] at (3,1) {$\mu$};
\node[below, scale=\vS] at (3,0) {$\lambda$};

\draw[fill=white] (0.5,1.0) circle (0.1);
\draw[fill=white] (1.5,1.0) circle (0.1);
\draw[fill=blue] (2.5,1.0) circle (0.1);
\draw[fill=blue] (3.5,1.0) circle (0.1);
\draw[fill=white] (4.5,1.0) circle (0.1);
\draw[fill=white] (5.5,1.0) circle (0.1);

\draw[fill=white] (0.5,0.0) circle (0.1);
\draw[fill=white] (1.5,0.0) circle (0.1);
\draw[fill=blue] (2.5,0.0) circle (0.1);
\draw[fill=blue] (3.5,0.0) circle (0.1);
\draw[fill=white] (4.5,0.0) circle (0.1);
\draw[fill=blue] (5.5,0.0) circle (0.1);
\node at (2,1) {x}; 
\node at (3,0) {x}; 
\draw[fill=white] (0,0.5) circle (0.1);
\draw[fill=blue] (6,0.5) circle (0.1);

\draw[ultra thick,blue](2.5,0)--(2.5,1);
\draw[ultra thick,blue](3.5,0)--(3.5,1);
\draw[ultra thick,blue](5.5,0)--(5.5,0.5)--(6,.5);

\end{tikzpicture}\ \ \
\begin{tikzpicture}[baseline=(current bounding box.center)]
\def\vS{1.8}
\draw[fill=violet] (0,0) rectangle (6,1);

\draw[step=1.0,black,thin] (0,0) grid (6,1); 
\node[below, scale=\vS] at (3,0) {$\mu'$};
\node[above, scale=\vS] at (3,1) {$\lambda'$};
\node at (3,1) {x}; 
\node at (3,0) {x}; 
\draw[fill=white] (0,0.5) circle (0.1);
\draw[fill=white] (6,0.5) circle (0.1);

\draw[fill=blue] (0.5,0.0) circle (0.1);
\draw[fill=white] (1.5,0.0) circle (0.1);
\draw[fill=blue] (4.5,0.0) circle (0.1);
\draw[fill=blue] (3.5,0.0) circle (0.1);
\draw[fill=white] (5.5,0.0) circle (0.1);
\draw[fill=white] (2.5,0.0) circle (0.1);
\draw[fill=white] (0.5,1.0) circle (0.1);
\draw[fill=blue] (1.5,1.0) circle (0.1);
\draw[fill=white] (2.5,1.0) circle (0.1);
\draw[fill=white] (3.5,1.0) circle (0.1);
\draw[fill=blue] (4.5,1.0) circle (0.1);
\draw[fill=blue] (5.5,1.0) circle (0.1);
\draw[ultra thick, blue](0.5,0)--(0.5,0.5)--(1.5,0.5)--(1.5,1);
\draw[ultra thick, blue](3.5,0)--(3.5,0.5)--(5.5,0.5)--(5.5,1);
\draw[ultra thick, blue](4.5,0)--(4.5,1);

\end{tikzpicture}\ \ \ 
\begin{tikzpicture}[baseline=(current bounding box.center)]
\def\vS{1.8}
\draw[fill=pink] (0,0) rectangle (6,1);
\draw[step=1.0,black,thin] (0,0) grid (6,1); 
\node[above, scale=\vS] at (3,1) {$\mu'$};
\node[below, scale=\vS] at (3,0) {$\lambda'$};
\node at (4,1) {x}; 
\node at (3,0) {x}; 
\draw[fill=blue] (0,0.5) circle (0.1);
\draw[fill=white] (6,0.5) circle (0.1);
\draw[fill=blue] (0.5,1.0) circle (0.1);
\draw[fill=blue] (1.5,1.0) circle (0.1);
\draw[fill=blue] (4.5,1.0) circle (0.1);
\draw[fill=blue] (5.5,1.0) circle (0.1);
\draw[fill=white] (3.5,1.0) circle (0.1);
\draw[fill=white] (2.5,1.0) circle (0.1);

\draw[fill=white] (0.5,0.0) circle (0.1);
\draw[fill=blue] (1.5,0.0) circle (0.1);
\draw[fill=white] (2.5,0.0) circle (0.1);
\draw[fill=white] (3.5,0.0) circle (0.1);
\draw[fill=blue] (4.5,0.0) circle (0.1);
\draw[fill=blue] (5.5,0.0) circle (0.1);
\draw[ultra thick, blue](0,0.5)--(0.5,0.5)--(0.5,1);
\draw[ultra thick, blue](1.5,0)--(1.5,1);
\draw[ultra thick, blue](4.5,0)--(4.5,1);
\draw[ultra thick, blue](5.5,0)--(5.5,1);

\end{tikzpicture}}
\caption{Rows of vertices for $\lambda=(3,3,2)$ and $\mu=(2,2,0)$}
\label{vertex}
\end{figure}
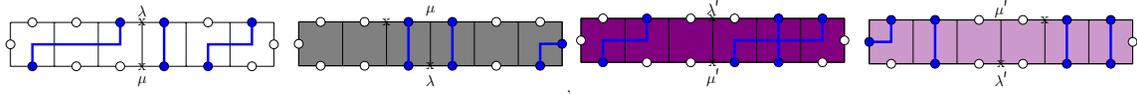

It is easy to generalize this interpretation of rows of vertices to general values of $k$.  Given two $k$-tuples of partitions $\bm{\mu}$ and $\bm{\lambda}$ such that $\lambda^{(i)}_1 + \ell(\lambda^{(i)}) \leq n$ for all $i \in [k]$, we can draw rows of $n$ vertices with border conditions as follows.
\[
\resizebox{\textwidth}{!}{
\begin{tikzpicture}[baseline=(current bounding box.center)]
\def\vS{1.8}
\draw[step=1.0,black,thin] (0,0) grid (6,1); 
\node[below, scale=\vS] at (3,0) {$\bm{\mu}$};
\node[above, scale=\vS] at (3,1) {$\bm{\lambda}$};
\node at (3,0) {x}; 
\node at (3,1) {x}; 
\draw[fill=white] (0,0.5) circle (0.1);
\draw[fill=white] (6,0.5) circle (0.1);
\end{tikzpicture}\ \ \ 
\begin{tikzpicture}[baseline=(current bounding box.center)]
\def\vS{1.8}
\draw[fill=gray] (0,0) rectangle (6,1);
\draw[step=1.0,black,thin, fill=gray] (0,0) grid (6,1); 
\node[above, scale=\vS] at (3,1) {$\bm{\mu}$};
\node[below, scale=\vS] at (3,0) {$\bm{\lambda}$};
\node at (2,1) {x}; 
\node at (3,0) {x}; 
\draw[fill=white] (0,0.5) circle (0.1);
\draw[fill=blue] (6,0.5) circle (0.1);
\end{tikzpicture}\ \ \  
\begin{tikzpicture}[baseline=(current bounding box.center)]
\def\vS{1.8}
\draw[fill=violet] (0,0) rectangle (6,1);
\draw[step=1.0,black,thin] (0,0) grid (6,1); 
\node[below, scale=\vS] at (3,0) {$\bm{\mu}'$};
\node[above, scale=\vS] at (3,1) {$\bm{\lambda}'$};
\node at (3,1) {x}; 
\node at (3,0) {x}; 
\draw[fill=white] (0,0.5) circle (0.1);
\draw[fill=white] (6,0.5) circle (0.1);
\end{tikzpicture}\ \ \
\begin{tikzpicture}[baseline=(current bounding box.center)]
\def\vS{1.8}
\draw[fill=pink] (0,0) rectangle (6,1);
\draw[step=1.0,black,thin] (0,0) grid (6,1); 
\node[above, scale=\vS] at (3,1) {$\bm{\mu}'$};
\node[below, scale=\vS] at (3,0) {$\bm{\lambda}'$};
\node at (4,1) {x}; 
\node at (3,0) {x}; 
\draw[fill=blue] (0,0.5) circle (0.1);
\draw[fill=white] (6,0.5) circle (0.1);
\end{tikzpicture}
}
\]
Here a $k$-tuple of partitions on the boundary of the array means that, for all $i \in [k]$, the $i$-th component of the border condition is given by the Maya diagram corresponding to the $i$-th partition.  It is easy to see that each row has no valid configurations unless $\bm{\mu} \preceq \bm{\lambda}$, in which case it has one valid configuration with weight $x^{|\bm{\lambda}| - |\bm{\mu}|}$ when $t=1$.  Thus the weight of each row is 
\[ \left \{ \begin{array}{ll} x^{|\bm{\lambda}|-|\bm{\mu}|} & \text{if $\bm{\mu} \preceq \bm{\lambda}$} \\ 0 & \text{otherwise} \end{array} \right.  \]
when $t=1$.

\section{$k$-tilings and vertex models}
\label{kcolors}

\subsection{The purple-gray partition function}

Consider the following lattice and its associated partition function.

\[
\resizebox{0.4\textwidth}{!}{
\begin{tikzpicture}[baseline = (current bounding box).center]
\def\vS{1.8}
\draw[fill=gray] (0,0) rectangle (5,3);
\draw[step=1.0] (0,0) grid (5,3);
\draw[fill=violet] (0,3) rectangle (5,6);
\draw[step=1.0] (0,3) grid (5,6);
\draw[fill=black] (5,0.5) circle (0.1);
\draw[fill=black] (5,1.5) circle (0.1);
\draw[fill=black] (5,2.5) circle (0.1);
\draw[fill=white] (5,3.5) circle (0.1);
\draw[fill=white] (5,4.5) circle (0.1);
\draw[fill=white] (5,5.5) circle (0.1);
\draw[fill=white] (0,0.5) circle (0.1);
\draw[fill=white] (0,1.5) circle (0.1);
\draw[fill=white] (0,2.5) circle (0.1);
\draw[fill=white] (0,3.5) circle (0.1);
\draw[fill=white] (0,4.5) circle (0.1);
\draw[fill=white] (0,5.5) circle (0.1);
\draw[fill=black] (0.5,0) circle (0.1);
\draw[fill=black] (1.5,0) circle (0.1);
\draw[fill=black] (2.5,0) circle (0.1);
\draw[fill=white] (0.5,6) circle (0.1);
\draw[fill=white] (1.5,6) circle (0.1);
\draw[fill=white] (2.5,6) circle (0.1);
\draw[fill=white] (3.5,6) circle (0.1);
\draw[fill=white] (4.5,6) circle (0.1);
\draw[fill=white] (3.5,0) circle (0.1);
\draw[fill=white] (4.5,0) circle (0.1);
\node[left,scale=\vS] at (0,0.5) {$y_1$};
\node[left,scale=\vS] at (0,1.5) {$\vdots$};
\node[left,scale=\vS] at (0,2.5) {$y_m$};
\node[left,scale=\vS] at (0,3.5) {$x_1$};
\node[left,scale=\vS] at (0,4.5) {$\vdots$};
\node[left,scale=\vS] at (0,5.5) {$x_m$};
\node[below,scale=\vS] at (1.5,0) {$\leftarrow \;\; m \;\; \rightarrow$};
\end{tikzpicture}
=
\begin{tikzpicture}[baseline = (current bounding box).center]
\def\vS{1.8}
\draw[fill=gray] (0,0) rectangle (5,3);
\draw[step=1.0] (0,0) grid (5,3);
\draw[fill=violet] (0,3) rectangle (5,6);
\draw[step=1.0] (0,3) grid (5,6);
\draw[fill=black] (5,0.5) circle (0.1);
\draw[fill=black] (5,1.5) circle (0.1);
\draw[fill=black] (5,2.5) circle (0.1);
\draw[fill=white] (5,3.5) circle (0.1);
\draw[fill=white] (5,4.5) circle (0.1);
\draw[fill=white] (5,5.5) circle (0.1);
\draw[fill=white] (0,0.5) circle (0.1);
\draw[fill=white] (0,1.5) circle (0.1);
\draw[fill=white] (0,2.5) circle (0.1);
\draw[fill=white] (0,3.5) circle (0.1);
\draw[fill=white] (0,4.5) circle (0.1);
\draw[fill=white] (0,5.5) circle (0.1);
\draw[fill=black] (0.5,0) circle (0.1);
\draw[fill=black] (1.5,0) circle (0.1);
\draw[fill=black] (2.5,0) circle (0.1);
\draw[fill=white] (0.5,6) circle (0.1);
\draw[fill=white] (1.5,6) circle (0.1);
\draw[fill=white] (2.5,6) circle (0.1);
\draw[fill=white] (3.5,6) circle (0.1);
\draw[fill=white] (4.5,6) circle (0.1);
\draw[fill=white] (3.5,0) circle (0.1);
\draw[fill=white] (4.5,0) circle (0.1);
\node[left,scale=\vS] at (0,0.5) {$y_1$};
\node[left,scale=\vS] at (0,1.5) {$\vdots$};
\node[left,scale=\vS] at (0,2.5) {$y_m$};
\node[left,scale=\vS] at (0,3.5) {$x_1$};
\node[left,scale=\vS] at (0,4.5) {$ \vdots$};
\node[left,scale=\vS] at (0,5.5) {$ x_m$};
\draw[ultra thick] (0.5,0)--(0.5,2.5)--(5,2.5);
\draw[ultra thick] (1.5,0)--(1.5,1.5)--(5,1.5);
\draw[ultra thick] (2.5,0)--(2.5,0.5)--(5,0.5);
\end{tikzpicture}
}
= (y^{\rho_m})^k t^{\binom{m}{2}\binom{k}{2}}
\]
Here $y^{\rho_m}=y_1^{m-1}y_2^{m-2}\ldots y_m^{m-m}$, a white dot indicates the absence of all colors, and a black dot indicates the presence of all colors.

By inserting a yellow cross on the left we can use the Yang Baxter equation (Proposition \ref{prop:YBEpurplegray}) to get
\[
\resizebox{0.25\textwidth}{!}{
$
\begin{tikzpicture}[baseline = (current bounding box).center]
\def\vS{1.8}
\draw[fill=gray] (0,0) rectangle (5,3);
\draw[step=1.0] (0,0) grid (5,3);
\draw[fill=violet] (0,3) rectangle (5,6);
\draw[step=1.0] (0,3) grid (5,6);
\draw[fill=black] (5,0.5) circle (0.1);
\draw[fill=black] (5,1.5) circle (0.1);
\draw[fill=black] (5,2.5) circle (0.1);
\draw[fill=white] (5,3.5) circle (0.1);
\draw[fill=white] (5,4.5) circle (0.1);
\draw[fill=white] (5,5.5) circle (0.1);
\draw[fill=white] (0,0.5) circle (0.1);
\draw[fill=white] (0,1.5) circle (0.1);
\draw[fill=white] (0,4.5) circle (0.1);
\draw[fill=white] (0,5.5) circle (0.1);
\draw[fill=black] (0.5,0) circle (0.1);
\draw[fill=black] (1.5,0) circle (0.1);
\draw[fill=black] (2.5,0) circle (0.1);
\draw[fill=white] (0.5,6) circle (0.1);
\draw[fill=white] (1.5,6) circle (0.1);
\draw[fill=white] (2.5,6) circle (0.1);
\draw[fill=white] (3.5,6) circle (0.1);
\draw[fill=white] (4.5,6) circle (0.1);
\draw[fill=white] (3.5,0) circle (0.1);
\draw[fill=white] (4.5,0) circle (0.1);

\node[left,scale=\vS] at (0,0.5) {$y_1$};
\node[left,scale=\vS] at (0,1.5) {$\vdots$};
\node[left,scale=\vS] at (0,4.5) {$\vdots$};
\node[left,scale=\vS] at (0,5.5) {$x_m$};
\node[below,scale=\vS] at (1.5,0) {$\leftarrow \;\; m \;\; \rightarrow$};
\draw[yellow,fill=yellow] (-1,2.5)--(-1,3.5)--(0,3.5)--(0,2.5)--cycle;
\draw (-1,2.5)--(0,3.5); \draw (-1,3.5)--(0,2.5);
\draw[fill=white] (-1,2.5) circle (0.1);
\draw[fill=white] (-1,3.5) circle (0.1);
\node[left,scale=\vS] at (-1,3.5) {$y_m$};
\node[left,scale=\vS] at (-1,2.5) {$x_1$};
\end{tikzpicture}
$
}
= 
\resizebox{0.25\textwidth}{!}{
$
\begin{tikzpicture}[baseline = (current bounding box).center]
\def\vS{1.8}
\draw[fill=gray] (0,0) rectangle (5,2);
\draw[fill=violet] (0,2) rectangle (5,3);
\draw[step=1.0] (0,0) grid (5,3);
\draw[fill=violet] (0,4) rectangle (5,6);
\draw[fill=gray] (0,3) rectangle (5,4);
\draw[step=1.0] (0,3) grid (5,6);
\draw[fill=black] (5,0.5) circle (0.1);
\draw[fill=black] (5,1.5) circle (0.1);
\draw[fill=white] (5,4.5) circle (0.1);
\draw[fill=white] (5,5.5) circle (0.1);
\draw[fill=white] (0,0.5) circle (0.1);
\draw[fill=white] (0,1.5) circle (0.1);
\draw[fill=white] (0,2.5) circle (0.1);
\draw[fill=white] (0,3.5) circle (0.1);
\draw[fill=white] (0,4.5) circle (0.1);
\draw[fill=white] (0,5.5) circle (0.1);
\draw[fill=black] (0.5,0) circle (0.1);
\draw[fill=black] (1.5,0) circle (0.1);
\draw[fill=black] (2.5,0) circle (0.1);
\draw[fill=white] (0.5,6) circle (0.1);
\draw[fill=white] (1.5,6) circle (0.1);
\draw[fill=white] (2.5,6) circle (0.1);
\draw[fill=white] (3.5,6) circle (0.1);
\draw[fill=white] (4.5,6) circle (0.1);
\draw[fill=white] (3.5,0) circle (0.1);
\draw[fill=white] (4.5,0) circle (0.1);
\node[left,scale=\vS] at (0,0.5) {$y_1$};
\node[left,scale=\vS] at (0,1.5) {$\vdots$};
\node[left,scale=\vS] at (0,3.5) {$y_m$};
\node[left,scale=\vS] at (0,2.5) {$x_1$};
\node[left,scale=\vS] at (0,4.5) {$\vdots$};
\node[left,scale=\vS] at (0,5.5) {$x_m$};
\node[below,scale=\vS] at (1.5,0) {$\leftarrow \;\; m \;\; \rightarrow$};
\draw[yellow,fill=yellow] (5,2.5)--(5,3.5)--(6,3.5)--(6,2.5)--cycle;
\draw (5,2.5)--(6,3.5); \draw (5,3.5)--(6,2.5);
\draw[fill=black] (6,2.5) circle (0.1);
\draw[fill=white] (6,3.5) circle (0.1);
\end{tikzpicture}
$
}
\]
from which we see that
\[
\resizebox{0.2\textwidth}{!}{
$
\begin{tikzpicture}[baseline = (current bounding box).center]
\def\vS{1.8}
\draw[fill=gray] (0,0) rectangle (5,3);
\draw[step=1.0] (0,0) grid (5,3);
\draw[fill=violet] (0,3) rectangle (5,6);
\draw[step=1.0] (0,3) grid (5,6);
\draw[fill=black] (5,0.5) circle (0.1);
\draw[fill=black] (5,1.5) circle (0.1);
\draw[fill=black] (5,2.5) circle (0.1);
\draw[fill=white] (5,3.5) circle (0.1);
\draw[fill=white] (5,4.5) circle (0.1);
\draw[fill=white] (5,5.5) circle (0.1);
\draw[fill=white] (0,0.5) circle (0.1);
\draw[fill=white] (0,1.5) circle (0.1);
\draw[fill=white] (0,2.5) circle (0.1);
\draw[fill=white] (0,3.5) circle (0.1);
\draw[fill=white] (0,4.5) circle (0.1);
\draw[fill=white] (0,5.5) circle (0.1);
\draw[fill=black] (0.5,0) circle (0.1);
\draw[fill=black] (1.5,0) circle (0.1);
\draw[fill=black] (2.5,0) circle (0.1);
\draw[fill=white] (0.5,6) circle (0.1);
\draw[fill=white] (1.5,6) circle (0.1);
\draw[fill=white] (2.5,6) circle (0.1);
\draw[fill=white] (3.5,6) circle (0.1);
\draw[fill=white] (4.5,6) circle (0.1);
\draw[fill=white] (3.5,0) circle (0.1);
\draw[fill=white] (4.5,0) circle (0.1);

\node[left,scale=\vS] at (0,0.5) {$y_1$};
\node[left,scale=\vS] at (0,1.5) {$\vdots$};
\node[left,scale=\vS] at (0,2.5) {$y_m$};
\node[left,scale=\vS] at (0,3.5) {$x_1$};
\node[left,scale=\vS] at (0,4.5) {$\vdots$};
\node[left,scale=\vS] at (0,5.5) {$x_m$};
\node[below,scale=\vS] at (1.5,0) {$\leftarrow \;\; m \;\; \rightarrow$};
\end{tikzpicture}
$
} 
=
\resizebox{0.2\textwidth}{!}{
$
\begin{tikzpicture}[baseline = (current bounding box).center]
\def\vS{1.8}
\draw[fill=gray] (0,0) rectangle (5,2);
\draw[fill=violet] (0,2) rectangle (5,3);
\draw[step=1.0] (0,0) grid (5,3);
\draw[fill=violet] (0,4) rectangle (5,6);
\draw[fill=gray] (0,3) rectangle (5,4);
\draw[step=1.0] (0,3) grid (5,6);
\draw[fill=black] (5,0.5) circle (0.1);
\draw[fill=black] (5,1.5) circle (0.1);
\draw[fill=white] (5,2.5) circle (0.1);
\draw[fill=black] (5,3.5) circle (0.1);
\draw[fill=white] (5,4.5) circle (0.1);
\draw[fill=white] (5,5.5) circle (0.1);
\draw[fill=white] (0,0.5) circle (0.1);
\draw[fill=white] (0,1.5) circle (0.1);
\draw[fill=white] (0,2.5) circle (0.1);
\draw[fill=white] (0,3.5) circle (0.1);
\draw[fill=white] (0,4.5) circle (0.1);
\draw[fill=white] (0,5.5) circle (0.1);
\draw[fill=black] (0.5,0) circle (0.1);
\draw[fill=black] (1.5,0) circle (0.1);
\draw[fill=black] (2.5,0) circle (0.1);
\draw[fill=white] (0.5,6) circle (0.1);
\draw[fill=white] (1.5,6) circle (0.1);
\draw[fill=white] (2.5,6) circle (0.1);
\draw[fill=white] (3.5,6) circle (0.1);
\draw[fill=white] (4.5,6) circle (0.1);
\draw[fill=white] (3.5,0) circle (0.1);
\draw[fill=white] (4.5,0) circle (0.1);
\node[left,scale=\vS] at (0,0.5) {$y_1$};
\node[left,scale=\vS] at (0,1.5) {$\vdots$};
\node[left,scale=\vS] at (0,3.5) {$y_m$};
\node[left,scale=\vS] at (0,2.5) {$x_1$};
\node[left,scale=\vS] at (0,4.5) {$\vdots$};
\node[left,scale=\vS] at (0,5.5) {$x_m$};
\node[below,scale=\vS] at (1.5,0) {$\leftarrow \;\; m \;\; \rightarrow$};
\end{tikzpicture}
$
}
\times \prod_{l=0}^{k-1}\left(1+x_1y_mt^l\right)^{-1}.
\]
We can repeat this to get
\[
\resizebox{0.2\textwidth}{!}{
$
\begin{tikzpicture}[baseline = (current bounding box).center]
\def\vS{1.8}
\draw[fill=gray] (0,0) rectangle (5,3);
\draw[step=1.0] (0,0) grid (5,3);
\draw[fill=violet] (0,3) rectangle (5,6);
\draw[step=1.0] (0,3) grid (5,6);
\draw[fill=black] (5,0.5) circle (0.1);
\draw[fill=black] (5,1.5) circle (0.1);
\draw[fill=black] (5,2.5) circle (0.1);
\draw[fill=white] (5,3.5) circle (0.1);
\draw[fill=white] (5,4.5) circle (0.1);
\draw[fill=white] (5,5.5) circle (0.1);
\draw[fill=white] (0,0.5) circle (0.1);
\draw[fill=white] (0,1.5) circle (0.1);
\draw[fill=white] (0,2.5) circle (0.1);
\draw[fill=white] (0,3.5) circle (0.1);
\draw[fill=white] (0,4.5) circle (0.1);
\draw[fill=white] (0,5.5) circle (0.1);
\draw[fill=black] (0.5,0) circle (0.1);
\draw[fill=black] (1.5,0) circle (0.1);
\draw[fill=black] (2.5,0) circle (0.1);
\draw[fill=white] (0.5,6) circle (0.1);
\draw[fill=white] (1.5,6) circle (0.1);
\draw[fill=white] (2.5,6) circle (0.1);
\draw[fill=white] (3.5,6) circle (0.1);
\draw[fill=white] (4.5,6) circle (0.1);
\draw[fill=white] (3.5,0) circle (0.1);
\draw[fill=white] (4.5,0) circle (0.1);
\node[left,scale=\vS] at (0,0.5) {$y_1$};
\node[left,scale=\vS] at (0,1.5) {$\vdots$};
\node[left,scale=\vS] at (0,2.5) {$y_m$};
\node[left,scale=\vS] at (0,3.5) {$x_1$};
\node[left,scale=\vS] at (0,4.5) {$\vdots$};
\node[left,scale=\vS] at (0,5.5) {$x_m$};
\node[below,scale=\vS] at (1.5,0) {$\leftarrow \;\; m \;\; \rightarrow$};
\end{tikzpicture}
$
} 
=
\resizebox{0.2\textwidth}{!}{
$
\begin{tikzpicture}[baseline = (current bounding box).center]
\def\vS{1.8}
\draw[fill=violet] (0,0) rectangle (5,1);
\draw[fill=gray] (0,1) rectangle (5,2);
\draw[fill=violet] (0,2) rectangle (5,3);
\draw[fill=gray] (0,3) rectangle (5,4);
\draw[fill=violet] (0,4) rectangle (5,5);
\draw[fill=gray] (0,5) rectangle (5,6);
\draw[step=1.0] (0,0) grid (5,3);
\draw[step=1.0] (0,3) grid (5,6);
\draw[fill=white] (5,0.5) circle (0.1);
\draw[fill=black] (5,1.5) circle (0.1);
\draw[fill=white] (5,2.5) circle (0.1);
\draw[fill=black] (5,3.5) circle (0.1);
\draw[fill=white] (5,4.5) circle (0.1);
\draw[fill=black] (5,5.5) circle (0.1);
\draw[fill=white] (0,0.5) circle (0.1);
\draw[fill=white] (0,1.5) circle (0.1);
\draw[fill=white] (0,2.5) circle (0.1);
\draw[fill=white] (0,3.5) circle (0.1);
\draw[fill=white] (0,4.5) circle (0.1);
\draw[fill=white] (0,5.5) circle (0.1);
\draw[fill=black] (0.5,0) circle (0.1);
\draw[fill=black] (1.5,0) circle (0.1);
\draw[fill=black] (2.5,0) circle (0.1);
\draw[fill=white] (0.5,6) circle (0.1);
\draw[fill=white] (1.5,6) circle (0.1);
\draw[fill=white] (2.5,6) circle (0.1);
\draw[fill=white] (3.5,6) circle (0.1);
\draw[fill=white] (4.5,6) circle (0.1);
\draw[fill=white] (3.5,0) circle (0.1);
\draw[fill=white] (4.5,0) circle (0.1);

\node[left,scale=\vS] at (0,0.5) {$x_1$};
\node[left,scale=\vS] at (0,1.5) {$y_1$};
\node[left,scale=\vS] at (0,3.5) {$\vdots$};
\node[left,scale=\vS] at (0,2.5) {$\vdots$};
\node[left,scale=\vS] at (0,4.5) {$x_m$};
\node[left,scale=\vS] at (0,5.5) {$y_m$};
\node[below,scale=\vS] at (1.5,0) {$\leftarrow \;\; m \;\; \rightarrow$};
\end{tikzpicture}
$
}
\times \prod_{l=0}^{k-1}\prod_{i\le j}\left(1+x_iy_jt^l\right)^{-1}
\]
from which we see that
\[
\resizebox{0.2\textwidth}{!}{
$
\begin{tikzpicture}[baseline = (current bounding box).center]
\def\vS{1.8}
\draw[fill=violet] (0,0) rectangle (5,1);
\draw[fill=gray] (0,1) rectangle (5,2);
\draw[fill=violet] (0,2) rectangle (5,3);
\draw[fill=gray] (0,3) rectangle (5,4);
\draw[fill=violet] (0,4) rectangle (5,5);
\draw[fill=gray] (0,5) rectangle (5,6);
\draw[step=1.0] (0,0) grid (5,3);
\draw[step=1.0] (0,3) grid (5,6);
\draw[fill=white] (5,0.5) circle (0.1);
\draw[fill=black] (5,1.5) circle (0.1);
\draw[fill=white] (5,2.5) circle (0.1);
\draw[fill=black] (5,3.5) circle (0.1);
\draw[fill=white] (5,4.5) circle (0.1);
\draw[fill=black] (5,5.5) circle (0.1);
\draw[fill=white] (0,0.5) circle (0.1);
\draw[fill=white] (0,1.5) circle (0.1);
\draw[fill=white] (0,2.5) circle (0.1);
\draw[fill=white] (0,3.5) circle (0.1);
\draw[fill=white] (0,4.5) circle (0.1);
\draw[fill=white] (0,5.5) circle (0.1);
\draw[fill=black] (0.5,0) circle (0.1);
\draw[fill=black] (1.5,0) circle (0.1);
\draw[fill=black] (2.5,0) circle (0.1);
\draw[fill=white] (0.5,6) circle (0.1);
\draw[fill=white] (1.5,6) circle (0.1);
\draw[fill=white] (2.5,6) circle (0.1);
\draw[fill=white] (3.5,6) circle (0.1);
\draw[fill=white] (4.5,6) circle (0.1);
\draw[fill=white] (3.5,0) circle (0.1);
\draw[fill=white] (4.5,0) circle (0.1);

\node[left,scale=\vS] at (0,0.5) {$x_1$};
\node[left,scale=\vS] at (0,1.5) {$y_1$};
\node[left,scale=\vS] at (0,3.5) {$\vdots$};
\node[left,scale=\vS] at (0,2.5) {$\vdots$};
\node[left,scale=\vS] at (0,4.5) {$x_m$};
\node[left,scale=\vS] at (0,5.5) {$y_m$};
\node[below,scale=\vS] at (1.5,0) {$\leftarrow \;\; m \;\; \rightarrow$};
\end{tikzpicture}
$
} 
=  (y^{\rho_m})^k t^{\binom{m}{2}\binom{k}{2}} \prod_{l=0}^{k-1}\prod_{i\le j}\left(1+x_iy_jt^l\right).
\]

Given a configuration of the lattice
\[
\resizebox{0.2\textwidth}{!}{
$
\begin{tikzpicture}[baseline = (current bounding box).center]
\def\vS{1.8}
\draw[fill=violet] (0,0) rectangle (5,1);
\draw[fill=gray] (0,1) rectangle (5,2);
\draw[fill=violet] (0,2) rectangle (5,3);
\draw[fill=gray] (0,3) rectangle (5,4);
\draw[fill=violet] (0,4) rectangle (5,5);
\draw[fill=gray] (0,5) rectangle (5,6);
\draw[step=1.0] (0,0) grid (5,3);
\draw[step=1.0] (0,3) grid (5,6);
\draw[fill=white] (5,0.5) circle (0.1);
\draw[fill=black] (5,1.5) circle (0.1);
\draw[fill=white] (5,2.5) circle (0.1);
\draw[fill=black] (5,3.5) circle (0.1);
\draw[fill=white] (5,4.5) circle (0.1);
\draw[fill=black] (5,5.5) circle (0.1);
\draw[fill=white] (0,0.5) circle (0.1);
\draw[fill=white] (0,1.5) circle (0.1);
\draw[fill=white] (0,2.5) circle (0.1);
\draw[fill=white] (0,3.5) circle (0.1);
\draw[fill=white] (0,4.5) circle (0.1);
\draw[fill=white] (0,5.5) circle (0.1);
\draw[fill=black] (0.5,0) circle (0.1);
\draw[fill=black] (1.5,0) circle (0.1);
\draw[fill=black] (2.5,0) circle (0.1);
\draw[fill=white] (0.5,6) circle (0.1);
\draw[fill=white] (1.5,6) circle (0.1);
\draw[fill=white] (2.5,6) circle (0.1);
\draw[fill=white] (3.5,6) circle (0.1);
\draw[fill=white] (4.5,6) circle (0.1);
\draw[fill=white] (3.5,0) circle (0.1);
\draw[fill=white] (4.5,0) circle (0.1);

\node[left,scale=\vS] at (0,0.5) {$x_1$};
\node[left,scale=\vS] at (0,1.5) {$y_1$};
\node[left,scale=\vS] at (0,3.5) {$\vdots$};
\node[left,scale=\vS] at (0,2.5) {$\vdots$};
\node[left,scale=\vS] at (0,4.5) {$x_m$};
\node[left,scale=\vS] at (0,5.5) {$y_m$};
\node[below,scale=\vS] at (1.5,0) {$\leftarrow \;\; m \;\; \rightarrow$};
\node at (3,0) {x}; \node at (3,1) {x}; \node at (2,2) {x}; \node at (2,3) {x}; \node at (1,4) {x}; \node at (1,5) {x}; \node at (0,6) {x};
\end{tikzpicture}
$
} 
\]
and looking at the labels on the horizontal edges row by row from bottom to top, we get a sequence of $2m+1$ $k$-tuples of Maya diagrams, where we mark the 0 content line with x's on the lattice.  The corresponding $2m+1$ $k$-tuples of partitions satisfy
\[ \bm{0} = \bm{\lambda}^0 \preceq '  \bm{\lambda}^1 \succeq \ldots \preceq'  \bm{\lambda}^{2m-1} \succeq  \bm{\lambda}^{2m} = \bm{0}. \]

\subsection{Relating lattice configurations and $k$-tilings in the purple-gray model}

Given a sequence of $k$-tuples of partitions 
\[ \bm{0} = \bm{\lambda}^0 \preceq '  \bm{\lambda}^1 \succeq \ldots \preceq'  \bm{\lambda}^{2m-1} \succeq  \bm{\lambda}^{2m} = \bm{0}, \]
we get both a configuration of the lattice
\[ 
\resizebox{0.2\textwidth}{!}{
$
\begin{tikzpicture}[baseline = (current bounding box).center]
\def\vS{1.8}
\draw[fill=violet] (0,0) rectangle (5,1);
\draw[fill=gray] (0,1) rectangle (5,2);
\draw[fill=violet] (0,2) rectangle (5,3);
\draw[fill=gray] (0,3) rectangle (5,4);
\draw[fill=violet] (0,4) rectangle (5,5);
\draw[fill=gray] (0,5) rectangle (5,6);
\draw[step=1.0] (0,0) grid (5,3);
\draw[step=1.0] (0,3) grid (5,6);
\draw[fill=white] (5,0.5) circle (0.1);
\draw[fill=black] (5,1.5) circle (0.1);
\draw[fill=white] (5,2.5) circle (0.1);
\draw[fill=black] (5,3.5) circle (0.1);
\draw[fill=white] (5,4.5) circle (0.1);
\draw[fill=black] (5,5.5) circle (0.1);
\draw[fill=white] (0,0.5) circle (0.1);
\draw[fill=white] (0,1.5) circle (0.1);
\draw[fill=white] (0,2.5) circle (0.1);
\draw[fill=white] (0,3.5) circle (0.1);
\draw[fill=white] (0,4.5) circle (0.1);
\draw[fill=white] (0,5.5) circle (0.1);
\draw[fill=black] (0.5,0) circle (0.1);
\draw[fill=black] (1.5,0) circle (0.1);
\draw[fill=black] (2.5,0) circle (0.1);
\node[left,scale=\vS] at (0,0.5) {$x_1$};
\node[left,scale=\vS] at (0,1.5) {$y_1$};
\node[left,scale=\vS] at (0,3.5) {$\vdots$};
\node[left,scale=\vS] at (0,2.5) {$\vdots$};
\node[left,scale=\vS] at (0,4.5) {$x_m$};
\node[left,scale=\vS] at (0,5.5) {$y_m$};
\node[below,scale=\vS] at (1.5,0) {$\leftarrow \;\; m \;\; \rightarrow$};
\node at (3,0) {x}; \node at (3,1) {x}; \node at (2,2) {x}; \node at (2,3) {x}; \node at (1,4) {x}; \node at (1,5) {x}; \node at (0,6) {x};
\draw[fill=black] (0.5,0) circle (0.1);
\draw[fill=black] (1.5,0) circle (0.1);
\draw[fill=black] (2.5,0) circle (0.1);
\draw[fill=white] (3.5,0) circle (0.1);
\draw[fill=white] (4.5,0) circle (0.1);
\draw[fill=white] (0.5,6) circle (0.1);
\draw[fill=white] (1.5,6) circle (0.1);
\draw[fill=white] (2.5,6) circle (0.1);
\draw[fill=white] (3.5,6) circle (0.1);
\draw[fill=white] (4.5,6) circle (0.1);
\end{tikzpicture}
$
} 
\]
and a $k$-tiling of the Aztec diamond of rank $m$.  For example, in Figure \ref{exvertex}, the tiling on the left corresponds to the configuration on the right, and in Figure \ref{grpu}, we give the 8 configurations corresponding to the tilings of the rank 2 Aztec diamond, which were listed in Figure \ref{rank2}.  It turns out that the weight of the lattice configuration and the weight of the $k$-tiling are related.

\begin{figure}[ht]
\begin{center}
\resizebox{0.5\textwidth}{!}{
\begin{tikzpicture}[baseline = (current bounding box).center]
\checkerboard{2}
\draw[ultra thick] (0,-1) rectangle (1,1); \draw[ultra thick] (1,-2) rectangle (2,0); \draw[ultra thick] (1,0) rectangle (2,2);
\draw[ultra thick] (2,-3) rectangle (4,-2); \draw[ultra thick] (2,-1) rectangle (4,0);
\draw[ultra thick] (2,2) rectangle (4,3);
\draw[ultra thick] (2,0) rectangle (3,2); \draw[ultra thick] (2,-2) rectangle (4,-1); \draw[ultra thick] (3,0) rectangle (4,2); 
\draw[ultra thick] (4,-2) rectangle (5,0); \draw[ultra thick] (4,0) rectangle (5,2); 
\draw[ultra thick] (5,-1) rectangle (6,1);       
\draw[ultra thick, blue, fill=white] (0.5,0.5) circle (5pt);
\draw[ultra thick, blue, fill=white] (0.5,-.5) circle (5pt);
\draw[ultra thick, blue, fill=white] (1.5,-1.5) circle (5pt);
\draw[ultra thick, blue, fill=white] (1.5,-.5) circle (5pt);
\draw[ultra thick, blue, fill=white] (1.5,1.5) circle (5pt);
\draw[ultra thick, blue, fill=white] (1.5,.5) circle (5pt);
\draw[ultra thick, blue, fill=white] (3.5,1.5) circle (5pt);
\draw[ultra thick, blue, fill=white] (3.5,.5) circle (5pt);
\draw[ultra thick, blue, fill=white] (3.5,2.5) circle (5pt);\draw[ultra thick, blue, fill=white] (2.5,2.5) circle (5pt);
\draw[ultra thick, blue, fill=white] (3.5,-1.5) circle (5pt);
\draw[ultra thick, blue, fill=white] (2.5,-1.5) circle (5pt);
\draw[ultra thick, blue, fill=blue] (3.5,-2.5) circle (5pt);
\draw[ultra thick, blue, fill=blue] (2.5,-2.5) circle (5pt);
\draw[ultra thick, blue, fill=blue] (3.5,-.5) circle (5pt);
\draw[ultra thick, blue, fill=blue] (2.5,-.5) circle (5pt);
\draw[ultra thick, blue, fill=blue] (4.5,-.5) circle (5pt);
\draw[ultra thick, blue, fill=blue] (4.5,-1.5) circle (5pt);
\draw[ultra thick, blue, fill=blue] (4.5,.5) circle (5pt);
\draw[ultra thick, blue, fill=blue] (4.5,1.5) circle (5pt);

\draw[ultra thick, blue, fill=blue] (2.5,.5) circle (5pt);
\draw[ultra thick, blue, fill=blue] (2.5,1.5) circle (5pt);
\draw[ultra thick, blue, fill=blue] (5.5,-.5) circle (5pt);
\draw[ultra thick, blue, fill=blue] (5.5,.5) circle (5pt);

\end{tikzpicture}\hspace{1cm}
\begin{tikzpicture}[baseline = (current bounding box).center]
\vertex{3}
\draw[ultra thick, blue](0.5,0)--(0.5,1.5)--(1.5,1.5)--(1.5,4.5)--(2.5,4.5)--(2.5,5.5)--(5,5.5);
\draw[ultra thick, blue](1.5,0)--(1.5,.5)--(3.5,.5)--(3.5,1.5)--(5,1.5);
\draw[ultra thick, blue](2.5,0)--(2.5,2.5)--(3.5,2.5)--(3.5,3.5)--(5,3.5);

\end{tikzpicture}}
\end{center}
\caption{Domino tilings and
vertex models}
\label{exvertex}
\end{figure}
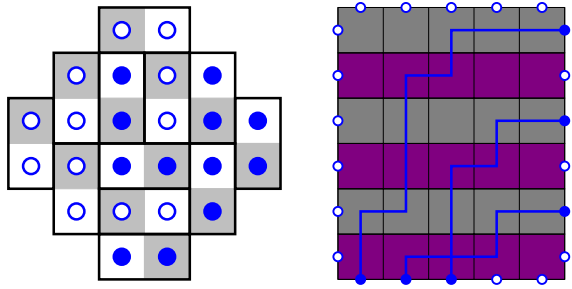

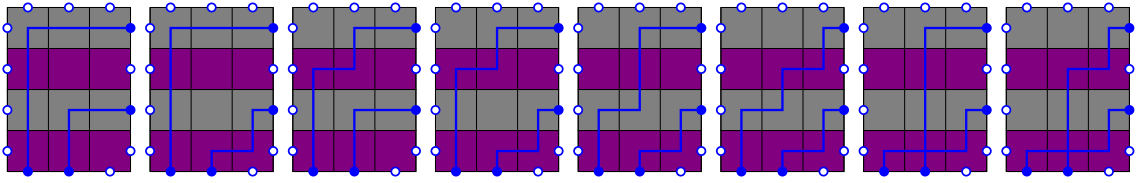
\begin{figure}[ht]
\resizebox{\textwidth}{!}{
\begin{tikzpicture}\vertex{2}
\draw[ultra thick, blue] (0.5,0)--(0.5,3.5)--(3,3.5);
\draw[ultra thick, blue] (1.5,0)--(1.5,1.5)--(3,1.5);
\end{tikzpicture}\ \ 
\begin{tikzpicture}
\vertex{2}
\draw[ultra thick, blue] (0.5,0)--(0.5,3.5)--(3,3.5);
\draw[ultra thick, blue] (1.5,0)--(1.5,0.5)--(2.5,0.5)--(2.5,1.5)--(3,1.5);
\end{tikzpicture}\ \ 
\begin{tikzpicture}\vertex{2}
\draw[ultra thick, blue] (0.5,0)--(0.5,2.5)--(1.5,2.5)--(1.5,3.5)--(3,3.5);
\draw[ultra thick, blue] (1.5,0)--(1.5,1.5)--(3,1.5);
\end{tikzpicture}\ \ 
\begin{tikzpicture}\vertex{2}
\draw[ultra thick, blue] (0.5,0)--(0.5,2.5)--(1.5,2.5)--(1.5,3.5)--(3,3.5);
\draw[ultra thick, blue]  (1.5,0)--(1.5,0.5)--(2.5,0.5)--(2.5,1.5)--(3,1.5);
\end{tikzpicture}\ \ 
\begin{tikzpicture}\vertex{2}
\draw[ultra thick, blue] (0.5,0)--(0.5,1.5)--(1.5,1.5)--(1.5,3.5)--(3,3.5);
\draw[ultra thick, blue]  (1.5,0)--(1.5,0.5)--(2.5,0.5)--(2.5,1.5)--(3,1.5);
\end{tikzpicture}\ \ 

\begin{tikzpicture}\vertex{2}
\draw[ultra thick, blue] (0.5,0)--(0.5,1.5)--(1.5,1.5)--(1.5,2.5)--(2.5,2.5)--(2.5,3.5)--(3,3.5);
\draw[ultra thick, blue]  (1.5,0)--(1.5,0.5)--(2.5,0.5)--(2.5,1.5)--(3,1.5);
\end{tikzpicture}\ \ 

\begin{tikzpicture}\vertex{2}
\draw[ultra thick, blue] (0.5,0)--(0.5,0.5)--(1.5,0.5)--(1.5,3.5)--(3,3.5);
\draw[ultra thick, blue]  (1.5,0)--(1.5,0.5)--(2.5,0.5)--(2.5,1.5)--(3,1.5);
\end{tikzpicture}\ \ 

\begin{tikzpicture}\vertex{2}
\draw[ultra thick, blue] (0.5,0)--(0.5,0.5)--(1.5,0.5)--(1.5,2.5)--(2.5,2.5)--(2.5,3.5)--(3,3.5);
\draw[ultra thick, blue]  (1.5,0)--(1.5,0.5)--(2.5,0.5)--(2.5,1.5)--(3,1.5);
\end{tikzpicture}}

\caption{Purple-gray lattice configurations for the Aztec diamond of rank 2}
\label{grpu}
\end{figure}

For the purple faces, one gets a $t$ whenever you have a face of the form
\[
\begin{tikzpicture}[baseline = (current bounding box).center]
\draw[fill=violet] (0,0) rectangle (1,1);
\draw[very thick, blue] (0.5,0.6)--(1,0.6);
\draw[very thick, red] (0.6,0)--(0.6,1); 
\end{tikzpicture}
\]
where blue is a smaller color than red. It is easy to see that this equals the number of domino configurations of the form
\[ \resizebox{1.0cm}{!}{
\begin{tikzpicture}[baseline = (current bounding box).center]
\draw[lightgray] (0,1) rectangle (1,2);
\draw[lightgray,fill=lightgray] (0,0) rectangle (1,1);
\draw[lightgray] (-1,0) rectangle (0,1);
\draw[line width=1mm, blue] (0,0) rectangle (1,2);
\draw[line width=1mm, red] (-1.01,-0.01) rectangle (0.99,0.99);
\end{tikzpicture}. } \]
which is one of the configurations that give a $t$.  One gets an $x_i$ whenever a path exits right in the $i$-th purple row. It is easy to see that this equals the number of dominos of the form
\[ \resizebox{.5cm}{!}{
\begin{tikzpicture}[baseline = (current bounding box).center]
\draw[lightgray, fill=lightgray] (0,0) rectangle (1,1);
\draw[lightgray] (0,1) rectangle (1,2);
\draw[fill=black] (0.5,0.5) circle (5pt);
\draw[fill=black] (0.5,1.5) circle (5pt);
\draw[very thick, blue] (0,0) rectangle (1,2);
\end{tikzpicture}}
\]
whose top square is on slice $2i-1$, which equals the $x_i$ power we give the dominos.

We are left to consider the gray faces.  Let's look at the $(m-p+1)$-th gray row
\[\resizebox{4cm}{!}{
\begin{tikzpicture}[baseline = (current bounding box)4center]
\draw[fill=gray] (0,0) rectangle (5,1);
\draw[step=1.0] (0,0) grid (5,1);
\draw[fill=white] (0,0.5) circle (0.1);
\draw[fill=black] (5,0.5) circle (0.1);
\node[below] at (1.5,0) {$\longleftarrow \;\; p \;\; \longrightarrow$};
\node at (3,0) {x}; \node at (2,1) {x};
\end{tikzpicture}}
\]
for some $p \in [m]$, where the row has total length $M$.  Let $y = y_{m-p+1}$.
For a single color $a \in [k]$, we can write the $y$-weight for the row as
\[ y^{\left(
\# \resizebox{0.75cm}{!}{\begin{tikzpicture}[baseline=(current bounding box.center)]\draw[thin,fill=gray] (0,0) rectangle (1,1); \end{tikzpicture}}
+ \# \resizebox{0.75cm}{!}{\begin{tikzpicture}[baseline=(current bounding box.center)] \draw[thin,fill=gray] (0,0) rectangle (1,1); \draw[blue, thick] (0,0.5)--(0.5,0.5)--(0.5,1); \end{tikzpicture}}
+ \# \resizebox{0.75cm}{!}{\begin{tikzpicture}[baseline=(current bounding box.center)] \draw[thin,fill=gray] (0,0) rectangle (1,1); \draw[blue, thick] (0.5,0)--(0.5,1); \end{tikzpicture}} \right)}.
\]

\begin{lem}
\[
\# \resizebox{0.75cm}{!}{\begin{tikzpicture}[baseline=(current bounding box.center)] \draw[thin,fill=gray] (0,0) rectangle (1,1); \draw[blue, thick] (0,0.5)--(0.5,0.5)--(0.5,1); \end{tikzpicture}}
+ \# \resizebox{0.75cm}{!}{\begin{tikzpicture}[baseline=(current bounding box.center)] \draw[thin,fill=gray] (0,0) rectangle (1,1); \draw[blue, thick] (0.5,0)--(0.5,1); \end{tikzpicture}}
= p-1
\]
\end{lem}
\begin{proof}
The left-hand side counts the number of paths that exit the row through the top.  There are $p$ paths entering the row through the bottom and 0 paths entering the row from the left, and there is 1 path exiting the row through the right.  By path conservation, this means that there are $p-1$ paths exiting the row through the top.
\end{proof}
\noindent Note that the number of empty vertices equals the number of cells that get removed from the partition, i.e.
\[ \# \resizebox{0.75cm}{!}{\begin{tikzpicture}[baseline=(current bounding box.center)]\draw[thin,fill=gray] (0,0) rectangle (1,1); \end{tikzpicture}} = |\lambda^{(a)}_{2(m-p+1)}| - |\lambda^{(a)}_{2(m-p+1)-1}|. \]
It is easy to see that this equals the number the number of dominos of the form
\[ \resizebox{0.5cm}{!}{ \begin{tikzpicture}[baseline = (current bounding box).center]
\draw[lightgray] (0,0) rectangle (1,1);
\draw[lightgray,fill=lightgray] (0,1) rectangle (1,2);
\draw[very thick, blue] (0,0) rectangle (1,2);
\end{tikzpicture}} \]
whose bottom square is on slice $2(m-p+1)-1$, which equals the $y$ power we give the dominos.  Thus if we pull out a factor of $y^{p-1}$ from the gray row, then the $y$-weights agree.  For the $t$-weight of the row, we can write the $t$ power as
\[ \begin{aligned}
&\sum_{1 \leq a < b \leq k} \# \{ \text{faces} | \text{$a$ right, $b$ top} \} + \# \{ \text{faces} | \text{$a$ not right, $b$ not right} \} \\
&= \sum_{1 \leq a < b \leq k} \# \{ \text{faces} | \text{$a$ right, $b$ top} \} + \# \{ \text{faces} | \text{$a$ not right, $b$ top} \} + \# \{ \text{faces} | \text{$a$ not right, $b$ absent} \} \\
&= \sum_{1 \leq a < b \leq k} \# \{ \text{faces} | \text{$b$ top} \} + \# \{ \text{faces} | \text{$a$ not right, $b$ absent} \} \\
&= \sum_{1 \leq a < b \leq k} (p-1) + \# \{ \text{faces} | \text{$a$ not right, $b$ absent} \} \\
&= \binom{k}{2}(p-1) + \sum_{1 \leq a < b \leq k} \# \{ \text{faces} | \text{$a$ not right, $b$ absent} \}
\end{aligned} \]
where we have used the fact that $\# \{ \text{faces} : \text{$b$ exits top} \} = p-1$ by the previous lemma.  It is easy to see that $\# \{ \text{faces} : \text{$a$ not right, $b$ absent} \}$ equals the number of domino configurations of the form
\[ 
\resizebox{0.5cm}{!}{
\begin{tikzpicture}[baseline = (current bounding box).center]
\draw[lightgray,fill=lightgray] (0,1) rectangle (1,2);
\draw[lightgray] (0,0) rectangle (1,1);
\draw[line width=1mm, blue] (0,0) rectangle (1,2);
\draw[line width=1mm, red] (-0.05,-0.05) rectangle (0.95,1.95);
\end{tikzpicture}}\ ,
\;\;\;
\resizebox{1cm}{!}{
 \begin{tikzpicture}[baseline = (current bounding box).center]
\draw[lightgray,fill=lightgray] (0,1) rectangle (1,2);
\draw[lightgray] (0,0) rectangle (1,1);
\draw[very thick, blue] (-1.01,0.99) rectangle (0.99,1.99);
\draw[very thick, red] (0,0) rectangle (1,2);
\end{tikzpicture}}\ , {\rm or}
\;\;\;
\resizebox{0.5cm}{!}{
 \begin{tikzpicture}[baseline = (current bounding box).center]
\draw[lightgray,fill=lightgray] (0,1) rectangle (1,2);
\draw[lightgray] (0,0) rectangle (1,1);
\draw[very thick, blue] (-0.01,0.99) rectangle (.99,2.99);
\draw[very thick, red] (0,0) rectangle (1,2);
\end{tikzpicture} }
\]
where blue is color $a$ and red is color $b$, which are three of the configurations that give a $t$.  Thus if we pull out a factor of $ t^{\binom{k}{2}(p-1)}$ from the gray row, then the $t$-weights agree.

Putting it all together, we arrive at the following results.

\begin{lem}
There is a weight-preserving bijection between configurations of the purple-gray lattice
\[
\resizebox{0.2\textwidth}{!}{
$
\begin{tikzpicture}[baseline = (current bounding box).center]
\def\vS{1.8}
\draw[fill=violet] (0,0) rectangle (5,1);
\draw[fill=gray] (0,1) rectangle (5,2);
\draw[fill=violet] (0,2) rectangle (5,3);
\draw[fill=gray] (0,3) rectangle (5,4);
\draw[fill=violet] (0,4) rectangle (5,5);
\draw[fill=gray] (0,5) rectangle (5,6);
\draw[step=1.0] (0,0) grid (5,3);
\draw[step=1.0] (0,3) grid (5,6);
\draw[fill=white] (5,0.5) circle (0.1);
\draw[fill=black] (5,1.5) circle (0.1);
\draw[fill=white] (5,2.5) circle (0.1);
\draw[fill=black] (5,3.5) circle (0.1);
\draw[fill=white] (5,4.5) circle (0.1);
\draw[fill=black] (5,5.5) circle (0.1);
\draw[fill=white] (0,0.5) circle (0.1);
\draw[fill=white] (0,1.5) circle (0.1);
\draw[fill=white] (0,2.5) circle (0.1);
\draw[fill=white] (0,3.5) circle (0.1);
\draw[fill=white] (0,4.5) circle (0.1);
\draw[fill=white] (0,5.5) circle (0.1);
\draw[fill=black] (0.5,0) circle (0.1);
\draw[fill=black] (1.5,0) circle (0.1);
\draw[fill=black] (2.5,0) circle (0.1);
\draw[fill=white] (3.5,0) circle (0.1);
\draw[fill=white] (4.5,0) circle (0.1);
\draw[fill=white] (0.5,6) circle (0.1);
\draw[fill=white] (1.5,6) circle (0.1);
\draw[fill=white] (2.5,6) circle (0.1);
\draw[fill=white] (3.5,6) circle (0.1);
\draw[fill=white] (4.5,6) circle (0.1);
\node[left,scale=\vS] at (0,0.5) {$x_1$};
\node[left,scale=\vS] at (0,1.5) {$y_1$};
\node[left,scale=\vS] at (0,3.5) {$\vdots$};
\node[left,scale=\vS] at (0,2.5) {$\vdots$};
\node[left,scale=\vS] at (0,4.5) {$x_m$};
\node[left,scale=\vS] at (0,5.5) {$y_m$};
\node[below,scale=\vS] at (1.5,0) {$\leftarrow \;\; m \;\; \rightarrow$};
\node at (3,0) {x}; \node at (3,1) {x}; \node at (2,2) {x}; \node at (2,3) {x}; \node at (1,4) {x}; \node at (1,5) {x}; \node at (0,6) {x};
\end{tikzpicture}
$
} 
\]
and $k$-tilings of the Aztec diamond of rank $m$. By weight-preserving, we mean that the weight of the configuration is $(y^{\rho_m})^k t^{\binom{m}{2}\binom{k}{2}}$ times the weight of the $k$-tiling.
\end{lem}

\begin{thm}
The partition function of 
\[
\resizebox{0.2\textwidth}{!}{
$
\begin{tikzpicture}[baseline = (current bounding box).center]
\def\vS{1.8}
\draw[fill=violet] (0,0) rectangle (5,1);
\draw[fill=gray] (0,1) rectangle (5,2);
\draw[fill=violet] (0,2) rectangle (5,3);
\draw[fill=gray] (0,3) rectangle (5,4);
\draw[fill=violet] (0,4) rectangle (5,5);
\draw[fill=gray] (0,5) rectangle (5,6);
\draw[step=1.0] (0,0) grid (5,3);
\draw[step=1.0] (0,3) grid (5,6);
\draw[fill=white] (5,0.5) circle (0.1);
\draw[fill=black] (5,1.5) circle (0.1);
\draw[fill=white] (5,2.5) circle (0.1);
\draw[fill=black] (5,3.5) circle (0.1);
\draw[fill=white] (5,4.5) circle (0.1);
\draw[fill=black] (5,5.5) circle (0.1);
\draw[fill=white] (0,0.5) circle (0.1);
\draw[fill=white] (0,1.5) circle (0.1);
\draw[fill=white] (0,2.5) circle (0.1);
\draw[fill=white] (0,3.5) circle (0.1);
\draw[fill=white] (0,4.5) circle (0.1);
\draw[fill=white] (0,5.5) circle (0.1);
\draw[fill=black] (0.5,0) circle (0.1);
\draw[fill=black] (1.5,0) circle (0.1);
\draw[fill=black] (2.5,0) circle (0.1);
\draw[fill=white] (3.5,0) circle (0.1);
\draw[fill=white] (4.5,0) circle (0.1);
\draw[fill=white] (0.5,6) circle (0.1);
\draw[fill=white] (1.5,6) circle (0.1);
\draw[fill=white] (2.5,6) circle (0.1);
\draw[fill=white] (3.5,6) circle (0.1);
\draw[fill=white] (4.5,6) circle (0.1);
\node[left,scale=\vS] at (0,0.5) {$x_1$};
\node[left,scale=\vS] at (0,1.5) {$y_1$};
\node[left,scale=\vS] at (0,3.5) {$\vdots$};
\node[left,scale=\vS] at (0,2.5) {$\vdots$};
\node[left,scale=\vS] at (0,4.5) {$x_m$};
\node[left,scale=\vS] at (0,5.5) {$y_m$};
\node[below,scale=\vS] at (1.5,0) {$\leftarrow \;\; m \;\; \rightarrow$};
\node at (3,0) {x}; \node at (3,1) {x}; \node at (2,2) {x}; \node at (2,3) {x}; \node at (1,4) {x}; \node at (1,5) {x}; \node at (0,6) {x};
\end{tikzpicture}
$
} 
\]
with $k$ colors is equal to $(y^{\rho_m})^k t^{\binom{m}{2}\binom{k}{2}}$ times the partition function of the $k$-tiling of the Aztec diamond of rank $m$ in the purple-gray model.  We have
\[
Z_{AD,purple-gray}^{(k)}(X_m;Y_m;t) = \prod_{l=0}^{k-1}\prod_{i\le j}\left(1+x_iy_jt^l\right).
\]
\end{thm}

\noindent In Figure \ref{ex:3til-bijection}, we exhibit an example of the bijection for $k,m=3$.
We have the following 3-tuples of partitions.
\begin{eqnarray*}
\bm{\lambda}^0&=&(\emptyset,\emptyset,\emptyset)\\
\bm{\lambda}^1&=&((1,1),(1,1,1),(1))\\
\bm{\lambda}^2&=&((1,1),(1,1),\emptyset)\\
\bm{\lambda}^3&=&((2,1),(1,1),(1))\\
\bm{\lambda}^4&=&((1),(1),\emptyset)\\
\bm{\lambda}^5&=&((2),(1),\emptyset)\\
\bm{\lambda}^6&=&(\emptyset,\emptyset,\emptyset)\\
\end{eqnarray*}

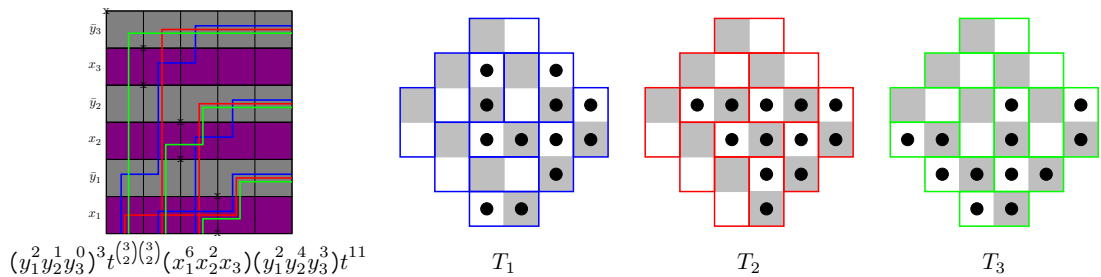
\begin{figure}[ht]
\begin{center}
\resizebox{\textwidth}{!}{
\begin{tabular}{cccc}
\resizebox{0.22\textwidth}{!}{
\begin{tikzpicture}[baseline = (current bounding box).center]
\def\vS{1.8}
\draw[fill=violet] (0,0) rectangle (5,1);
\draw[fill=gray] (0,1) rectangle (5,2);
\draw[fill=violet] (0,2) rectangle (5,3);
\draw[fill=gray] (0,3) rectangle (5,4);
\draw[fill=violet] (0,4) rectangle (5,5);
\draw[fill=gray] (0,5) rectangle (5,6);
\draw[step=1.0] (0,0) grid (5,3);
\draw[step=1.0] (0,3) grid (5,6);
\node[left,scale=\vS] at (0,0.5) {$x_1$};
\node[left,scale=\vS] at (0,1.5) {$y_1$};
\node[left,scale=\vS] at (0,3.5) {$y_2$};
\node[left,scale=\vS] at (0,2.5) {$x_2$};
\node[left,scale=\vS] at (0,4.5) {$x_3$};
\node[left,scale=\vS] at (0,5.5) {$y_3$};
\node at (3,0) {x}; \node at (3,1) {x}; \node at (2,2) {x}; \node at (2,3) {x}; \node at (1,4) {x}; \node at (1,5) {x}; \node at (0,6) {x};
\draw[very thick, blue] (0.4,0)--(0.4,1.6)--(1.4,1.6)--(1.4,4.6)--(2.4,4.6)--(2.4,5.6)--(5,5.6);
\draw[very thick, red] (0.5,0)--(0.5,0.5)--(1.5,0.5)--(1.5,5.5)--(5,5.5);
\draw[very thick, green] (0.6,0)--(0.6,5.4)--(5,5.4);
\draw[very thick, blue] (1.4,0)--(1.4,0.6)--(2.4,0.6)--(2.4,2.6)--(3.4,2.6)--(3.4,3.6)--(5,3.6);
\draw[very thick, red] (1.5,0)--(1.5,0.5)--(2.5,0.5)--(2.5,3.5)--(5,3.5);
\draw[very thick, green] (1.6,0)--(1.6,2.4)--(2.6,2.4)--(2.6,3.4)--(5,3.4);
\draw[very thick, blue] (2.4,0)--(2.4,0.6)--(3.4,0.6)--(3.4,1.6)--(5,1.6);
\draw[very thick, red] (2.5,0)--(2.5,0.5)--(3.5,0.5)--(3.5,1.5)--(5,1.5);
\draw[very thick, green] (2.6,0)--(2.6,0.4)--(3.6,0.4)--(3.6,1.4)--(5,1.4);
\end{tikzpicture}
} 
&
\resizebox{0.22\textwidth}{!}{
\begin{tikzpicture}[baseline = (current bounding box).center]
\checkerboard{2}
\draw[fill=black] (2.5,1.5) circle (5pt); \draw[fill=black] (2.5,0.5) circle (5pt); \draw[fill=black] (4.5,1.5) circle (5pt); \draw[fill=black] (3.5,-0.5) circle (5pt); \draw[fill=black] (4.5,-0.5) circle (5pt); \draw[fill=black] (5.5,-0.5) circle (5pt);
\draw[fill=black] (2.5,-0.5) circle (5pt); \draw[fill=black] (4.5,0.5) circle (5pt); \draw[fill=black] (2.5,-2.5) circle (5pt); \draw[fill=black] (4.5,-1.5) circle (5pt);
\draw[fill=black] (5.5,0.5) circle (5pt); \draw[fill=black] (3.5,-2.5) circle (5pt);
\draw[very thick, blue] (0,-1) rectangle (1,1); \draw[very thick, blue] (1,-2) rectangle (2,0); \draw[very thick, blue] (1,0) rectangle (2,2);
\draw[very thick, blue] (2,-3) rectangle (4,-2); \draw[very thick, blue] (2,-1) rectangle (4,0);
\draw[very thick, blue] (2,2) rectangle (4,3);
\draw[very thick, blue] (2,0) rectangle (3,2); \draw[very thick, blue] (2,-2) rectangle (4,-1); \draw[very thick, blue] (3,0) rectangle (4,2); 
\draw[very thick, blue] (4,-2) rectangle (5,0); \draw[very thick, blue] (4,0) rectangle (5,2); 
\draw[very thick, blue] (5,-1) rectangle (6,1);
\end{tikzpicture}
}
&
\resizebox{0.22\textwidth}{!}{
\begin{tikzpicture}[baseline = (current bounding box).center]
\checkerboard{2}
\draw[fill=black] (1.5,0.5) circle (5pt); \draw[fill=black] (2.5,0.5) circle (5pt); \draw[fill=black] (3.5,0.5) circle (5pt); \draw[fill=black] (3.5,-0.5) circle (5pt); \draw[fill=black] (4.5,-0.5) circle (5pt); \draw[fill=black] (5.5,-0.5) circle (5pt);
\draw[fill=black] (2.5,-0.5) circle (5pt); \draw[fill=black] (4.5,0.5) circle (5pt); \draw[fill=black] (3.5,-1.5) circle (5pt); \draw[fill=black] (4.5,-1.5) circle (5pt);
\draw[fill=black] (5.5,0.5) circle (5pt); \draw[fill=black] (3.5,-2.5) circle (5pt);
\draw[very thick, red] (0,-1) rectangle (1,1); \draw[very thick, red] (1,-2) rectangle (2,0); 
\draw[very thick, red] (1,0) rectangle (3,1);
\draw[very thick, red] (2,-3) rectangle (3,-1); 
\draw[very thick, red] (2,-1) rectangle (4,0);
\draw[very thick, red] (1,1) rectangle (3,2);
\draw[very thick, red] (3,-3) rectangle (4,-1);
\draw[very thick, red] (3,0) rectangle (5,1); 
\draw[very thick, red] (3,1) rectangle (5,2); 
\draw[very thick, red] (4,-2) rectangle (5,0); 
\draw[very thick, red] (2,2) rectangle (4,3); 
\draw[very thick, red] (5,-1) rectangle (6,1);
\end{tikzpicture}
}
&
\resizebox{0.22\textwidth}{!}{
\begin{tikzpicture}[baseline = (current bounding box).center]
\checkerboard{2}
\draw[fill=black] (0.5,-0.5) circle (5pt); \draw[fill=black] (1.5,-0.5) circle (5pt); \draw[fill=black] (3.5,0.5) circle (5pt); \draw[fill=black] (2.5,-1.5) circle (5pt); \draw[fill=black] (3.5,-1.5) circle (5pt); \draw[fill=black] (5.5,-0.5) circle (5pt);
\draw[fill=black] (1.5,-1.5) circle (5pt); \draw[fill=black] (3.5,-0.5) circle (5pt); \draw[fill=black] (2.5,-2.5) circle (5pt); \draw[fill=black] (4.5,-1.5) circle (5pt);
\draw[fill=black] (5.5,0.5) circle (5pt); \draw[fill=black] (3.5,-2.5) circle (5pt);
\draw[very thick, green] (0,-1) rectangle (2,0); 
\draw[very thick, green] (0,0) rectangle (2,1);
\draw[very thick, green] (1,-2) rectangle (3,-1);
\draw[very thick, green] (2,-3) rectangle (4,-2);
\draw[very thick, green] (3,-2) rectangle (5,-1);
\draw[very thick, green] (1,1) rectangle (3,2);
\draw[very thick, green] (3,1) rectangle (5,2);
\draw[very thick, green] (2,2) rectangle (4,3);
\draw[very thick, green] (2,-1) rectangle (3,1);
\draw[very thick, green] (3,-1) rectangle (4,1);
\draw[very thick, green] (4,-1) rectangle (5,1);
\draw[very thick, green] (5,-1) rectangle (6,1);
\end{tikzpicture}
}
\\
$(y_1^2y_2^1y_3^0)^3t^{\binom{3}{2}\binom{3}{2}}(x_1^6x_2^2x_3)(y_1^2 y_2^4y_3^3) t^{11}  $ & $T_1$ & $T_2$ & $T_3$
\end{tabular}
}
\end{center}
\caption{An example of a 3-tiling of the rank 3 Aztec diamond and the corresponding purple-gray lattice configuration} 
\label{ex:3til-bijection}
\end{figure}

\subsection{The white-pink partition function}

We can apply similar techniques to analyze the white-pink model.  Consider the following lattice and its associated partition function.
\[
\resizebox{0.2\textwidth}{!}{
$
\begin{tikzpicture}[baseline = (current bounding box).center]
\def\vS{1.8}
\draw[fill=pink] (0,0) rectangle (5,3);
\draw[step=1.0] (0,0) grid (5,3);
\draw[fill=white] (0,3) rectangle (5,6);
\draw[step=1.0] (0,3) grid (5,6);
\draw[fill=black] (0,0.5) circle (0.1);
\draw[fill=black] (0,1.5) circle (0.1);
\draw[fill=black] (0,2.5) circle (0.1);
\draw[fill=white] (5,3.5) circle (0.1);
\draw[fill=white] (5,4.5) circle (0.1);
\draw[fill=white] (5,5.5) circle (0.1);
\draw[fill=white] (5,0.5) circle (0.1);
\draw[fill=white] (5,1.5) circle (0.1);
\draw[fill=white] (5,2.5) circle (0.1);
\draw[fill=white] (0,3.5) circle (0.1);
\draw[fill=white] (0,4.5) circle (0.1);
\draw[fill=white] (0,5.5) circle (0.1);
\draw[fill=black] (0.5,0) circle (0.1);
\draw[fill=black] (1.5,0) circle (0.1);
\draw[fill=white] (2.5,0) circle (0.1);
\draw[fill=black] (0.5,6) circle (0.1);
\draw[fill=black] (1.5,6) circle (0.1);
\draw[fill=black] (2.5,6) circle (0.1);
\draw[fill=black] (3.5,6) circle (0.1);
\draw[fill=black] (4.5,6) circle (0.1);
\draw[fill=white] (3.5,0) circle (0.1);
\draw[fill=white] (4.5,0) circle (0.1);
\node[left,scale=\vS] at (0,0.5) {$y_1$};
\node[left,scale=\vS] at (0,1.5) {$\vdots$};
\node[left,scale=\vS] at (0,2.5) {$y_m$};
\node[left,scale=\vS] at (0,3.5) {$x_1$};
\node[left,scale=\vS] at (0,4.5) {$\vdots$};
\node[left,scale=\vS] at (0,5.5) {$x_m$};
\node[below,scale=\vS] at (3.5,0) {$\leftarrow \;\; m \;\;\rightarrow$};
\end{tikzpicture}
$
}
= (y_1^m y_2^{m-1}\ldots y_m ^1)^k
\]
\noindent Using the Yand Baxter equation (Proposition \ref{prop:YBEpinkwhite}), we get that
\[
\resizebox{0.2\textwidth}{!}{
$
\begin{tikzpicture}[baseline = (current bounding box).center]
\def\vS{1.8}
\draw[fill=violet!40!white] (0,1) rectangle (5,2);
\draw[fill=violet!40!white] (0,3) rectangle (5,4);
\draw[fill=violet!40!white] (0,5) rectangle (5,6);
\draw[step=1.0] (0,0) grid (5,3);
\draw[] (0,0) rectangle (5,1);
\draw[] (0,2) rectangle (5,3);
\draw[] (0,4) rectangle (5,4);
\draw[step=1.0] (0,3) grid (5,6);
\draw[fill=white] (5,5.5) circle (0.1);
\draw[fill=white] (5,4.5) circle (0.1);
\draw[fill=white] (5,3.5) circle (0.1);
\draw[fill=white] (5,2.5) circle (0.1);
\draw[fill=white] (5,1.5) circle (0.1);
\draw[fill=white] (5,0.5) circle (0.1);
\draw[fill=black] (0,5.5) circle (0.1);
\draw[fill=white] (0,4.5) circle (0.1);
\draw[fill=black] (0,3.5) circle (0.1);
\draw[fill=white] (0,2.5) circle (0.1);
\draw[fill=black] (0,1.5) circle (0.1);
\draw[fill=white] (0,0.5) circle (0.1);
\draw[fill=black] (0.5,0) circle (0.1);
\draw[fill=black] (1.5,0) circle (0.1);
\node[left,scale=\vS] at (0,0.5) {$x_1$};
\node[left,scale=\vS] at (0,1.5) {$y_1$};
\node[left,scale=\vS] at (0,2.5) {$\vdots$};
\node[left,scale=\vS] at (0,3.5) {$\vdots$};
\node[left,scale=\vS] at (0,4.5) {$x_m$};
\node[left,scale=\vS] at (0,5.5) {$y_m$};
\node[below,scale=\vS] at (3.5,0) {$\leftarrow \;\; m \;\;\rightarrow$};
\draw[fill=black] (0.5,0) circle (0.1);
\draw[fill=black] (1.5,0) circle (0.1);
\draw[fill=white] (2.5,0) circle (0.1);
\draw[fill=black] (0.5,6) circle (0.1);
\draw[fill=black] (1.5,6) circle (0.1);
\draw[fill=black] (2.5,6) circle (0.1);
\draw[fill=black] (3.5,6) circle (0.1);
\draw[fill=black] (4.5,6) circle (0.1);
\draw[fill=white] (3.5,0) circle (0.1);
\draw[fill=white] (4.5,0) circle (0.1);
\end{tikzpicture}
$
} 
= (y_1\ldots y_m)^k(y^{\rho_m})^k \prod_{l=0}^{k-1}\prod_{i\le j}(1+x_iy_jt^l)
\]

Given a configuration of the lattice
\[
\resizebox{0.2\textwidth}{!}{
$
\begin{tikzpicture}[baseline = (current bounding box).center]
\def\vS{1.8}
\draw[fill=violet!40!white] (0,1) rectangle (5,2);
\draw[fill=violet!40!white] (0,3) rectangle (5,4);
\draw[fill=violet!40!white] (0,5) rectangle (5,6);
\draw[step=1.0] (0,0) grid (5,3);
\draw[] (0,0) rectangle (5,1);
\draw[] (0,2) rectangle (5,3);
\draw[] (0,4) rectangle (5,4);
\draw[step=1.0] (0,3) grid (5,6);
\draw[fill=white] (5,5.5) circle (0.1);
\draw[fill=white] (5,4.5) circle (0.1);
\draw[fill=white] (5,3.5) circle (0.1);
\draw[fill=white] (5,2.5) circle (0.1);
\draw[fill=white] (5,1.5) circle (0.1);
\draw[fill=white] (5,0.5) circle (0.1);
\draw[fill=black] (0,5.5) circle (0.1);
\draw[fill=white] (0,4.5) circle (0.1);
\draw[fill=black] (0,3.5) circle (0.1);
\draw[fill=white] (0,2.5) circle (0.1);
\draw[fill=black] (0,1.5) circle (0.1);
\draw[fill=white] (0,0.5) circle (0.1);
\draw[fill=black] (0.5,0) circle (0.1);
\draw[fill=black] (1.5,0) circle (0.1);
\node[left,scale=\vS] at (0,0.5) {$x_1$};
\node[left,scale=\vS] at (0,1.5) {$y_1$};
\node[left,scale=\vS] at (0,2.5) {$\vdots$};
\node[left,scale=\vS] at (0,3.5) {$\vdots$};
\node[left,scale=\vS] at (0,4.5) {$x_m$};
\node[left,scale=\vS] at (0,5.5) {$y_m$};
\node[below,scale=\vS] at (3.5,0) {$\leftarrow \;\; m \;\;\rightarrow$};
\draw[fill=black] (0.5,0) circle (0.1);
\draw[fill=black] (1.5,0) circle (0.1);
\draw[fill=white] (2.5,0) circle (0.1);
\draw[fill=black] (0.5,6) circle (0.1);
\draw[fill=black] (1.5,6) circle (0.1);
\draw[fill=black] (2.5,6) circle (0.1);
\draw[fill=black] (3.5,6) circle (0.1);
\draw[fill=black] (4.5,6) circle (0.1);
\draw[fill=white] (3.5,0) circle (0.1);
\draw[fill=white] (4.5,0) circle (0.1);
\node at (2,0) {x}; \node at (2,1) {x}; \node at (3,2) {x}; \node at (3,3) {x}; \node at (4,4) {x}; \node at (4,5) {x}; \node at (5,6) {x};
\end{tikzpicture}
$
} 
\]
and looking at the labels on the horizontal edges row by row from bottom to top, we get a sequence of $2m+1$ $k$-tuples of Maya diagrams, where we mark the 0 content line with x's on the lattice.  The corresponding $2m+1$ $k$-tuples of partitions satisfy
\[
\bm{0} = \bm{\lambda}^0 \preceq  \bm{\lambda}^1 \succeq ' \ldots \preceq \bm{\lambda}^{2m-1} \succeq ' \bm{\lambda}^{2m}= \bm{0}.
\]

\subsection{Relating lattice configurations and $k$-tilings in the white-pink model}

Give a sequence of $k$-tuples of partitions
\[
\bm{0} = \bm{\lambda}^0 \preceq  \bm{\lambda}^1 \succeq ' \ldots \preceq \bm{\lambda}^{2m-1} \succeq ' \bm{\lambda}^{2m}= \bm{0},
\]
we get both a configuration of the lattice
\[
\resizebox{0.2\textwidth}{!}{
$
\begin{tikzpicture}[baseline = (current bounding box).center]
\def\vS{1.8}
\draw[fill=violet!40!white] (0,1) rectangle (5,2);
\draw[fill=violet!40!white] (0,3) rectangle (5,4);
\draw[fill=violet!40!white] (0,5) rectangle (5,6);
\draw[step=1.0] (0,0) grid (5,3);
\draw[] (0,0) rectangle (5,1);
\draw[] (0,2) rectangle (5,3);
\draw[] (0,4) rectangle (5,4);
\draw[step=1.0] (0,3) grid (5,6);
\draw[fill=white] (5,5.5) circle (0.1);
\draw[fill=white] (5,4.5) circle (0.1);
\draw[fill=white] (5,3.5) circle (0.1);
\draw[fill=white] (5,2.5) circle (0.1);
\draw[fill=white] (5,1.5) circle (0.1);
\draw[fill=white] (5,0.5) circle (0.1);
\draw[fill=black] (0,5.5) circle (0.1);
\draw[fill=white] (0,4.5) circle (0.1);
\draw[fill=black] (0,3.5) circle (0.1);
\draw[fill=white] (0,2.5) circle (0.1);
\draw[fill=black] (0,1.5) circle (0.1);
\draw[fill=white] (0,0.5) circle (0.1);
\draw[fill=black] (0.5,0) circle (0.1);
\draw[fill=black] (1.5,0) circle (0.1);
\node[left,scale=\vS] at (0,0.5) {$x_1$};
\node[left,scale=\vS] at (0,1.5) {$y_1$};
\node[left,scale=\vS] at (0,2.5) {$\vdots$};
\node[left,scale=\vS] at (0,3.5) {$\vdots$};
\node[left,scale=\vS] at (0,4.5) {$x_m$};
\node[left,scale=\vS] at (0,5.5) {$y_m$};
\node[below,scale=\vS] at (3.5,0) {$\leftarrow \;\; m \;\;\rightarrow$};
\draw[fill=black] (0.5,0) circle (0.1);
\draw[fill=black] (1.5,0) circle (0.1);
\draw[fill=white] (2.5,0) circle (0.1);
\draw[fill=black] (0.5,6) circle (0.1);
\draw[fill=black] (1.5,6) circle (0.1);
\draw[fill=black] (2.5,6) circle (0.1);
\draw[fill=black] (3.5,6) circle (0.1);
\draw[fill=black] (4.5,6) circle (0.1);
\draw[fill=white] (3.5,0) circle (0.1);
\draw[fill=white] (4.5,0) circle (0.1);
\node at (2,0) {x}; \node at (2,1) {x}; \node at (3,2) {x}; \node at (3,3) {x}; \node at (4,4) {x}; \node at (4,5) {x}; \node at (5,6) {x};
\end{tikzpicture}
$
} 
\]
and a $k$-tiling of the Aztec diamond of rank $m$.  For example, in Figures \ref{whpi}, we give the 8 configurations corresponding to the 1-tilings of the rank 2 Aztec diamond, which were listed in Figure \ref{rank2}.  It turns out that the weight of the lattice configuration and the weight of the $k$-tiling are related.

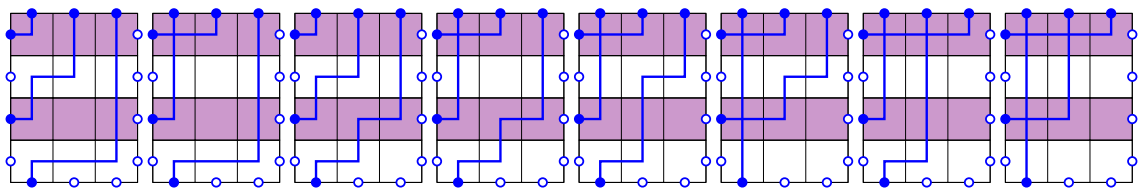
\begin{figure}[ht]
\resizebox{\textwidth}{!}{
\begin{tikzpicture}\vertexb{2}
\draw[ultra thick, blue] (0.5,0)--(0.5,0.5)--(2.5,0.5)--(2.5,4);;
\draw[ultra thick, blue] (0,1.5)--(0.5,1.5)--(0.5,2.5)--(1.5,2.5)--(1.5,4);
\draw[ultra thick, blue] (0,3.5)--(0.5,3.5)--(0.5,4);
\end{tikzpicture}
\begin{tikzpicture}\vertexb{2}
\draw[ultra thick, blue] (0.5,0)--(0.5,0.5)--(2.5,0.5)--(2.5,4);;
\draw[ultra thick, blue] (0,1.5)--(0.5,1.5)--(0.5,4);
\draw[ultra thick, blue] (0,3.5)--(1.5,3.5)--(1.5,4);
\end{tikzpicture}
\begin{tikzpicture}\vertexb{2}
\draw[ultra thick, blue] (0.5,0)--(0.5,0.5)--(1.5,0.5)--(1.5,1.5)--(2.5,1.5)--(2.5,4);;
\draw[ultra thick, blue] (0,1.5)--(0.5,1.5)--(0.5,2.5)--(1.5,2.5)--(1.5,4);
\draw[ultra thick, blue] (0,3.5)--(0.5,3.5)--(0.5,4);
\end{tikzpicture}
\begin{tikzpicture}\vertexb{2}
\draw[ultra thick, blue] (0.5,0)--(0.5,0.5)--(1.5,0.5)--(1.5,1.5)--(2.5,1.5)--(2.5,4);;
\draw[ultra thick, blue] (0,1.5)--(0.5,1.5)--(0.5,4);
\draw[ultra thick, blue] (0,3.5)--(1.5,3.5)--(1.5,4);
\end{tikzpicture}
\begin{tikzpicture}\vertexb{2}
\draw[ultra thick, blue] (0.5,0)--(0.5,0.5)--(1.5,0.5)--(1.5,2.5)--(2.5,2.5)--(2.5,4);;
\draw[ultra thick, blue] (0,1.5)--(0.5,1.5)--(0.5,4);
\draw[ultra thick, blue] (0,3.5)--(1.5,3.5)--(1.5,4);
\end{tikzpicture}
\begin{tikzpicture}\vertexb{2}
\draw[ultra thick, blue] (0.5,0)--(0.5,4);
\draw[ultra thick, blue] (0,1.5)--(1.5,1.5)--(1.5,2.5)--(2.5,2.5)--(2.5,4);
\draw[ultra thick, blue] (0,3.5)--(1.5,3.5)--(1.5,4);
\end{tikzpicture}
\begin{tikzpicture}\vertexb{2}
\draw[ultra thick, blue] (0.5,0)--(0.5,0.5)--(1.5,0.5)--(1.5,4);
\draw[ultra thick, blue] (0,1.5)--(0.5,1.5)--(0.5,4);
\draw[ultra thick, blue] (0,3.5)--(2.5,3.5)--(2.5,4);
\end{tikzpicture}
\begin{tikzpicture}\vertexb{2}
\draw[ultra thick, blue] (0.5,0)--(0.5,4);
\draw[ultra thick, blue] (0,1.5)--(1.5,1.5)--(1.5,4);
\draw[ultra thick, blue] (0,3.5)--(2.5,3.5)--(2.5,4);
\end{tikzpicture}}
\caption{White-pink lattice configurations for the Aztec diamond of rank 2}
\label{whpi}
\end{figure}

\begin{lem}
In the $l$-th pink row, for each color,
\[ \# 
\begin{tikzpicture}[baseline = (current bounding box).center]
\draw[fill=violet!40!white] (0,0) rectangle (1,1);
\draw[very thick, blue] (0,0.5)--(0.5,0.5)--(0.5,1);
\end{tikzpicture}
+ \# 
\begin{tikzpicture}[baseline = (current bounding box).center]
\draw[fill=violet!40!white] (0,0) rectangle (1,1);
\end{tikzpicture}
= m-l+1. \]
\end{lem}

\begin{proof}
In the $l$-th pink row, there are $l+1$ particles entering from the bottom and $m+2$ vertices.  This means there are $(m+2)-(l+1)$ vertices in which a particle does not enter from the bottom.
\end{proof}

If we pull out a factor of $y_l^{m-l+1}$ for the $l$-th pink row for each $l \in [m]$, then we get an overall factor of $(y_1\ldots y_m)^k(y^{\rho_m})^k$ on the left-hand side, which cancels with the same factor on the right-hand side.  After removing this factor, the $l$-th pink row now contributes a $y_l$ whenever there is a vertical path, which corresponds to a domino of the form
\[ \resizebox{!}{0.25cm}{
\begin{tikzpicture}[baseline = (current bounding box).center]
\draw[lightgray,fill=lightgray] (0,0) rectangle (1,1);
\draw[lightgray] (1,0) rectangle (2,1);
\draw[very thick, blue] (0,0) rectangle (2,1);
\draw[very thick, blue,fill=blue] (0.5,0.5) circle (5pt);
\draw[very thick, blue,fill=blue] (1.5,0.5) circle (5pt);
\end{tikzpicture} } \]
whose right square is on slice $2l-1$.  We get a $t$ whenever a smaller color exits right and a larger color is vertical in a pink row, and whenever a smaller color exits right and a larger color is present in a white row.  This corresponds to a pair of dominos of the form
\begin{equation*}
\resizebox{1cm}{!}{
\begin{tikzpicture}[baseline = (current bounding box).center]
\draw[lightgray,fill=lightgray] (1,-1) rectangle (2,0);
\draw[lightgray,fill=lightgray] (0,0) rectangle (1,1);
\draw[line width=1mm, blue] (0,-1) rectangle (2,0);
\draw[line width=1mm, red] (-0.01,-1.01) rectangle (0.99,0.99);
\end{tikzpicture}}\ ,
\;\;\;
\resizebox{1cm}{!}{ \begin{tikzpicture}[baseline = (current bounding box).center]
\draw[lightgray,fill=lightgray] (1,0) rectangle (2,1);
\draw[line width=1mm, blue] (0,0) rectangle (2,1);
\draw[line width=1mm, red] (-0.05,-0.05) rectangle (1.95,0.95);
\end{tikzpicture}}\ ,
\;\;\;
\resizebox{1cm}{!}{
 \begin{tikzpicture}[baseline = (current bounding box).center]
\draw[lightgray,fill=lightgray] (0,0) rectangle (1,1);
\draw[lightgray,fill=lightgray] (1,1) rectangle (2,2);
\draw[line width=1mm, red] (0,0) rectangle (2,1);
\draw[line width=1mm, blue] (0.99,-0.01) rectangle (1.99,1.99);
\end{tikzpicture}}\ , {\rm or}
\;\;\;
\resizebox{1.5cm}{!}{
 \begin{tikzpicture}[baseline = (current bounding box).center]
\draw[lightgray] (1,0) rectangle (2,1);
\draw[lightgray,fill=lightgray] (0,0) rectangle (1,1);
\draw[lightgray,fill=lightgray] (2,0) rectangle (3,1);
\draw[line width=1mm, blue] (0.99,-0.01) rectangle (2.99,0.99);
\draw[line width=1mm, red] (0,0) rectangle (2,1);
\end{tikzpicture}}
\end{equation*}
where blue is a smaller color than red.  

Putting it all together, we arrive at the following results.

\begin{lem}
There is a weight-preserving bijection between configurations of the white-pink lattice
\[
\resizebox{0.2\textwidth}{!}{
$
\begin{tikzpicture}[baseline = (current bounding box).center]
\def\vS{1.8}
\draw[fill=violet!40!white] (0,1) rectangle (5,2);
\draw[fill=violet!40!white] (0,3) rectangle (5,4);
\draw[fill=violet!40!white] (0,5) rectangle (5,6);
\draw[step=1.0] (0,0) grid (5,3);
\draw[] (0,0) rectangle (5,1);
\draw[] (0,2) rectangle (5,3);
\draw[] (0,4) rectangle (5,4);
\draw[step=1.0] (0,3) grid (5,6);
\draw[fill=white] (5,5.5) circle (0.1);
\draw[fill=white] (5,4.5) circle (0.1);
\draw[fill=white] (5,3.5) circle (0.1);
\draw[fill=white] (5,2.5) circle (0.1);
\draw[fill=white] (5,1.5) circle (0.1);
\draw[fill=white] (5,0.5) circle (0.1);
\draw[fill=black] (0,5.5) circle (0.1);
\draw[fill=white] (0,4.5) circle (0.1);
\draw[fill=black] (0,3.5) circle (0.1);
\draw[fill=white] (0,2.5) circle (0.1);
\draw[fill=black] (0,1.5) circle (0.1);
\draw[fill=white] (0,0.5) circle (0.1);
\draw[fill=black] (0.5,0) circle (0.1);
\draw[fill=black] (1.5,0) circle (0.1);
\node[left,scale=\vS] at (0,0.5) {$x_1$};
\node[left,scale=\vS] at (0,1.5) {$y_1$};
\node[left,scale=\vS] at (0,2.5) {$\vdots$};
\node[left,scale=\vS] at (0,3.5) {$\vdots$};
\node[left,scale=\vS] at (0,4.5) {$x_m$};
\node[left,scale=\vS] at (0,5.5) {$y_m$};
\node[below,scale=\vS] at (3.5,0) {$\leftarrow \;\; m \;\;\rightarrow$};
\draw[fill=black] (0.5,0) circle (0.1);
\draw[fill=black] (1.5,0) circle (0.1);
\draw[fill=white] (2.5,0) circle (0.1);
\draw[fill=black] (0.5,6) circle (0.1);
\draw[fill=black] (1.5,6) circle (0.1);
\draw[fill=black] (2.5,6) circle (0.1);
\draw[fill=black] (3.5,6) circle (0.1);
\draw[fill=black] (4.5,6) circle (0.1);
\draw[fill=white] (3.5,0) circle (0.1);
\draw[fill=white] (4.5,0) circle (0.1);
\node at (2,0) {x}; \node at (2,1) {x}; \node at (3,2) {x}; \node at (3,3) {x}; \node at (4,4) {x}; \node at (4,5) {x}; \node at (5,6) {x};
\end{tikzpicture}
$
} 
\]
and $k$-tilings of the Aztec diamond of rank $m$. By weight-preserving, we mean that the weight of the configuration is $(y^{\rho_m})^k t^{\binom{m}{2}\binom{k}{2}}$ times the weight of the $k$-tiling.
\end{lem}

\begin{thm}
The partition function of 
\[
\resizebox{0.2\textwidth}{!}{
$
\begin{tikzpicture}[baseline = (current bounding box).center]
\def\vS{1.8}
\draw[fill=violet!40!white] (0,1) rectangle (5,2);
\draw[fill=violet!40!white] (0,3) rectangle (5,4);
\draw[fill=violet!40!white] (0,5) rectangle (5,6);
\draw[step=1.0] (0,0) grid (5,3);
\draw[] (0,0) rectangle (5,1);
\draw[] (0,2) rectangle (5,3);
\draw[] (0,4) rectangle (5,4);
\draw[step=1.0] (0,3) grid (5,6);
\draw[fill=white] (5,5.5) circle (0.1);
\draw[fill=white] (5,4.5) circle (0.1);
\draw[fill=white] (5,3.5) circle (0.1);
\draw[fill=white] (5,2.5) circle (0.1);
\draw[fill=white] (5,1.5) circle (0.1);
\draw[fill=white] (5,0.5) circle (0.1);
\draw[fill=black] (0,5.5) circle (0.1);
\draw[fill=white] (0,4.5) circle (0.1);
\draw[fill=black] (0,3.5) circle (0.1);
\draw[fill=white] (0,2.5) circle (0.1);
\draw[fill=black] (0,1.5) circle (0.1);
\draw[fill=white] (0,0.5) circle (0.1);
\draw[fill=black] (0.5,0) circle (0.1);
\draw[fill=black] (1.5,0) circle (0.1);
\node[left,scale=\vS] at (0,0.5) {$x_1$};
\node[left,scale=\vS] at (0,1.5) {$y_1$};
\node[left,scale=\vS] at (0,2.5) {$\vdots$};
\node[left,scale=\vS] at (0,3.5) {$\vdots$};
\node[left,scale=\vS] at (0,4.5) {$x_m$};
\node[left,scale=\vS] at (0,5.5) {$y_m$};
\node[below,scale=\vS] at (3.5,0) {$\leftarrow \;\; m \;\;\rightarrow$};
\draw[fill=black] (0.5,0) circle (0.1);
\draw[fill=black] (1.5,0) circle (0.1);
\draw[fill=white] (2.5,0) circle (0.1);
\draw[fill=black] (0.5,6) circle (0.1);
\draw[fill=black] (1.5,6) circle (0.1);
\draw[fill=black] (2.5,6) circle (0.1);
\draw[fill=black] (3.5,6) circle (0.1);
\draw[fill=black] (4.5,6) circle (0.1);
\draw[fill=white] (3.5,0) circle (0.1);
\draw[fill=white] (4.5,0) circle (0.1);

\node at (2,0) {x}; \node at (2,1) {x}; \node at (3,2) {x}; \node at (3,3) {x}; \node at (4,4) {x}; \node at (4,5) {x}; \node at (5,6) {x};
\end{tikzpicture}
$
}
\]
with $k$ colors is equal to $(y_1\ldots y_m)^k(y^{\rho_m})^k$ times the partition function of the $k$-tiling of the Aztec diamond in the white-pink model.  We have
\[
Z_{AD,white-pink}^{(k)}(X_m;Y_m;t) = \prod_{l=0}^{k-1}\prod_{i\le j}\left(1+x_iy_jt^l\right).
\]
\end{thm}

\subsection{Combining the two models}

Since the partition functions of the two models are equal, we will write 
\[ Z_{AD}^{(k)}(X_m;Y_m;t) = \prod_{l=0}^{k-1}\prod_{i\le j}\left(1+x_iy_jt^l\right). \]
Moreover, we get a surprising combinatorial statement.
\begin{prop}
Fix integers $k,m,l,r_1,\ldots,r_m,s_1,\ldots,s_m \geq 0$.  There exists a bijection between:
\begin{itemize}
\item the $k$-tilings of the Aztec diamond of rank $m$ with $\ell$ pairs of dominos of the form

\begin{equation*}
\resizebox{1cm}{!}{
\begin{tikzpicture}[baseline = (current bounding box).center]
\draw[lightgray] (0,1) rectangle (1,2);
\draw[lightgray,fill=lightgray] (0,0) rectangle (1,1);
\draw[lightgray] (-1,0) rectangle (0,1);
\draw[line width=1mm, blue] (0,0) rectangle (1,2);
\draw[line width=1mm, red] (-1.01,-0.01) rectangle (0.99,0.99);
\end{tikzpicture}}\ ,
\;\;\;
\resizebox{.5cm}{!}{
 \begin{tikzpicture}[baseline = (current bounding box).center]
\draw[lightgray,fill=lightgray] (0,1) rectangle (1,2);
\draw[lightgray] (0,0) rectangle (1,1);
\draw[line width=1mm, blue] (0,0) rectangle (1,2);
\draw[line width=1mm, red] (-0.05,-0.05) rectangle (0.95,1.95);
\end{tikzpicture}}\ ,
\;\;\;
\resizebox{1cm}{!}{
 \begin{tikzpicture}[baseline = (current bounding box).center]
\draw[lightgray,fill=lightgray] (0,1) rectangle (1,2);
\draw[lightgray] (0,0) rectangle (1,1);
\draw[line width=1mm, blue] (-1.01,0.99) rectangle (0.99,1.99);
\draw[line width=1mm, red] (0,0) rectangle (1,2);
\end{tikzpicture}}\ , {\rm or}
\;\;\; 
\resizebox{.5cm}{!}{
 \begin{tikzpicture}[baseline = (current bounding box).center]
\draw[lightgray,fill=lightgray] (0,1) rectangle (1,2);
\draw[lightgray] (0,0) rectangle (1,1);
\draw[line width=1mm, blue] (-0.01,0.99) rectangle (.99,2.99);
\draw[line width=1mm, red] (0,0) rectangle (1,2);
\end{tikzpicture}}
\end{equation*}
where where blue is a smaller color than red, $r_i$ dominos of the form
\resizebox{!}{0.3cm}{
\begin{tikzpicture}[baseline = (current bounding box).center]
\draw[lightgray, fill=lightgray] (0,0) rectangle (1,1);
\draw[lightgray] (0,1) rectangle (1,2);
\draw[very thick, blue] (0,0) rectangle (1,2);
\end{tikzpicture} }
whose top square is on slice $2i-1$ for each $i \in [m]$, and $s_i$ dominos of the form
\resizebox{!}{0.3cm}{
\begin{tikzpicture}[baseline = (current bounding box).center]
\draw[lightgray] (0,0) rectangle (1,1);
\draw[lightgray,fill=lightgray] (0,1) rectangle (1,2);
\draw[very thick, blue] (0,0) rectangle (1,2);
\end{tikzpicture} }
whose bottom square is on slice $2i-1$
for each $i \in [m]$; and
\item the $k$-tilings of the Aztec diamond of rank $m$ with $\ell$ pairs of dominos of the form
\begin{equation*}
\resizebox{1cm}{!}{
\begin{tikzpicture}[baseline = (current bounding box).center]
\draw[lightgray,fill=lightgray] (1,-1) rectangle (2,0);
\draw[lightgray,fill=lightgray] (0,0) rectangle (1,1);
\draw[line width=1mm, blue] (0,-1) rectangle (2,0);
\draw[line width=1mm, red] (-0.01,-1.01) rectangle (0.99,0.99);
\end{tikzpicture}}\ ,
\;\;\;
\resizebox{1cm}{!}{ \begin{tikzpicture}[baseline = (current bounding box).center]
\draw[lightgray,fill=lightgray] (1,0) rectangle (2,1);
\draw[line width=1mm, blue] (0,0) rectangle (2,1);
\draw[line width=1mm, red] (-0.05,-0.05) rectangle (1.95,0.95);
\end{tikzpicture}}\ ,
\;\;\;
\resizebox{1cm}{!}{
 \begin{tikzpicture}[baseline = (current bounding box).center]
\draw[lightgray,fill=lightgray] (0,0) rectangle (1,1);
\draw[lightgray,fill=lightgray] (1,1) rectangle (2,2);
\draw[line width=1mm, blue] (0.99,-0.01) rectangle (1.99,1.99);
\draw[line width=1mm, red] (0,0) rectangle (2,1);
\end{tikzpicture}}\ , {\rm or}
\;\;\;
\resizebox{1.5cm}{!}{
 \begin{tikzpicture}[baseline = (current bounding box).center]
\draw[lightgray] (1,0) rectangle (2,1);
\draw[lightgray,fill=lightgray] (0,0) rectangle (1,1);
\draw[lightgray,fill=lightgray] (2,0) rectangle (3,1);
\draw[line width=1mm, blue] (0.99,-0.01) rectangle (2.99,0.99);
\draw[line width=1mm, red] (0,0) rectangle (2,1);
\end{tikzpicture}}
\end{equation*}
where where blue is a smaller color than red, $r_i$ dominos of the form
\resizebox{!}{0.4cm}{
\begin{tikzpicture}
\draw[lightgray] (0,0) rectangle (1,1);
\draw[lightgray,fill=lightgray] (1,0) rectangle (2,1);
\draw[very thick, blue] (0,0) rectangle (2,1);
\end{tikzpicture} }
whose left square is on slice $2i-1$ for each $i \in [m]$, and $s_i$ dominos of the form
\resizebox{!}{0.4cm}{
\begin{tikzpicture}
\draw[lightgray,fill=lightgray] (0,0) rectangle (1,1);
\draw[lightgray] (1,0) rectangle (2,1);
\draw[very thick, blue] (0,0) rectangle (2,1);
\end{tikzpicture} }
whose right square is on slice $2i-1$ for each $i \in [m]$.
\end{itemize}
\label{prop:open}
\end{prop}

\noindent We leave it as an open problem to find a combinatorial proof of the proposition. 

\begin{remark}\label{rmk:purplegray}
    In the remainder of this article we will focus on the purple-gray model of the tilings. Note that corresponding results (for example, Theorem \ref{thm:2tilingAC} and Lemma \ref{lem:0toinf}) for the white-pink model require separate proof without knowing more about the bijection above. However, we believe the techniques used for the purple-gray model can be easily translated to the white-pink model. In order to keep the present article at a reasonable length and not repeat similar calculations, we do not pursue that here. However, in Remark \ref{rmk:whitepinkpaths}, we sketch how the techniques of Section \ref{paths} may be translated to the white-pink model.
\end{remark}

\section{The case $t=0$}
\label{paths}
In what follows we only consider the purple-gray model of $k$-tiling, as noted in Remark \ref{rmk:purplegray}.

\subsection{Schr\"oder paths} \label{schroder-section}

We can assign paths to the dominos according to the following rules.
\[
\resizebox{.5cm}{!}{
\begin{tikzpicture}[baseline = (current bounding box).center]
\draw[lightgray, fill=lightgray] (0,0) rectangle (1,1);
\draw[lightgray] (0,1) rectangle (1,2);
\draw[very thick] (0,0) rectangle (1,2);
\draw[ultra thick, red] (0,1.5)--(1,0.5);
\end{tikzpicture}},\;\;\;
\resizebox{1cm}{!}{
\begin{tikzpicture}[baseline = (current bounding box).center]
\draw[lightgray] (0,0) rectangle (1,1);
\draw[lightgray,fill=lightgray] (1,0) rectangle (2,1);
\draw[very thick] (0,0) rectangle (2,1);
\draw[ultra thick, red] (0,0.5)--(2,0.5);
\end{tikzpicture}},\;\;\;
\resizebox{.5cm}{!}{
\begin{tikzpicture}[baseline = (current bounding box).center]
\draw[lightgray] (0,0) rectangle (1,1);
\draw[lightgray,fill=lightgray] (0,1) rectangle (1,2);
\draw[very thick] (0,0) rectangle (1,2);
\draw[ultra thick, red, red] (0,0.5)--(1,1.5);
\end{tikzpicture}},\;\;\;
\resizebox{1cm}{!}{
\begin{tikzpicture}[baseline = (current bounding box).center]
\draw[lightgray,fill=lightgray] (0,0) rectangle (1,1);
\draw[lightgray] (1,0) rectangle (2,1);
\draw[very thick] (0,0) rectangle (2,1);
\end{tikzpicture}}
\]

Then instead of domino tilings, we can consider non-intersecting Schr\"oder paths \cite{ardila,aztecschroder,johansson2}. Schr\"oder paths are lattice paths using North-East (1,1), South-East (1,-1) and East (2,0) steps starting at (0,0) ending at $(n,0)$ and they do not cross the $y=0$ axis.
For example, the tiling from Figure \ref{ex} 
gives the set of paths in Figure \ref{ex2}.
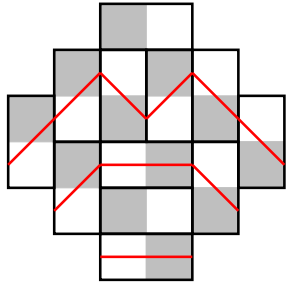
\begin{figure}[ht]
\centering
\resizebox{0.25\textwidth}{!}{
\begin{tikzpicture}[baseline = (current bounding box).center]
\checkerboard{2}
\draw[ultra thick] (0,-1) rectangle (1,1); \draw[ultra thick] (1,-2) rectangle (2,0); \draw[ultra thick] (1,0) rectangle (2,2);
\draw[ultra thick] (2,-3) rectangle (4,-2); \draw[ultra thick] (2,-1) rectangle (4,0);
\draw[ultra thick] (2,2) rectangle (4,3);
\draw[ultra thick] (2,0) rectangle (3,2); \draw[ultra thick] (2,-2) rectangle (4,-1); \draw[ultra thick] (3,0) rectangle (4,2); 
\draw[ultra thick] (4,-2) rectangle (5,0); \draw[ultra thick] (4,0) rectangle (5,2); 
\draw[ultra thick] (5,-1) rectangle (6,1);       
\draw[ultra thick, red] (0,-0.5)--(1,0.5); 
\draw[ultra thick, red] (1,-1.5)--(2,-0.5); 
\draw[ultra thick, red] (1,0.5)--(2,1.5);
\draw[ultra thick, red] (2,-2.5)--(4,-2.5); 
\draw[ultra thick, red] (2,-0.5)--(4,-0.5);
\draw[ultra thick, red] (2,1.5)--(3,0.5); 
\draw[ultra thick, red] (3,0.5)--(4,1.5); 
\draw[ultra thick, red] (4,-0.5)--(5,-1.5); 
\draw[ultra thick, red] (4,1.5)--(5,0.5); 
\draw[ultra thick, red] (5,0.5)--(6,-0.5);       
\end{tikzpicture}}
\label{ex2}
\caption{Non-intersecting paths and domino tiling}
\end{figure}

\noindent This gives a well-known bijection (see for example \cite{aztecschroder}) between:
\begin{itemize}
    \item domino tilings of the Aztec diamond of rank $m$
    \item $m$-tuples of non-intersecting paths using North-East $(1,1)$, South-East $(1,-1)$ and East  $(2,0)$ steps such that the $i$-th path starts at $(-m-1+i,-i+\frac{1}{2})$ and ends at $(m+1-i,-i+\frac{1}{2})$.
\end{itemize}
An example for $m=2$ is given in Figure \ref{expath}.

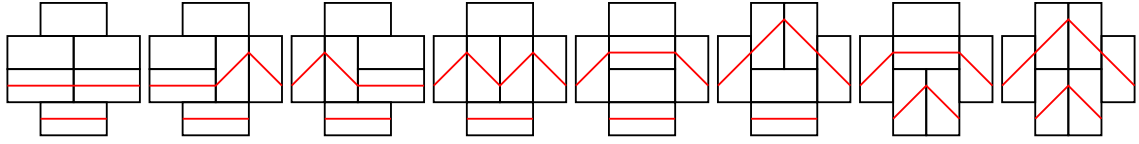
\begin{figure}[ht]
\resizebox{\textwidth}{!}{
\begin{tikzpicture}[baseline = (current bounding box).center]
\draw[ultra thick](0,0) rectangle (2,1);
\draw[ultra thick](4,0) rectangle (2,1);
\draw[ultra thick](0,-1) rectangle (2,0);
\draw[ultra thick](4,-1) rectangle (2,0);
\draw[ultra thick](1,-2) rectangle (3,-1);
\draw[ultra thick](1,1) rectangle (3,2);
\draw[ultra thick, red](0,-.5)--(4,-.5);
\draw[ultra thick, red](1,-1.5)--(3,-1.5);

\end{tikzpicture}\ \ 
\begin{tikzpicture}[baseline = (current bounding box).center]
\draw[ultra thick](0,0) rectangle (2,1);
\draw[ultra thick](3,1) rectangle (2,-1);
\draw[ultra thick](0,-1) rectangle (2,0);
\draw[ultra thick](4,1) rectangle (3,-1);
\draw[ultra thick](1,-2) rectangle (3,-1);
\draw[ultra thick](1,1) rectangle (3,2);
\draw[ultra thick, red](0,-.5)--(2,-.5)--(3,.5)--(4,-.5);
\draw[ultra thick, red](1,-1.5)--(3,-1.5);

\end{tikzpicture}\ \ 
\begin{tikzpicture}[baseline = (current bounding box).center]
\draw[ultra thick](0,-1) rectangle (1,1);
\draw[ultra thick](1,-1) rectangle (2,1);
\draw[ultra thick](2,1) rectangle (4,0);
\draw[ultra thick](4,-1) rectangle (2,0);
\draw[ultra thick](1,-2) rectangle (3,-1);
\draw[ultra thick](1,1) rectangle (3,2);
\draw[ultra thick, red](0,-.5)--(1,.5)--(2,-.5)--(4,-.5);
\draw[ultra thick, red](1,-1.5)--(3,-1.5);

\end{tikzpicture}\ \ 
\begin{tikzpicture}[baseline = (current bounding box).center]
\draw[ultra thick](0,-1) rectangle (1,1);
\draw[ultra thick](1,-1) rectangle (2,1);
\draw[ultra thick](3,1) rectangle (2,-1);
\draw[ultra thick](4,1) rectangle (3,-1);
\draw[ultra thick](1,-2) rectangle (3,-1);
\draw[ultra thick](1,1) rectangle (3,2);
\draw[ultra thick, red](0,-.5)--(1,.5)--(2,-.5)--(3,.5)--(4,-.5);
\draw[ultra thick, red](1,-1.5)--(3,-1.5);

\end{tikzpicture}\ \ 
\begin{tikzpicture}[baseline = (current bounding box).center]
\draw[ultra thick](0,-1) rectangle (1,1);
\draw[ultra thick](1,-1) rectangle (3,0);
\draw[ultra thick](3,1) rectangle (1,0);
\draw[ultra thick](4,1) rectangle (3,-1);
\draw[ultra thick](1,-2) rectangle (3,-1);
\draw[ultra thick](1,1) rectangle (3,2);
\draw[ultra thick, red](0,-.5)--(1,0.5)--(3,0.5)--(4,-.5);
\draw[ultra thick, red](1,-1.5)--(3,-1.5);

\end{tikzpicture}\ \ 
\begin{tikzpicture}[baseline = (current bounding box).center]
\draw[ultra thick](0,-1) rectangle (1,1);
\draw[ultra thick](1,-1) rectangle (3,0);
\draw[ultra thick](2,0) rectangle (3,2);
\draw[ultra thick](4,1) rectangle (3,-1);
\draw[ultra thick](1,-2) rectangle (3,-1);
\draw[ultra thick](1,0) rectangle (2,2);
\draw[ultra thick, red](0,-.5)--(2,1.5)--(4,-.5);
\draw[ultra thick, red](1,-1.5)--(3,-1.5);

\end{tikzpicture}\ \ 
\begin{tikzpicture}[baseline = (current bounding box).center]
\draw[ultra thick](0,-1) rectangle (1,1);
\draw[ultra thick](3,1) rectangle (1,0);
\draw[ultra thick](4,1) rectangle (3,-1);
\draw[ultra thick](2,-2) rectangle (3,0);
\draw[ultra thick](1,-2) rectangle (2,0);
\draw[ultra thick](1,1) rectangle (3,2);
\draw[ultra thick, red](0,-.5)--(1,0.5)--(3,0.5)--(4,-.5);
\draw[ultra thick, red](1,-1.5)--(2,-.5)--(3,-1.5);
\end{tikzpicture}\ \ 
\begin{tikzpicture}[baseline = (current bounding box).center]
\draw[ultra thick](0,-1) rectangle (1,1);
\draw[ultra thick](1,-2) rectangle (2,0);
\draw[ultra thick](2,0) rectangle (3,2);
\draw[ultra thick](4,1) rectangle (3,-1);
\draw[ultra thick](2,-2) rectangle (3,0);
\draw[ultra thick](1,0) rectangle (2,2);
\draw[ultra thick, red](0,-.5)--(2,1.5)--(4,-.5);
\draw[ultra thick, red](1,-1.5)--(2,-.5)--(3,-1.5);
\end{tikzpicture}
}
\label{ex}
\caption{Non-intersecting paths for the Aztec diamond of rank 2}
\label{expath}
\end{figure}

The weight of a domino tiling can be expressed in terms of Schr\"oder paths.  In the purple-gray model, the power of $x_i$ is the number of down-right steps starting on slice $2i-1$, the power of $y_i$ is the number of up-right steps starting on slice $2i-1$.  In the white-pink model, the power of $x_i$ is the number of horizontal steps starting on slice $2i-1$, the power of $y_i$ is the number of dominos with no paths whose right square is on slice $2i-1$.  This follows easily from the definition of the weight of a domino tiling in the purple-gray and white-pink models in Section \ref{two-models-section}.

In the purple-gray model, the power of $t$ is the number of configurations of the form
\[ 
 \resizebox{!}{.8cm}{
\begin{tikzpicture}[baseline = (current bounding box).center]
\draw[lightgray] (0,1) rectangle (1,2);
\draw[lightgray,fill=lightgray] (0,0) rectangle (1,1);
\draw[lightgray] (-1,0) rectangle (0,1);
\draw[very thick, blue] (0,1.5)--(1,0.5);
\draw[very thick, red] (-1,0.5)--(1,0.5);
\end{tikzpicture},
\;\;\;
 \begin{tikzpicture}[baseline = (current bounding box).center]
\draw[lightgray,fill=lightgray] (0,1) rectangle (1,2);
\draw[lightgray] (0,0) rectangle (1,1);
\draw[very thick, blue] (0,0.5)--(1,1.5);
\draw[very thick, red] (0,0.55)--(1,1.55);
\end{tikzpicture},
\;\;\;
 \begin{tikzpicture}[baseline = (current bounding box).center]
\draw[lightgray,fill=lightgray] (0,1) rectangle (1,2);
\draw[lightgray] (0,0) rectangle (1,1);
\draw[lightgray] (-1,1) rectangle (0,2);
\draw[very thick, blue] (-1,1.5)--(1,1.5);
\draw[very thick, red] (0,0.5)--(1,1.5);
\end{tikzpicture},
\;\;\;
 \begin{tikzpicture}[baseline = (current bounding box).center]
\draw[lightgray,fill=lightgray] (0,1) rectangle (1,2);
\draw[lightgray] (0,0) rectangle (1,1);
\draw[lightgray] (0,2) rectangle (1,3);
\draw[very thick, blue] (0,2.5)--(1,1.5);
\draw[very thick, red] (0,0.5)--(1,1.5);
\end{tikzpicture}
}
\] 
where blue is a smaller color than red.  In other words, we get a factor of $t$ when a blue path meets a red path from above, or when a blue and a red path take an up-right step together.  This follows easily from the definition of the weight of a $k$-tiling in the purple-gray model in Section \ref{interactions-subsection}.

\subsection{The case $t=0$} \label{t-0-section}

When $t=0$, we have
\[ Z_{AD}^{(k)}(X_m;Y_m;0) = \prod_{l=0}^{k-1}\prod_{i\le j}\left(1+x_iy_j0^l\right) = \prod_{i\le j}\left(1+x_iy_j\right) = Z_{AD}^{(1)}(X_m;Y_m). \]
In this section, we will prove $Z_{AD}^{(k)}(X_m;Y_m;0) = Z_{AD}^{(1)}(X_m;Y_m)$ combinatorially, by constructing a weight-preserving bijection between $k$-tilings of the Aztec diamond with $t=0$ and domino tilings of the Aztec diamond.  This bijection can be expressed quite nicely in terms of Schr\"oder paths.  Label the starting and ending points of the paths as follows.
\[
\resizebox{0.3\textwidth}{!}{
\begin{tikzpicture}[baseline = (current bounding box).center]
\checkerboard{3}
\draw[fill=black] (0,-0.5) circle (5pt); \node[scale=2,left] at (0,-0.5) {$1$};
\draw[fill=black] (1,-1.5) circle (5pt); \node[scale=2,left] at (1,-1.5) {$2$};
\draw[fill=black] (2,-2.5) circle (5pt); \node[scale=2,left] at (2,-2.5) {$\ddots$};
\draw[fill=black] (3,-3.5) circle (5pt); \node[scale=2,left] at (3,-3.5) {$m$};
\draw[fill=black] (8,-0.5) circle (5pt); \node[scale=2,right] at (8,-0.5) {$1$};
\draw[fill=black] (7,-1.5) circle (5pt); \node[scale=2,right] at (7,-1.5) {$2$};
\draw[fill=black] (6,-2.5) circle (5pt); \node[scale=2,right] at (6,-2.5) {$\iddots$};
\draw[fill=black] (5,-3.5) circle (5pt); \node[scale=2,right] at (5,-3.5) {$m$};
\node[scale=1.5] at (0,-2) {start};
\node[scale=1.5] at (8,-2) {end};
\end{tikzpicture}
}
\]
Note that starting point $i$ and ending point $i$ can be connected via $m-i+1$ horizontal steps. 

Before constructing the bijection, we need a better understanding of the behavior of the Schr\"oder paths when $t=0$.  We begin with the case $k=2$.  We let blue be color 1 and red be color 2.
\begin{prop} \label{prop:2-color-t-0-frozen}
When $t=0$, for any $2$-tiling of the rank-$m$ Aztec diamond with non-zero weight: 
\begin{enumerate}
    \item  If $i<\frac{m+1}{2}$, then the $i$-th blue path is forced to have its first $i$ steps be horizontal while the $i$-th red path is forced to have its first $i-1$ steps be horizontal.
    \item If $i=\lfloor \frac{m}{2}+1 \rfloor$, then the $i$-th blue path is forced to have all its steps be horizontal while the $i$-th red path is forced to have its first $i-1$ steps horizontal.
    \item If $i\ge \lceil \frac{m}{2}+1 \rceil$, then the $i$-th blue path is forced to have all its steps be horizontal and the $i$-th red path is forced to have all its steps horizontal.
\end{enumerate}
In other words, the $i$-th blue path starts with $\min(i,m-i+1)$ horizontal steps, and the $i$-th red path starts with $\min(i-1,m-i+1)$ horizontal steps.
\end{prop}
\begin{proof} We begin with three important observations.
\begin{enumerate}
    \item The $i$-th and $j$-th paths of the same color may not intersect for $i \neq j$.  This implies (by a simple induction argument from the $m$-th path to the 1st path) that the $i$-th path of each color may not go below the horizontal line connecting starting point $i$ and ending point $i$.
    \item The $i$-th blue path and the $j$-th red path may not intersect for $i < j$ when $t=0$.  The $i$-th blue path starts above the $j$-th red path, so if they did intersect, then at the first point $P$ of intersection, the blue path would meet the red path from above, which gives a $t$.
    \item Suppose a blue path and a red path meet at two points $A,Z$ with $A$ left of $Z$.  Let's consider the behavior of the two paths at $A$.  If both paths go up-right, then we get a $t$.  If the blue path goes up-right and the red path goes horizontal or down-right, then the blue path is above the red path, but the two paths both reach $Z$ later, so eventually the two paths will meet, hence we will get a $t$.  Therefore, when $t=0$, the blue path must not go up-right at $A$.
\end{enumerate}

The 1st blue path and the 1st red path start at the same point and end at the same point.  Therefore, by observation 3, the 1st blue path must start with a horizontal step.

Now assume the proposition holds for the first $i-1$ paths.  We will show the proposition holds for the $i$-th paths.  Suppose $i < m-i+1$.  We know the $i-1$-th blue path begins with $i-1$ horizontal steps.  Thus the $i$-th blue path must begin with $i-1$ horizontal steps by observation 1, and moreover the $i$-th red path must begin with $i-1$ horizontal steps by observation 2.  Since the $i$-th blue path and the $i$-th red path begin by taking $i-1$ horizontal steps together, the $i$-th blue path must take another horizontal step.  (It can't go up-right by observation 3, and it can't go down-right by observation 1.)  Thus the $i$-th blue path begins with $i$ horizontal steps.  If we suppose instead that $i \geq m-i+1$, then the same argument works, except the paths are forced to take all their steps horizontally.
\end{proof}

\begin{cor} \label{cor:2-color-t-0}
When $t=0$, for any $2$-tiling of the rank $m$ Aztec diamond with non-zero weight, the $i$-th blue path is weakly below the $i$-th red path and strictly above the $(i+1)$-th red path.
\end{cor}

\begin{proof}
The fact that the $i$-th blue path is strictly above the $(i+1)$-th red path follows from observation 2 (and the fact that the $i$-th blue path starts above the $(i+1)$-th red path).  If the $i$-th blue path were ever strictly above the $i$-th red path, then since these paths end at the same point, there must be a point where the $i$-th blue path meets the $i$-th red path from above, giving a $t$.
\end{proof}

\noindent As an example of Proposition \ref{prop:2-color-t-0-frozen}, when $m=4$, the frozen paths are as follows.
\[
\resizebox{0.3\textwidth}{!}{
\begin{tikzpicture}[baseline = (current bounding box).center]
\checkerboard{3}
\draw[fill=black] (0,-0.5) circle (5pt); \node[scale=2,left] at (0,-0.5) {$1$};
\draw[fill=black] (1,-1.5) circle (5pt); \node[scale=2,left] at (1,-1.5) {$2$};
\draw[fill=black] (2,-2.5) circle (5pt); \node[scale=2,left] at (2,-2.5) {$3$};
\draw[fill=black] (3,-3.5) circle (5pt); \node[scale=2,left] at (3,-3.5) {$4$};
\draw[fill=black] (8,-0.5) circle (5pt); \node[scale=2,right] at (8,-0.5) {$1$};
\draw[fill=black] (7,-1.5) circle (5pt); \node[scale=2,right] at (7,-1.5) {$2$};
\draw[fill=black] (6,-2.5) circle (5pt); \node[scale=2,right] at (6,-2.5) {$3$};
\draw[fill=black] (5,-3.5) circle (5pt); \node[scale=2,right] at (5,-3.5) {$4$};
\draw[very thick, blue] (0,-0.5)--(2,-0.5);
\draw[very thick, blue] (1,-1.52)--(5,-1.52);
\draw[very thick, red] (1,-1.48)--(3,-1.48);
\draw[very thick, blue] (2,-2.52)--(6,-2.52);
\draw[very thick, red] (2,-2.48)--(6,-2.48);
\draw[very thick, blue] (3,-3.52)--(5,-3.52);
\draw[very thick, red] (3,-3.48)--(5,-3.48);
\end{tikzpicture}
}
\]

\noindent We are now ready to construct the bijection in the case $k=2$.

\begin{prop}
There is a weight-preserving bijection between 2-tilings of the rank $m$ Aztec diamond at $t=0$ and domino tilings of the rank $m$ Aztec diamond, given by shifting the $i$-th blue path down $i$ steps and left $i$ steps, and shifting the $i$-th red path down $i-1$ steps and left $i-1$ steps.
\end{prop}

\begin{proof}

By Proposition \ref{prop:2-color-t-0-frozen}, it is easy to see that after shifting the paths, the frozen part of each path is shifted completely outside the Aztec diamond and the non-frozen part of each path remains inside the Aztec diamond.  Since the $i$-th blue path is weakly below the $i$-th red path and strictly above the $(i+1)$-th red path before the shift by Corollary \ref{cor:2-color-t-0}, after the shift it is strictly below the $i$-th red path and strictly above the $(i+1)$-th red path.  That is, now the paths are non-intersecting.  Non-intersecting Schr\"oder paths are in bijection with domino tilings of the Aztec diamond.  Since a horizontal step has a weight of 1, it follows that the bijection is weight-preserving.
\end{proof}

An example of the bijection for two colors is given in Figure \ref{ex:dom}. The bijection can also be defined directly on the tilings.
See Figure \ref{ex:dom1}.

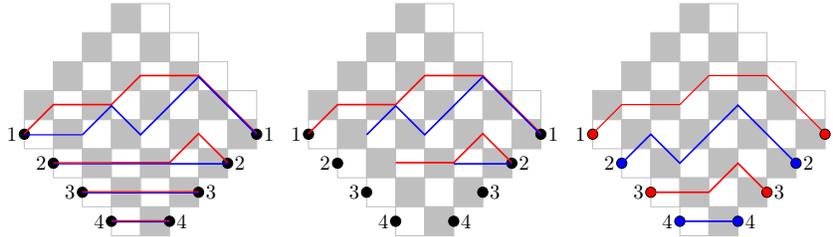
\begin{figure}[ht]
\begin{center}
\resizebox{0.75\textwidth}{!}{
\begin{tikzpicture}[baseline = (current bounding box).center]
\checkerboard{3}
\draw[fill=black] (0,-0.5) circle (5pt); \node[scale=2,left] at (0,-0.5) {$1$};
\draw[fill=black] (1,-1.5) circle (5pt); \node[scale=2,left] at (1,-1.5) {$2$};
\draw[fill=black] (2,-2.5) circle (5pt); \node[scale=2,left] at (2,-2.5) {$3$};
\draw[fill=black] (3,-3.5) circle (5pt); \node[scale=2,left] at (3,-3.5) {$4$};
\draw[fill=black] (8,-0.5) circle (5pt); \node[scale=2,right] at (8,-0.5) {$1$};
\draw[fill=black] (7,-1.5) circle (5pt); \node[scale=2,right] at (7,-1.5) {$2$};
\draw[fill=black] (6,-2.5) circle (5pt); \node[scale=2,right] at (6,-2.5) {$3$};
\draw[fill=black] (5,-3.5) circle (5pt); \node[scale=2,right] at (5,-3.5) {$4$};
\draw[ultra thick, red] (0,-0.48)--(1,0.52)--(3,0.52)--(4,1.52)--(6,1.52)--(8,-.48);
\draw[ultra thick, red] (3,-1.48)--(5,-1.48)--(6,-0.48)--(7,-1.48);
\draw[ultra thick, blue] (2,-0.52)--(3,0.48)--(4,-0.52)--(6,1.48)--(8,-.52);
\draw[ultra thick, blue] (5,-1.52)--(7,-1.52);
\draw[very thick, blue] (0,-0.52)--(2,-0.52);
\draw[very thick, blue] (1,-1.52)--(5,-1.52);
\draw[very thick, red] (1,-1.48)--(3,-1.48);
\draw[very thick, blue] (2,-2.52)--(6,-2.52);
\draw[very thick, red] (2,-2.48)--(6,-2.48);
\draw[very thick, blue] (3,-3.52)--(5,-3.52);
\draw[very thick, red] (3,-3.48)--(5,-3.48);
\end{tikzpicture}
\begin{tikzpicture}[baseline = (current bounding box).center]
\checkerboard{3}
\draw[fill=black] (0,-0.5) circle (5pt); \node[scale=2,left] at (0,-0.5) {$1$};
\draw[fill=black] (1,-1.5) circle (5pt); \node[scale=2,left] at (1,-1.5) {$2$};
\draw[fill=black] (2,-2.5) circle (5pt); \node[scale=2,left] at (2,-2.5) {$3$};
\draw[fill=black] (3,-3.5) circle (5pt); \node[scale=2,left] at (3,-3.5) {$4$};
\draw[fill=black] (8,-0.5) circle (5pt); \node[scale=2,right] at (8,-0.5) {$1$};
\draw[fill=black] (7,-1.5) circle (5pt); \node[scale=2,right] at (7,-1.5) {$2$};
\draw[fill=black] (6,-2.5) circle (5pt); \node[scale=2,right] at (6,-2.5) {$3$};
\draw[fill=black] (5,-3.5) circle (5pt); \node[scale=2,right] at (5,-3.5) {$4$};
\draw[ultra thick, red] (0,-0.48)--(1,0.52)--(3,0.52)--(4,1.52)--(6,1.52)--(8,-.48);
\draw[ultra thick, red] (3,-1.48)--(5,-1.48)--(6,-0.48)--(7,-1.48);
\draw[ultra thick, blue] (2,-0.52)--(3,0.48)--(4,-0.52)--(6,1.48)--(8,-.52);
\draw[ultra thick, blue] (5,-1.52)--(7,-1.52);
\end{tikzpicture}
\begin{tikzpicture}[baseline = (current bounding box).center]
\checkerboard{3}
\draw[fill=red] (0,-0.5) circle (5pt); \node[scale=2,left] at (0,-0.5) {$1$};
\draw[very thick, red] (0,-0.48)--(1,0.52)--(3,0.52)--(4,1.52)--(6,1.52)--(8,-.48);
\draw[fill=blue] (1,-1.5) circle (5pt); \node[scale=2,left] at (1,-1.5) {$2$};
\draw[fill=red] (2,-2.5) circle (5pt); \node[scale=2,left] at (2,-2.5) {$3$};
\draw[fill=blue] (3,-3.5) circle (5pt); \node[scale=2,left] at (3,-3.5) {$4$};
\draw[fill=red] (8,-0.5) circle (5pt); \node[scale=2,right] at (8,-0.5) {$1$};
\draw[fill=blue] (7,-1.5) circle (5pt); \node[scale=2,right] at (7,-1.5) {$2$};
\draw[fill=red] (6,-2.5) circle (5pt); \node[scale=2,right] at (6,-2.5) {$3$};
\draw[fill=blue] (5,-3.5) circle (5pt); \node[scale=2,right] at (5,-3.5) {$4$};
\draw[ultra thick, blue](1,-1.5)--(2,-0.5)--(3,-1.5)--(5,0.5)--(7,-1.5);
\draw[ultra thick, blue] (3,-3.5)--(5,-3.5);
\draw[ultra thick, red] (2,-2.5)--(4,-2.5)--(5,-1.5)--(6,-2.5);
\end{tikzpicture}}
\end{center}
\caption{Example of the bijection at $t=0$ for $k=2$. On the left, the 2-tuple of paths, in the middle the paths without the frozen steps and on the right the 1-tuple of paths}
\label{ex:dom}
\end{figure}

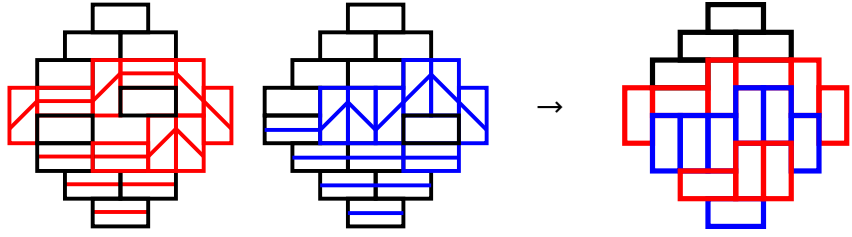
\begin{figure}[ht]
\begin{center}
\resizebox{0.75\textwidth}{!}{
\begin{tikzpicture}[baseline = (current bounding box).center, scale=0.3]

\draw[very thick, red] (0,-1) rectangle (1,1);
\draw[very thick] (1,1) rectangle (3,2);
\draw[very thick] (2,2) rectangle (4,3);
\draw[very thick] (3,3) rectangle (5,4);
\draw[very thick] (4,2) rectangle (6,3);

\draw[very thick, red] (1,0) rectangle (3,1);
\draw[very thick, red] (3,0) rectangle (4,2);
\draw[very thick, red] (4,1) rectangle (6,2);
\draw[very thick, red] (7,0) rectangle (6,2);
\draw[very thick, red] (7,1) rectangle (8,-1);

\draw[very thick, red] (0,-0.48)--(1,0.52)--(3,0.52)--(4,1.52)--(6,1.52)--(8,-.48);
\draw[very thick, red] (3,-1.48)--(5,-1.48)--(6,-0.48)--(7,-1.48);
\draw[very thick, red] (1,-1.48)--(3,-1.48);
\draw[very thick, red] (2,-2.48)--(6,-2.48);
\draw[very thick, red] (3,-3.48)--(5,-3.48);
\draw[very thick] (1,-1) rectangle (3,0);
\draw[very thick] (1,-2) rectangle (3,-1);
\draw[very thick] (2,-3) rectangle (4,-2);
\draw[very thick] (6,-3) rectangle (4,-2);

\draw[very thick, red] (3,-2) rectangle (5,-1);
\draw[very thick, red] (5,0) rectangle (6,-2);
\draw[very thick, red] (6,0) rectangle (7,-2);

\draw[very thick] (3,-4) rectangle (5,-3);

\draw[very thick] (4,0) rectangle (6,1);

\end{tikzpicture}\hspace{.2cm}
\begin{tikzpicture}[baseline = (current bounding box).center, scale=0.3]

\draw[very thick] (0,0) rectangle (2,1);
\draw[very thick] (1,1) rectangle (3,2);
\draw[very thick] (5,1) rectangle (3,2);

\draw[very thick] (2,2) rectangle (4,3);
\draw[very thick] (3,3) rectangle (5,4);
\draw[very thick] (4,2) rectangle (6,3);

\draw[very thick] (0,-1) rectangle (2,0);
\draw[very thick] (1,-2) rectangle (3,-1);
\draw[very thick] (2,-3) rectangle (4,-2);
\draw[very thick] (5,-2) rectangle (3,-1);
\draw[very thick] (6,-3) rectangle (4,-2);

\draw[very thick] (3,-4) rectangle (5,-3);
\draw[very thick, blue] (5,-1.52)--(7,-1.52);
\draw[very thick, blue] (0,-0.52)--(2,-0.52);
\draw[very thick, blue] (1,-1.52)--(5,-1.52);
\draw[very thick, blue] (2,-0.52)--(3,0.48)--(4,-0.52)--(6,1.48)--(8,-.52);
\draw[very thick, blue] (3,-3.52)--(5,-3.52);
\draw[very thick, blue] (2,-2.52)--(6,-2.52);

\draw[very thick,  blue](2,-1) rectangle (3,1);
\draw[very thick,  blue](3,-1) rectangle (4,1);
\draw[very thick,  blue](4,-1) rectangle (5,1);
\draw[very thick,  blue](5,0) rectangle (6,2);
\draw[very thick,  blue](6,0) rectangle (7,2);
\draw[very thick,  blue](7,-1) rectangle (8,1);
\draw[very thick,  blue](5,-2) rectangle (7,-1);
\draw[very thick] (5,-1) rectangle (7,0);

\end{tikzpicture}\hspace{.5cm}${\huge \rightarrow}$\hspace{.5cm}
\begin{tikzpicture}[baseline = (current bounding box).center, scale=0.3]
\draw[ultra thick, red] (0,-1) rectangle (1,1);
\draw[ultra thick] (1,1) rectangle (3,2);
\draw[ultra thick] (2,2) rectangle (4,3);
\draw[ultra thick] (3,3) rectangle (5,4);
\draw[ultra thick] (4,2) rectangle (6,3);

\draw[ultra thick, red] (1,0) rectangle (3,1);
\draw[ultra thick, red] (3,0) rectangle (4,2);
\draw[ultra thick, red] (4,1) rectangle (6,2);
\draw[ultra thick, red] (7,0) rectangle (6,2);
\draw[ultra thick, red] (7,1) rectangle (8,-1);

\draw[ultra thick,  blue](1,-2) rectangle (2,0);
\draw[ultra thick,  blue](2,-2) rectangle (3,0);
\draw[ultra thick,  blue](3,-2) rectangle (4,0);
\draw[ultra thick,  blue](4,-1) rectangle (5,1);
\draw[ultra thick,  blue](5,-1) rectangle (6,1);
\draw[ultra thick,  blue](6,-2) rectangle (7,0);

\draw[ultra thick,  blue](3,-4) rectangle (5,-3);

\draw[ultra thick, red] (2,-2) rectangle (4,-3);
\draw[ultra thick, red] (4,-1) rectangle (5,-3);
\draw[ultra thick, red] (5,-1) rectangle (6,-3);

\end{tikzpicture}}
\end{center}
\caption{Example of the bijection at $t=0$ for $k=2$. On the left, the 2-tiling with the ``frozen" dominos in black. On the right, the corresponding 1-tiling. }
\label{ex:dom1}
\end{figure}

\begin{figure}[ht]
    \centering
    \resizebox{0.9\textwidth}{!}{
    \begin{tabular}{cc}
    \includegraphics[width=0.5\textwidth]{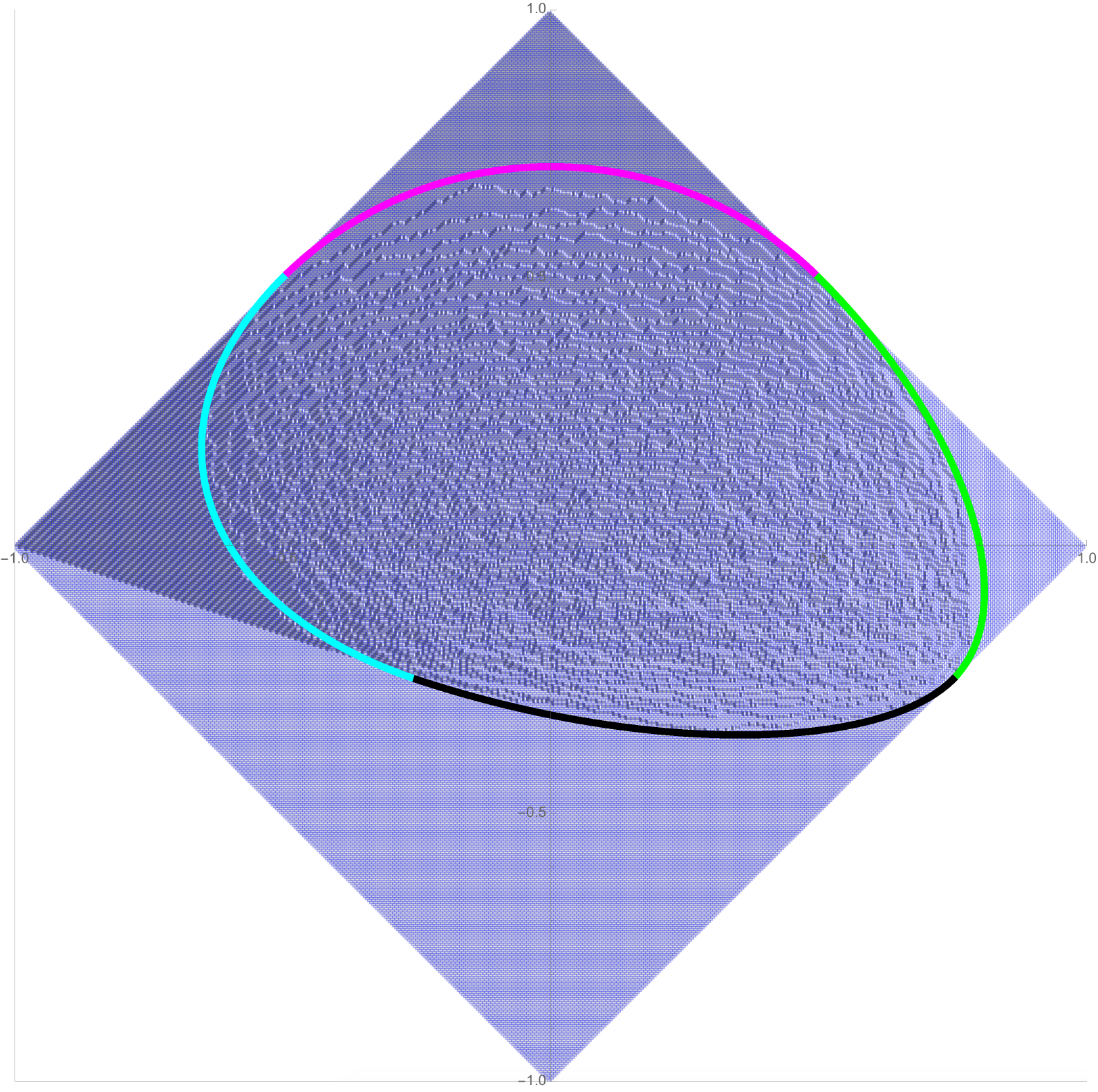} & \includegraphics[width=0.5\textwidth]{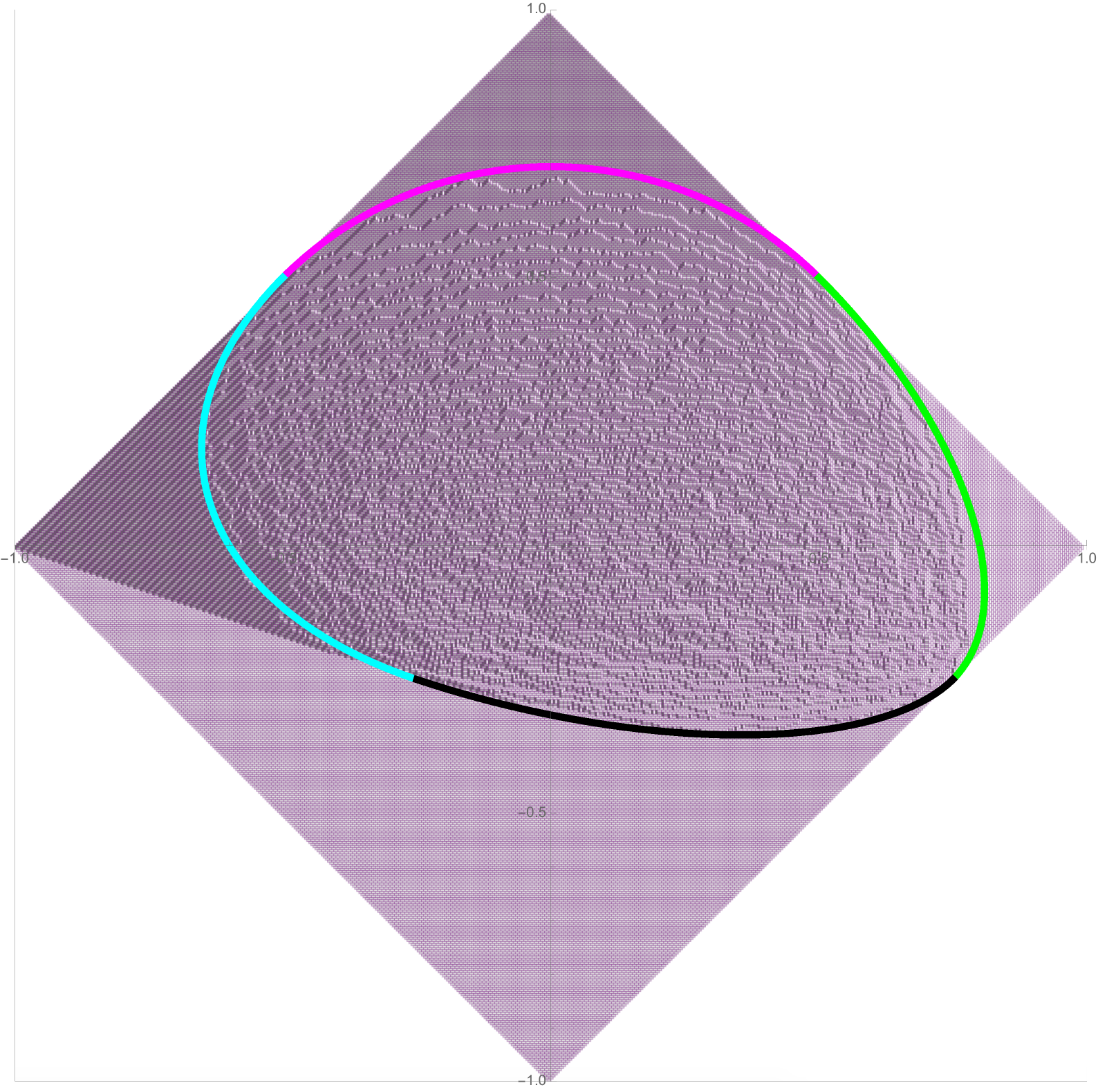}
    \end{tabular}
    }
    \caption{Simulation and computed arctic curve for a 2-tiling of the Aztec diamond of rank 256 for $t=0$. The blue tiling is the smaller color while the red tiling is the larger color.}
    \label{fig:ArcticCurveDomino}
\end{figure}

With this we can compute the arctic curve when $t=0$ (see Figure \ref{fig:ArcticCurveDomino}). We use Theorem \ref{arcticthm}.
For each $m$,
consider a uniformly random $k$-tiling of the Aztec diamond of rank $m$ and we scale each tiling by a factor $1/m$
in each axis to fit into the limiting diamond
$$
AD_\infty=\{|x| + |y| \le 1\}.
$$
We fix $k=2$.
\begin{thm}\label{thm:2tilingAC}
The arctic curve for 2-tilings of the Aztec diamond when $t=0$ is given by
\[
\begin{cases}
x^2+y^2=\frac{1}{2},& x\in[-1/2,1/2],\; y>1/2 \\
(x+y)^2+(2y)^2=\frac{1}{2},& x\in[-1/4,3/4],\; y<-1/4 \\
\left(\frac{3x+y-1}{2}\right)^2+\left(\frac{3y+x-1}{2}\right)^2=\frac{1}{2},& y\in[-1/4,1/2],\; x>-\frac{1}{3}y+\frac{2}{3} \\
\left(\frac{3x+y-1}{4}\right)^2+\left(\frac{5y-x-1}{4}\right)^2=\frac{1}{2},& y\in[-1/4,1/2],\; x<-\frac{1}{3}y-\frac{1}{3}
\end{cases}
\]
for both colors.
\end{thm}

\begin{proof} From Theorem \ref{arcticthm},
we know that for the normal Aztec diamond the arctic curve is 
the circle $x^2+y^2=1/2$. Reversing the bijection in the previous Proposition determines how one should deform the circle to get the arctic curve for the 2-tilings of the Aztec diamond of rank $m$ when $m\rightarrow \infty$. (Each piece of the arctic circle becomes a piece of a different ellipse.) 

For example, in terms if the Schr\"oder paths, the upper portion of the arctic curve separates the region of no paths from the disordered region inside the arctic curve. This boundary is determined by the asymptotic trajectory of the upper most path. As this path doesn't shift under our bijection, the portion of the arctic curve remains the same.
This gives us the first region of the Theorem.

Now consider the western portion of the arctic curve. The section of the arctic curve separates the region of up-right paths in the standard Aztec diamond from the disordered region. Recall that the $2i$-th path of the 1-tiling of the Aztec diamond maps to the $i$-th blue path and the $(2i-1)$-th path of the 1-tiling of the Aztec diamond maps to the $i$-th red path. Reversing the bijection means shifting these paths up and right $i$ steps or $i-1$ steps, respectively.

Suppose our Aztec diamond is rank $m$ and we rescale it by a factor $1/m$ in each axis.
Now each square of our checkerboard has size $\frac{1}{m} \times \frac{1}{m}$ as in Theorem \ref{arcticthm}. Then the starting location of the $2i$-th path is at the coordinate 
\[
(x_{2i},y_{2i}) = \left(-1+\frac{2i-1}{m}, - \frac{1}{2m} - \frac{2i-1}{m} \right).
\]
Reversing the bijection the will shift the $2i$-th path of the 1-tiling of the Aztec diamond up and right by $i/m = (x_{2i} - y_{2i} + 1)/4 + O(\frac{1}{m})$ where it will become the $i$-th path of color blue in the $2$-tiling of the  Aztec diamond. In fact, since we are considering the frozen region of up-right paths, any point $(x,y)$ along the $2i$-th path will also shift by $(x-y+1)/4+O(\frac{1}{m})$. The same holds for the $(2i-1)$-th path, except that it will become the $i$-th path of color red in the $2$-tiling of the Aztec diamond.

Now we take $m\to \infty$. With this choice of coordinates, any point $(x,y)$ in this region of up-right paths in the 1-tiling of the Aztec diamond maps to a point in the red or blue Aztec diamond according to
\[
(x,y) \mapsto \left(x+\frac{x-y+1}{4}, y+\frac{x-y+1}{4}\right) = \left(\frac{5x-y+1}{4},\frac{x+3y+1}{4}\right). 
\]
Since the arctic curve separating the region of up-right paths from the disordered region in the $1$-tiling is given by $x^2+y^2=\frac{1}{2}$ with $x<-\frac{1}{2}, -\frac{1}{2}<y<\frac{1}{2}$, inverting the above map, we see that the arctic curve in the $2$-tiling is given by
\[
\left(\frac{3x+y-1}{4}\right)^2+\left(\frac{5y-x-1}{4}\right)^2=\frac{1}{2}
\]
with $-1<\frac{3x+y-1}{4}<-\frac{1}{2}$ and $-\frac{1}{2}<\frac{5y-x-1}{4}<\frac{1}{2}$, for both colors. Simplifying the constraints gives the last region in the theorem.

The remaining two portions of the arctic curve can be worked out similarly.
\end{proof}

It is straightforward to generalize our discussion for 2-tilings to $k$-tilings. We start with the generalization of Proposition \ref{prop:2-color-t-0-frozen}.

\begin{prop}\label{prop:k-color-t-0-frozen}
    Let $a\in\{1,\ldots,k\}$. When $t=0$, for any $k$-tiling of the rank-m Aztec diamond with non-zero weight, we have that the $i$-th path of color $a$ is forced to have its first $i(k-1)+a-1$ steps be horizontal. If the total number of steps is less than that amount, then all its steps are horizontal.
\end{prop}
\begin{proof}
    This follows from the exact same reasoning as in Prop. \ref{prop:2-color-t-0-frozen}. We will induct on $i$, and for each $i$ we will induct downward on $a$,
    
    As in Prop. \ref{prop:2-color-t-0-frozen} (replacing 2 with $k$ and 1 with $k-1$) the first paths of color $k$ and $k-1$ start an end at the same location, the path of color $k-1$ must go horizontal on its first step. 
    
    Let $a\ne k$ and consider the $i$-th path of color $a$. Suppose the result is true for $i$-th path of color $a+1$. Since the paths start and end at the same location, the path of color $a$ cannot go above the path of color $a+1$ without introducing an interaction when they eventually must return to the same height. Thus the path of color $a$ must have at least the same number of forced horizontal steps as the path of color $a+1$. But now the paths are at the same location and the same reasoning from Prop. \ref{prop:2-color-t-0-frozen} shows that the smaller color must have an additional forced horizontal step.

    Now let $a=k$, consider it's $i$-th path, and suppose the result is true for the $(i-1)$-th path of color $1$. If the path of color $k$ hits the path of color $1$ from below, it would introduce an interaction. Since the path of color $k$ starts one row below the path of color $1$, it cannot go up while the path of color $1$ is horizontal. Thus the path of color $k$ must have the same number of forced horizontal steps, as desired.
\end{proof}

From the above, we end up with the following bijection.

\begin{prop} For any $k\ge 1$,
there is a weight-preserving bijection between $k$-tilings of the rank-$m$ Aztec diamond at $t=0$ and domino tilings of the rank-$m$ Aztec diamond, given by shifting the $i$-th path of color $a$ down $i(k-1)-a+1$ steps and left $i(k-1)-a+1$ steps.
\end{prop}


For $k$-tilings, we can then compute the arctic curve for $t=0$, given in Theorem \ref{karctic}, which we restate here.

\begin{thm}
    The arctic curve for $k$-tilings of the Aztec diamond when $t=0$ is given by
    \[
    \begin{cases}
    x^2+y^2=\frac{1}{2},& x\in[-1/2,1/2],\; y>1/2 \\
    (x+(k-1)y)^2+(k y)^2=\frac{1}{2},& x\in[-1/(2k),1 - 1/(2k)],\; y<-1/(2k) \\
    \left(\frac{2x+(k-1)(x+y-1)}{2}\right)^2+\left(\frac{2y+(k-1)(x+y-1)}{2}\right)^2=\frac{1}{2},& y\in[-1/(2k),1/2],\; x>-\frac{k-1}{k+1}y+\frac{k}{k+1} \\
    \left(\frac{2x+(k-1)(x+y-1)}{2k}\right)^2+\left(\frac{2y+(k-1)(3y-x-1)}{2k}\right)^2=\frac{1}{2},& y\in[-1/(2k),1/2],\; x<-\frac{k-1}{k+1}y-\frac{1}{k+1}
    \end{cases}
    \]
    for each color.
\end{thm}
\begin{proof}
Suppose we have $k$ colors and let $a\in\{1,\ldots,k\}$. Under the bijection, the $(ik-a+1)$-th path in the 1-color Aztec diamond corresponds to the $i$-th path of color $a$ Aztec diamond. It has starting and ending points 
\[
\begin{aligned}
(x_{ik-a+1},y_{ik-a+1}) &\;= \left(-1+\frac{ik-a}{m}, -\frac{1}{2m}-\frac{ik-a}{m}\right) \\
(w_{ik-a+1},z_{ik-a+1}) &\;= \left(1-\frac{ik-a}{m}, -\frac{1}{2m}-\frac{ik-a}{m}\right) 
\end{aligned}
\]
respectively. Moreover, $(ik-a+1)$-th path in the 1-color Aztec diamond will shift by $\frac{i(k-1)-a+1}{m}$ units up and to the right when mapping from the rescaled 1-color Aztec diamond to the color $a$ Aztec diamond.

Note that it is enough for us only to consider only what happens to the smallest color $a=1$. If we were to consider paths corresponding to a different color, the difference in starting location and shift amount would only result in an order $\frac{1}{m}$ difference in the results.

Now we proceed as in the 2-color case: for each type of frozen region, we find a point $(x,y)$ on a path corresponding to the smallest color. We then see where that point maps after shifting the paths. Since, as $m\to \infty$, points on the arctic circle will satisfy $x^2+y^2=\frac{1}{2}$, we can deduce the location of the arctic circle in the coupled tilings. Let $(\tilde x,\tilde y)$ be coordinates on the coupled Aztec diamond. We summarize the calculation in the table below.  

\begin{center}
\resizebox{\textwidth}{!}{
    \begin{tabular}{|c||c|c|} \hline
    Frozen Region & Shift Amount & Coordinate Change \\ \hline \hline & & \\
    Empty & N/A & $ (\tilde x, \tilde y) = (x,y)$ \\ & & \\ \hline & & \\
    Up-Right Paths & $\begin{aligned} &x-y = x_{ik}-y_{ik} \\ &\implies \frac{i(k-1)}{m} = \frac{k-1}{2k} (x-y+1) +O\left(\frac{1}{m}\right)\end{aligned}$ & $\begin{aligned} & (\tilde x,\tilde y) =\\ &\left(x+\frac{k-1}{2k}(x-y+1),y+\frac{k-1}{2k}(x-y+1) \right) \end{aligned}$ \\ & & \\ \hline & & \\
    Horizontal Paths & $\begin{aligned} & y=y_{ik} \\ & \implies \frac{i(k-1)}{m} = - \frac{k-1}{k} y +O\left(\frac{1}{m}\right) \end{aligned}$ & $(\tilde x,\tilde y)=\left(x- \frac{k-1}{k}y, \frac{1}{k} y\right)$ \\ & & \\ \hline & & \\
    Down-Right Paths & $\begin{aligned} & x+y=w_{ik}+z_{ik} \\ &\implies \frac{i(k-1)}{m} = -\frac{k-1}{2k}(x+y-1) + O\left(\frac{1}{m}\right) \end{aligned}$ & $\begin{aligned}&(\tilde x,\tilde y)=
    &\left(x-\frac{k-1}{2k}(x+y-1),y-\frac{k-1}{2k}(x+y-1)\right) \end{aligned}$  \\ & & \\ \hline 
    \end{tabular}
}
\end{center}
We now invert the coordinate change in each region. For example, in the frozen region of up-right paths we have
\[
(x,y)= \left(\frac{2\tilde x + (k-1)(\tilde x + \tilde y - 1)}{2k}, \frac{2\tilde y + (k-1)(3\tilde y - \tilde x - 1)}{2k} \right)
\]
Since this region corresponds to $x<-\frac{1}{2}$ and $-\frac{1}{2}<y<\frac{1}{2}$, we find that for the coupled Aztec diamond, in the region $\tilde x < -\frac{k-1}{k+1}\tilde y - \frac{1}{k+1}$ and $-\frac{1}{2k}<\tilde y<\frac{1}{2}$,  the arctic curve satisfies
\[
\left(\frac{2\tilde x + (k-1)(\tilde x + \tilde y - 1)}{2k}\right)^2 + \left(\frac{2\tilde y + (k-1)(3\tilde y - \tilde x - 1)}{2k}\right)^2=\frac{1}{2}
\]
This is the fourth case stated in the theorem. The other cases can be done similarly.

\end{proof}

\begin{remark}
Note that one should also be able to use this bijection to construct the full limit shape for the $k$-tilings at $t=0$, as it is a deformation of the normal Aztec diamond limit shape computed in \cite{johansson},. 
\end{remark}

\begin{remark}\label{rmk:whitepinkpaths}
    One can show analogous results for the white-pink model. In that case, one should consider Schr\"oder paths that start on the south-west boundary and end on the north-west boundary. Then, when $t=0$, the whote-pink interactions will force the paths to start with many vertical steps. Deleting these forced path segments and then translating paths down and to the left will give a bijections with tilings of the 1-color Aztec diamond. Then one could compute the $k$-color arctic curve in a similar manner to what we have done for the purple-gray model.
\end{remark}

\section{The case $t \rightarrow \infty$} \label{t-infinity-section}

In this section, we will compute the arctic curve for the $k$-tilings of the Aztec diamond as $t \rightarrow \infty$.  First we will relate the $t=0$ case to the $t \rightarrow \infty$ case, by defining a bijection between tilings with no pairs of coupled dominos (i.e. $t=0$) and tilings with the maximum possible number of coupled dominos.  Then we will apply the arctic curve computations in the $t=0$ case given in the previous section.

Let $\phi$ be the involution on the set of $k$-tilings of the  Aztec diamond of rank $m$ given by reflecting over the line $y=x$.  This involution leads to the following lemma.
\begin{lem}\label{lem:0toinf}
Let $\bm{T}$ be a $k$-tiling of the Aztec diamond of rank $m$ with $j$
pairs of coupled dominos. Then $\phi(\bm{T})$ is a $k$-tiling of the Aztec diamond of rank $m$ with ${k\choose 2}{m+1\choose 2}-j$
pairs of coupled dominos.
\end{lem}

\begin{proof} 
The dominos are of four types, as shown below.
\begin{center}
\begin{tabular}{cccc}
\resizebox{0.8cm}{!}{
\begin{tikzpicture}[baseline = (current bounding box).center]
\draw[lightgray, fill=lightgray] (0,0) rectangle (1,1);
\draw[very thick, red] (-1,0) rectangle (1,1);
\end{tikzpicture}
}
& 
\resizebox{0.4cm}{!}{
\begin{tikzpicture}[baseline = (current bounding box).center]
\draw[lightgray, fill=lightgray] (0,0) rectangle (1,1);
\draw[very thick, red] (0,-1) rectangle (1,1);
\end{tikzpicture}
}
& 
\resizebox{0.8cm}{!}{
\begin{tikzpicture}[baseline = (current bounding box).center]
\draw[lightgray, fill=lightgray] (0,0) rectangle (1,1);
\draw[very thick, red] (0,0) rectangle (2,1);
\end{tikzpicture}
}
& 
\resizebox{0.4cm}{!}{
\begin{tikzpicture}[baseline = (current bounding box).center]
\draw[lightgray, fill=lightgray] (0,0) rectangle (1,1);
\draw[very thick, red] (0,0) rectangle (1,2);
\end{tikzpicture}
}
\\
type I & type II & type III & type IV
\end{tabular}
\end{center}
In a $1$-tiling, the number of dominos of type I or II equals ${m+1\choose 2}$ for the $1$-tiling where all the dominos are horizontal.
When we perform a flip
\begin{center}
\begin{tabular}{ccccccc}
\resizebox{0.8cm}{!}{
\begin{tikzpicture}[baseline = (current bounding box).center]

\draw[lightgray, fill=lightgray] (0,0) rectangle (1,1);
\draw[lightgray, fill=lightgray] (-1,1) rectangle (0,2);
\draw[very thick, red] (-1,0) rectangle (1,1);
\draw[very thick, red] (-1,1) rectangle (1,2);
\end{tikzpicture}
}&$\leftrightarrow$
& 
\resizebox{0.8cm}{!}{
\begin{tikzpicture}[baseline = (current bounding box).center]
\draw[lightgray, fill=lightgray] (0,0) rectangle (1,1);
\draw[lightgray, fill=lightgray] (-1,1) rectangle (0,2);
\draw[very thick, red] (-1,0) rectangle (0,2);
\draw[very thick, red] (0,-0) rectangle (1,2);
\end{tikzpicture}
}
& or &
\resizebox{0.8cm}{!}{
\begin{tikzpicture}[baseline = (current bounding box).center]
\draw[lightgray, fill=lightgray] (-1,0) rectangle (0,1);
\draw[lightgray, fill=lightgray] (0,1) rectangle (1,2);
\draw[very thick, red] (-1,0) rectangle (1,1);
\draw[very thick, red] (-1,1) rectangle (1,2);
\end{tikzpicture}
}&$\leftrightarrow$
& 
\resizebox{0.8cm}{!}{
\begin{tikzpicture}[baseline = (current bounding box).center]
\draw[lightgray, fill=lightgray] (-1,0) rectangle (0,1);
\draw[lightgray, fill=lightgray] (0,1) rectangle (1,2);
\draw[very thick, red] (-1,0) rectangle (0,2);
\draw[very thick, red] (0,-0) rectangle (1,2);
\end{tikzpicture}
}
\end{tabular}
\end{center}
The number of tiles of type I or II is unchanged.
We can get all tilings starting from the tiling with all horizontal tiles and performing flips. Therefore there are 
${m+1\choose 2}$ type I or II dominos for all $1$-tilings of the Aztec diamond of rank $m$.

When applying $\phi$, the dominos of type I become dominos of type II (and vice versa), and the dominos of type III become dominos of type IV (and vice versa).  

Fix a $k$-tiling $\bm{T}$. Suppose that $\bm{T}$ has $j$ pairs of coupled dominos and that $\phi(\bm{T})$ has $\ell$ pairs of coupled dominos.  Fix two colors $a$ (blue) $<$ $b$ (red).  When the red domino is of type I, we get a power of $t$ for
\[\resizebox{0.8cm}{!}{
\begin{tikzpicture}[baseline = (current bounding box).center]
\draw[lightgray, fill=lightgray] (0,0) rectangle (1,1);
\draw[lightgray] (0,1) rectangle (1,2);
\draw[line width=1mm, blue] (0,0) rectangle (1,2);
\draw[line width=1mm, red] (-1,0) rectangle (1,1);
\end{tikzpicture}}
\]
in which case we say the red domino is in case I-A with the color blue, and no power of $t$ for 
\[
\resizebox{1.2cm}{!}{
\begin{tikzpicture}[baseline = (current bounding box).center]
\draw[lightgray, fill=lightgray] (0,0) rectangle (1,1);
\draw[line width=1mm, blue] (0,0) rectangle (2,1);
\draw[line width=1mm, red] (-1,0) rectangle (1,1);
\end{tikzpicture}},\;\;\;\;\;\;
\resizebox{0.8cm}{!}{
\begin{tikzpicture}[baseline = (current bounding box).center]
\draw[lightgray, fill=lightgray] (0,0) rectangle (1,1);
\draw[line width=1mm, blue] (0,1) rectangle (1,-1);
\draw[line width=1mm, red] (-1,0) rectangle (1,1);
\end{tikzpicture}},\;\;\;\;\;\;
\resizebox{0.8cm}{!}{
\begin{tikzpicture}[baseline = (current bounding box).center]
\draw[lightgray, fill=lightgray] (0,0) rectangle (1,1);
\draw[line width=1mm, blue] (-1.1,-0.1) rectangle (.9,.9);
\draw[line width=1mm, red] (-1,0) rectangle (1,1);
\end{tikzpicture}}
\]
in which case we say the red domino is in case I-B with the color blue.  When the red domino is of type II, we get
no power of $t$ for
\[
\resizebox{0.8cm}{!}{
\begin{tikzpicture}[baseline = (current bounding box).center]
\draw[lightgray, fill=lightgray] (0,0) rectangle (1,1);
\draw[line width=1mm, blue] (0,0) rectangle (2,1);
\draw[line width=1mm, red] (0,-1) rectangle (1,1);
\end{tikzpicture}}
\]
in which case we say the red domino is in case II-A with the color blue, and a power of $t$ for 
\[
\resizebox{0.4cm}{!}{
\begin{tikzpicture}[baseline = (current bounding box).center]
\draw[lightgray, fill=lightgray] (0,0) rectangle (1,1);
\draw[line width=1mm, blue] (0,0) rectangle (1,2);
\draw[line width=1mm, red] (0,-1) rectangle (1,1);
\end{tikzpicture}},\;\;\;\;\;\;
\resizebox{0.8cm}{!}{
\begin{tikzpicture}[baseline = (current bounding box).center]
\draw[lightgray, fill=lightgray] (0,0) rectangle (1,1);
\draw[line width=1mm, blue] (-1,0) rectangle (1,1);
\draw[line width=1mm, red] (0,-1) rectangle (1,1);
\end{tikzpicture}},\;\;\;
\resizebox{0.4cm}{!}{
\begin{tikzpicture}[baseline = (current bounding box).center]
\draw[lightgray, fill=lightgray] (0,0) rectangle (1,1);
\draw[line width=1mm, blue] (0.1,-0.9) rectangle (1.1,1.1);
\draw[line width=1mm, red] (0,-1) rectangle (1,1);
\end{tikzpicture}}
\]
in which case we say the red domino is in case II-B with the color blue.  The red dominos of type III or IV never get a power of $t$.  Therefore 
\[ \begin{aligned}
j + \ell &= \sum_{1 \leq a < b \leq k} \sum_{\substack{\text{dominos $D$} \\ \text{of color $b$}}} \textbf{1}_\text{$D$ is in case I-A or II-B with $a$} + \sum_{1 \leq a < b \leq k} \sum_{\substack{\text{dominos $D$} \\ \text{of color $b$}}} \textbf{1}_\text{$D$ is in case I-B or II-A with $a$} \\
&= \sum_{1 \leq a < b \leq k} \sum_{\substack{\text{dominos $D$} \\ \text{of color $b$}}} \textbf{1}_\text{$D$ is in case I-A, I-B, II-A, or II-B with $a$} \\
&= \sum_{1 \leq a < b \leq k} \sum_{\substack{\text{dominos $D$} \\ \text{of color $b$}}} \textbf{1}_\text{$D$ has type I or II} \\
&= \sum_{1 \leq b \leq k} (b-1) \sum_{\substack{\text{dominos $D$} \\ \text{of color $b$}}} \textbf{1}_\text{$D$ has type I or II} \\
&= \sum_{1 \leq b \leq k} (b-1) {m+1 \choose 2} \\
&= {k \choose 2} {m+1 \choose 2}.
\end{aligned} \]
\end{proof}

This gives us that $k$-tilings for $t=0$ are in bijection with the
$k$-tilings for $t\rightarrow \infty$ and that $\phi$ is this bijection.
Therefore
\begin{cor}
The case $t=0$ and $t\rightarrow\infty$ have the same limit shape (up to the reflection $y=x$).
\end{cor}
See for example Figure \ref{fig:ArcticCurveDominoLarge}.

\begin{figure}[ht]
    \centering
    \includegraphics[width=.7\textwidth]{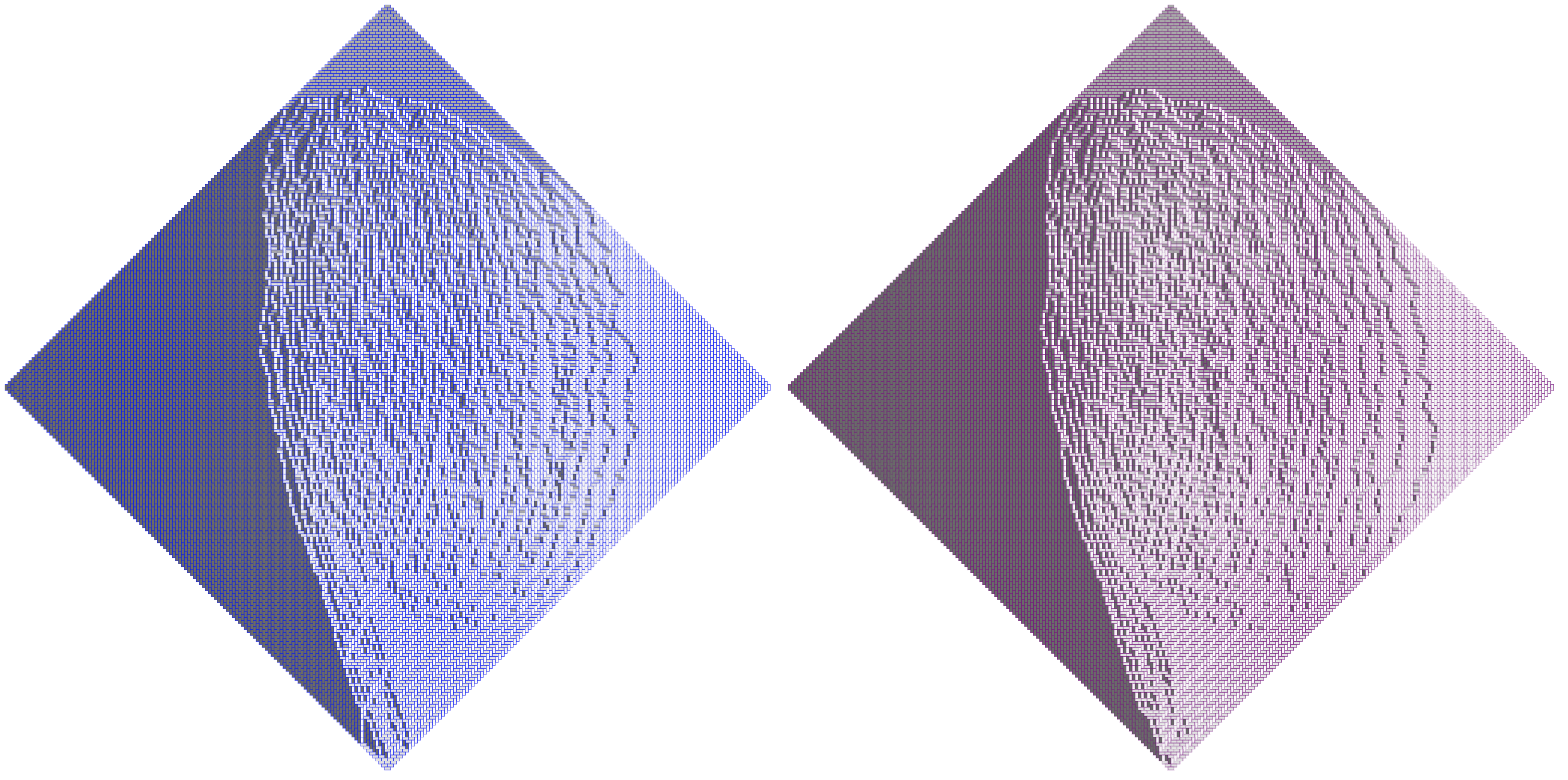} 
    \caption{Simulation of a 2-tiling Aztec diamond of rank 128 for $t=1000$.}
    \label{fig:ArcticCurveDominoLarge}
\end{figure}

\begin{figure}[ht]
    \centering
    \includegraphics[width=.7\textwidth]{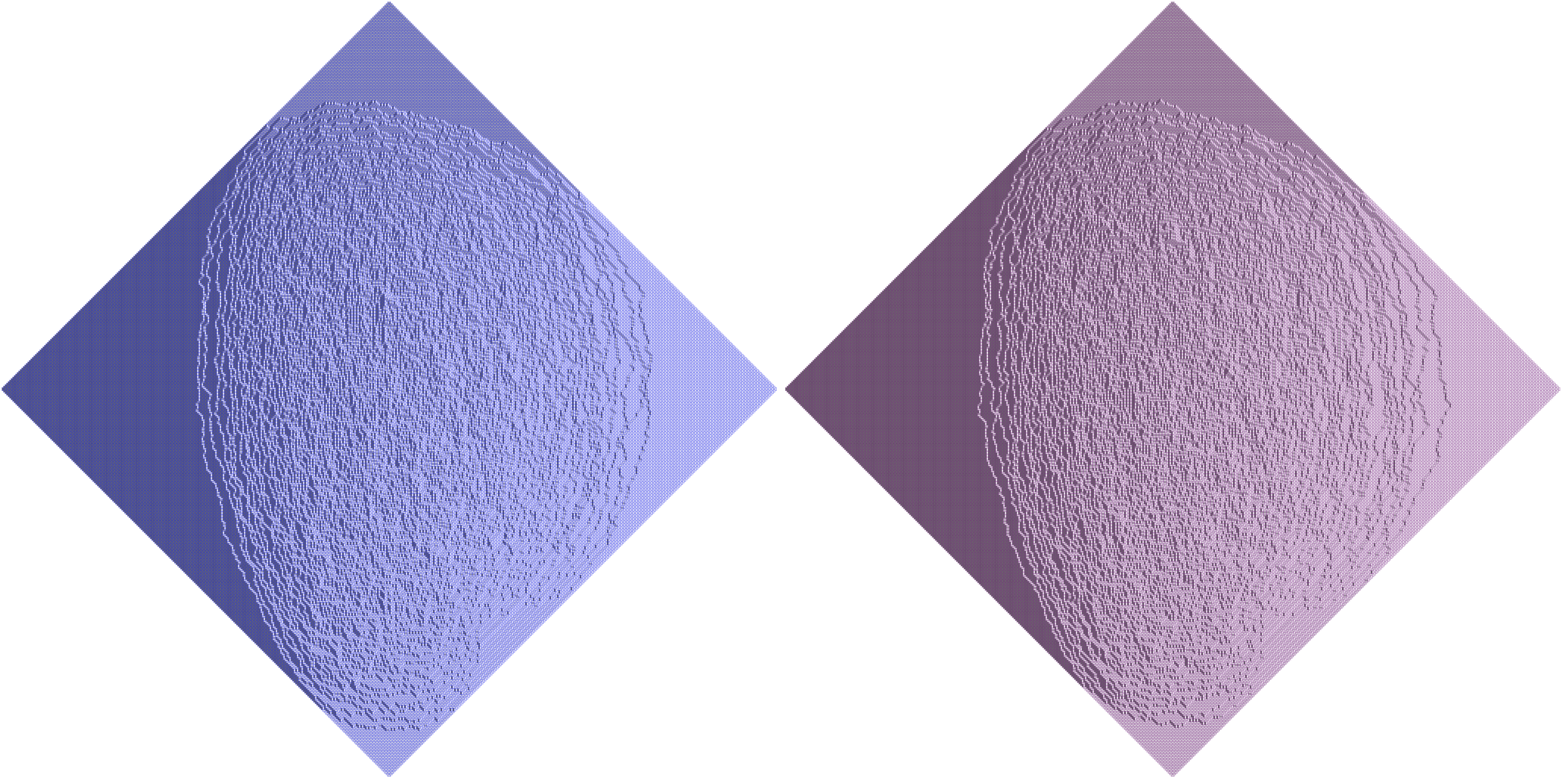} 
    \caption{Simulation of a 2-tiling Aztec diamond of rank 256 for $t=5$. Note the formation of a cusp along the south-east boundary of the Aztec diamond.}
    \label{fig:ArcticCurveDomino256}
\end{figure}

\section{Conclusion and open questions}
\label{open-section}

In this section, we list a few open questions and further results. 

Since the writing of this article a version of the domino shuffle algorithm has been created for the coupled tilings. The domino-shuffling algorithm was originally introduced in \cite{aztec0} and further studied in \cite{shuffling1,shuffling2} for counting the domino tilings of the Aztec diamond of rank $m$.  This algorithm can also be used  to generate a uniform sampling of a tiling \cite{aztec3}.
In \cite{KN23} the third author of this paper and Matthew Nicoletti generalize the shuffling algorithm to generate $k$-tilings of the 
Aztec diamond of rank $m$ such that the probability of generating a given $k$-tiling is proportional to its weight. They describe the algorithm both in terms of dynamics on a system of colored particles and as operations on the dominos themselves.

Here we list several open problems and areas of further study.

\begin{problem}
Give a combinatorial proof of Proposition \ref{prop:open}. 
\end{problem}

\begin{problem}
In this paper we compute  the arctic curves of the $k$-tilings of the Aztec diamond of rank $m$ when $t=0,1,\infty$. (The case $t=1$ is the uniform case
and we get back Theorem \ref{arcticthm}.)
Our current techniques do not generalize to any other values of $t$. It is a challenging open problem to
compute the arctic curves of the $k$-tilings of the Aztec diamond of rank $m$ when $m \rightarrow\infty$ and $t\neq 0,1$ or $\infty$. See Figure \ref{fig:ArcticCurveDomino256} for an example. In \cite[Corollary 4.4.9]{AndrewThesis}, the second author proves that if the arctic curve exists for $0\le t\le 1$ then it is the same for $1/t$ (up to the reflection $y=x$). See Figure \ref{david}.
We also conjecture that the $k$-tilings have the same arctic curve for all the colors. 
\end{problem}

\begin{figure}[ht]
    \centering
    \begin{tabular}{c}
    \includegraphics[width=.7\textwidth]{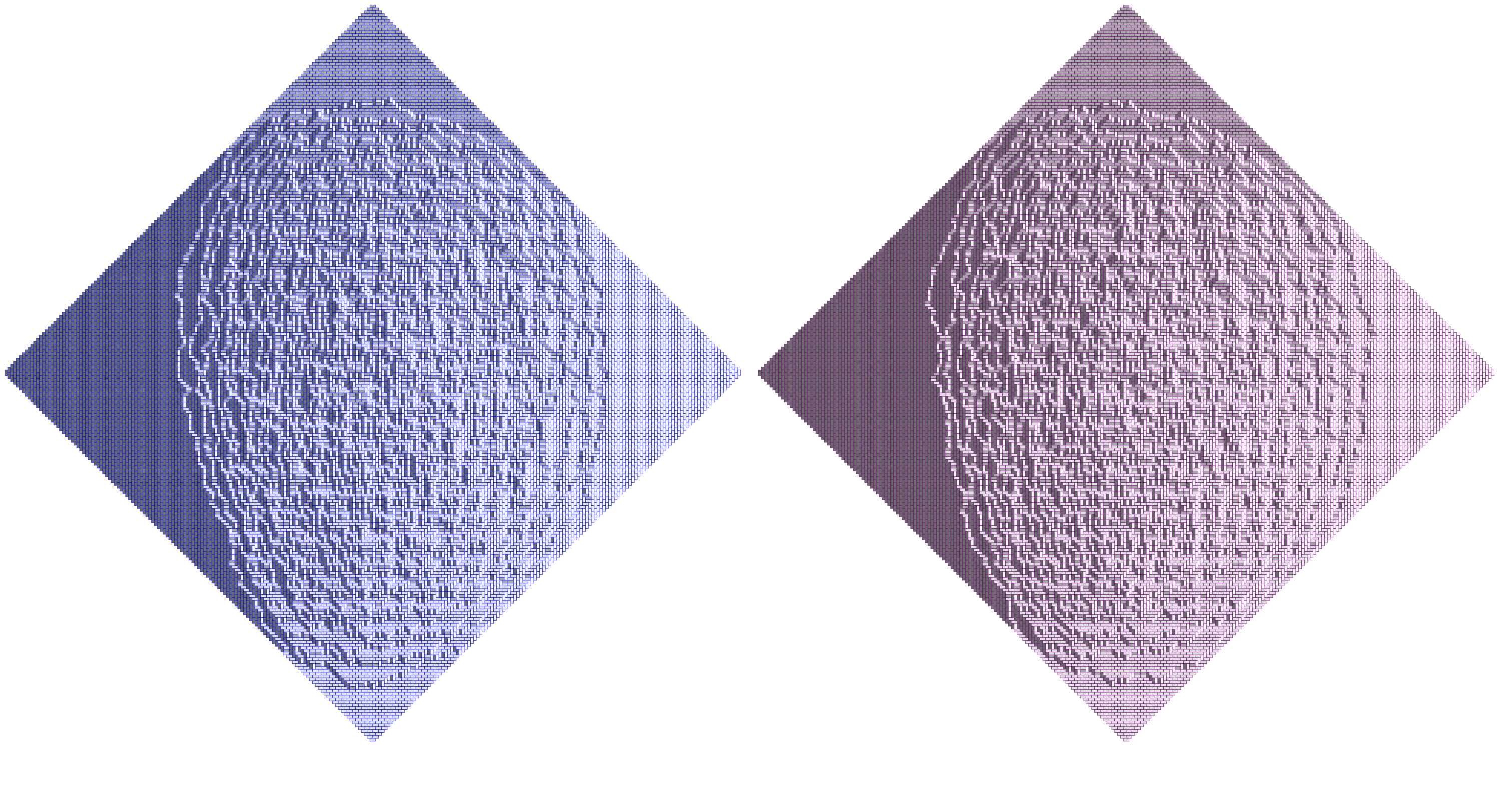} \\
    \includegraphics[width=.7\textwidth]{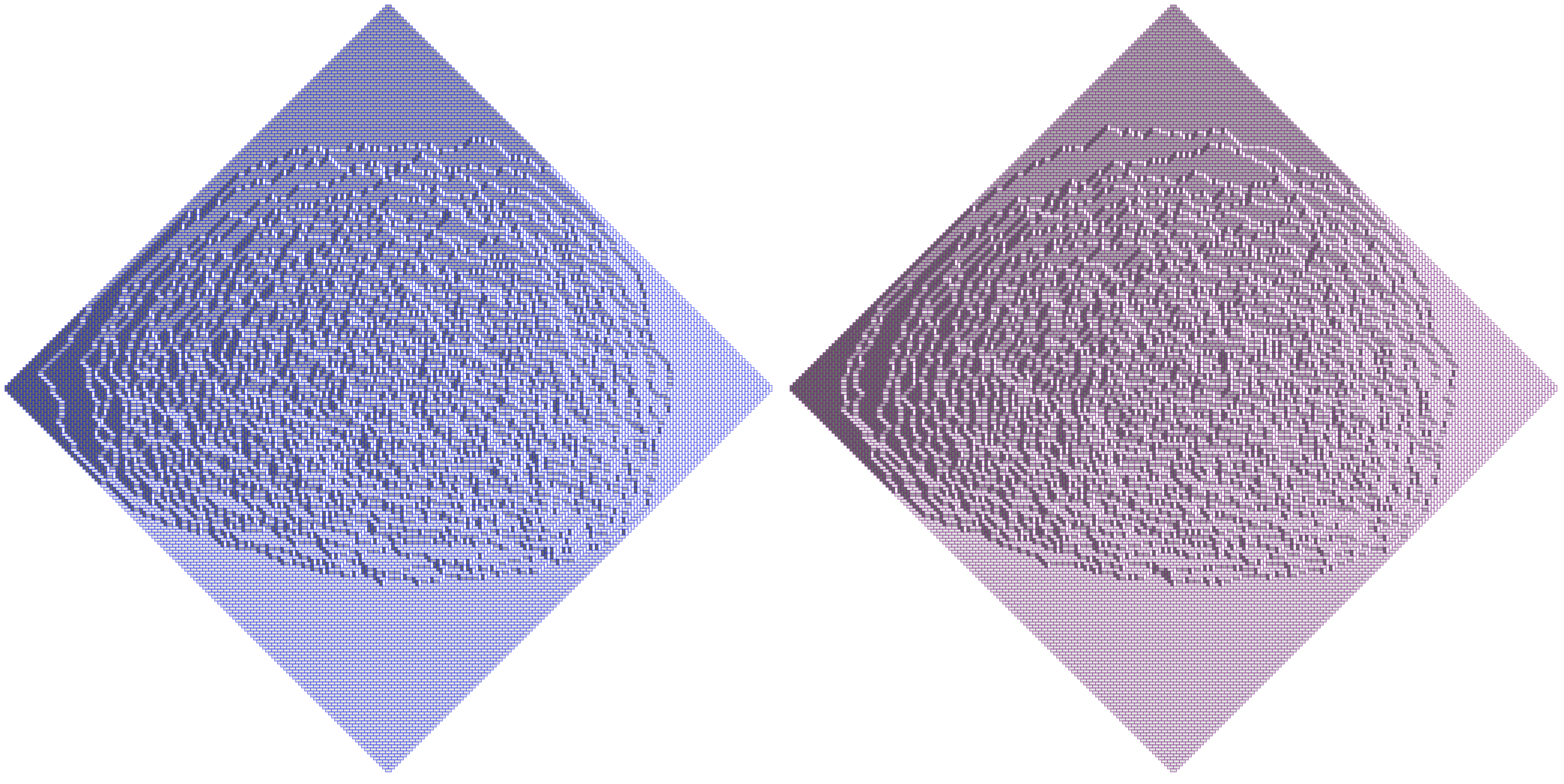}
    \end{tabular}
    \caption{Top: Simulation of a rank-128 Aztec diamond for $t=3$. Bottom: Simulation of a rank-128 Aztec diamond for $t=1/3$.}
    \label{david}
\end{figure}

\begin{problem}
In Appendix \ref{ap:a} we discuss lozenge $k$-tilings of the $a\times b\times c$ hexagon. We define their interactions using a vertex model as well (in fact, we use the white vertices with $x_i=q^{i-1}$).  When $t=1$, the generating polynomial of such tilings is equal to:
$$
q^{k{a\choose 2}b}\left(\prod_{i=1}^a\prod_{j=1}^b\frac{1-q^{c+i+j-1}}{1-q^{i+j-1}}\right)^k
$$
When $t=0$, the generating polynomial of such tilings is equal to:
$$
q^{k{a\choose 2}b}\\\prod_{i=1}^{ka}\prod_{j=1}^b\left(\frac{1-q^{c-(k-1)a+i+j-1}}{1-q^{i+j-1}}\right).
$$
The generating polynomial
of the $k$-tilings of the $a\times b\times c$ hexagon according to their number of pairs of coupled lozenges. This is a polynomial in $t$ of degree ${k\choose 2}ab$. For example the coefficient of $t^{{k\choose 2}ab}$ is
$$
q^{k{a\choose 2}b}\prod_{i=1}^{a}\prod_{j=1}^{kb}\left(\frac{1-q^{c+i+j-1}}{1-q^{i+j-1}}\right).
$$
A table for small values of $a,b,c$ and $q=1,k=2$ is presented in Table \ref{ex:tab}. Note that for $q=1$, from Remark \ref{rmk:lozengeLLT}, this polynomial is an LLT polynomial $\mathcal{L}_{(\underbrace{\lambda,\ldots,\lambda}_{k \text{ times}})}(\underbrace{1,\ldots,1}_{a+c \text{ times}};t)$ where $\lambda = (\underbrace{b,\ldots,b}_{a \text{ times}})$.
We leave as an open problem the study of these $k$-tilings of the hexagon.
\end{problem}

These (domino or lozenge) $k$-tilings come from vertex models for LLT polynomials \cite{LLT,GKsuper,ABW}. For now not much is now about the LLT process. A first study of colored corner processes was started in \cite{ABW23} and gives rise to beautiful results and open questions in combinatorics and probability. The third author and Nicoletti are working on extending these results to the case of the Aztec diamond. These are very nice first steps in the study of LLT processes.
It would be interesting to understand the big picture.

As another example, one can try to study the fluctuations of the arctic curve in the $k$-tilings. For a 1-tiling, it is known due to Johansson \cite{johansson} that the fluctuations of Schr\"oder paths are given by the Airy process. In particular, the fluctuations of the arctic curve is given by the Tracy-Widom $F_2$ distribution \cite{tracywidom}. Some numerical results on such fluctuations for 2-tilings were obtained by Lucas Allen, Braeden Bertz, and Harsha Kenchareddy as part of an REU through the Madison Experimental Mathematics Laboratory \cite{MXMreport}. The numerically computed fluctuations are shown in Figure \ref{fig:fluctuations}.

\begin{figure}
    \begin{tabular}{c}
\includegraphics[width=.8\textwidth]{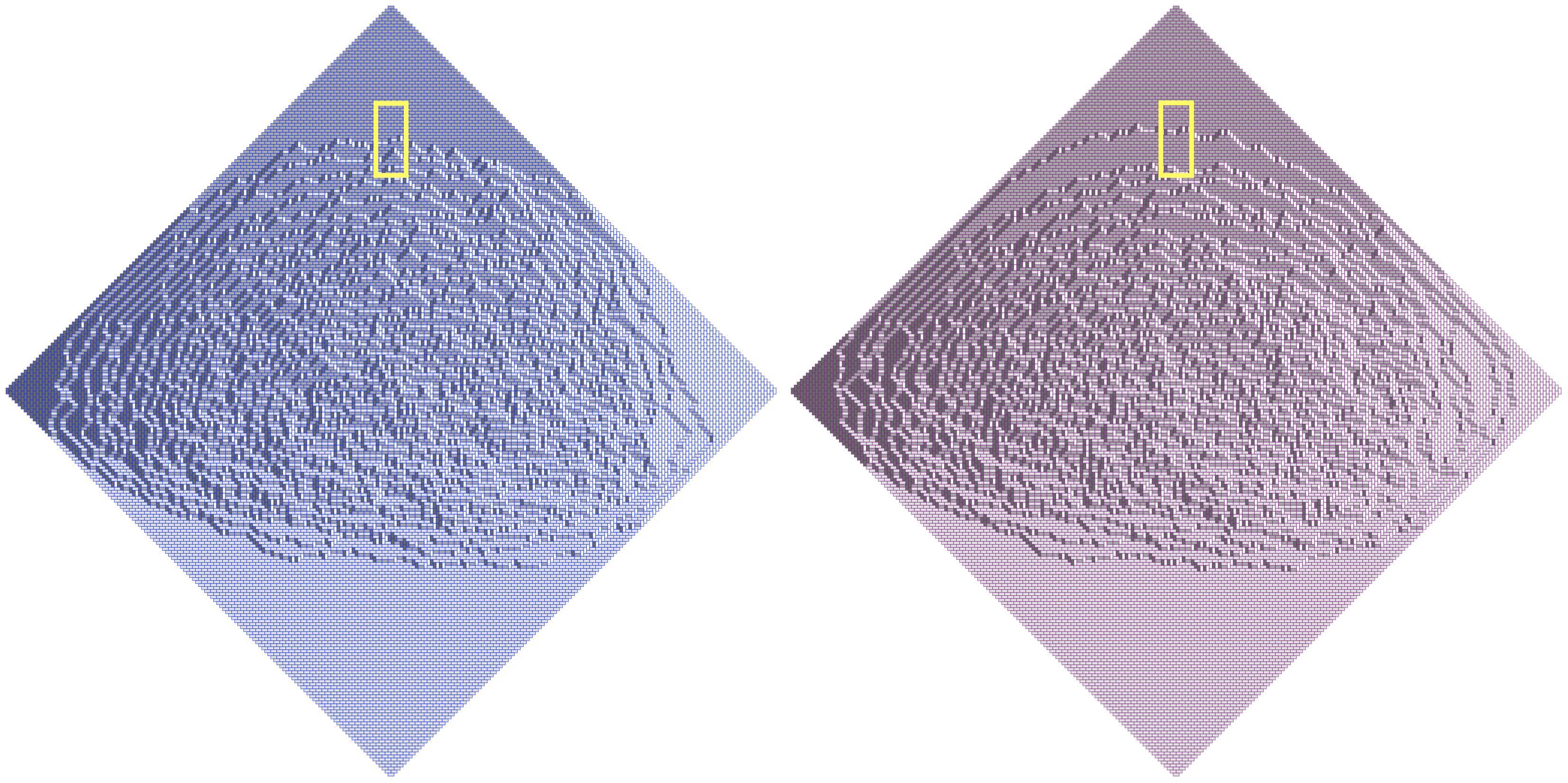} \\
\begin{tabular}{cccc}
\includegraphics[width=.22\textwidth]{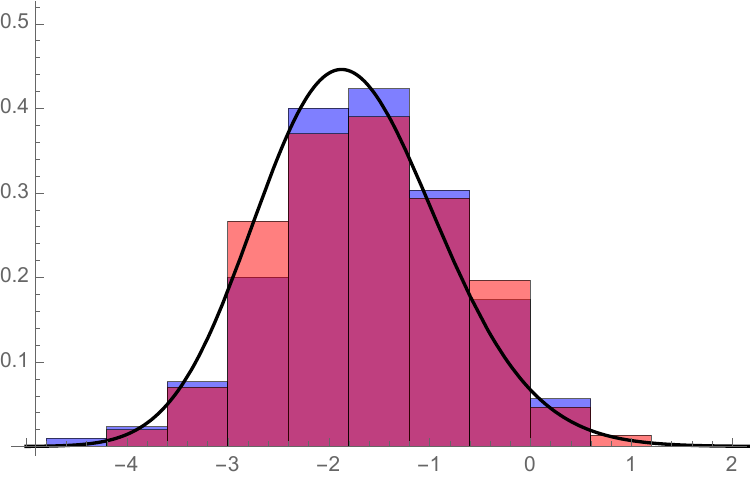}
&
\includegraphics[width=.22\textwidth]{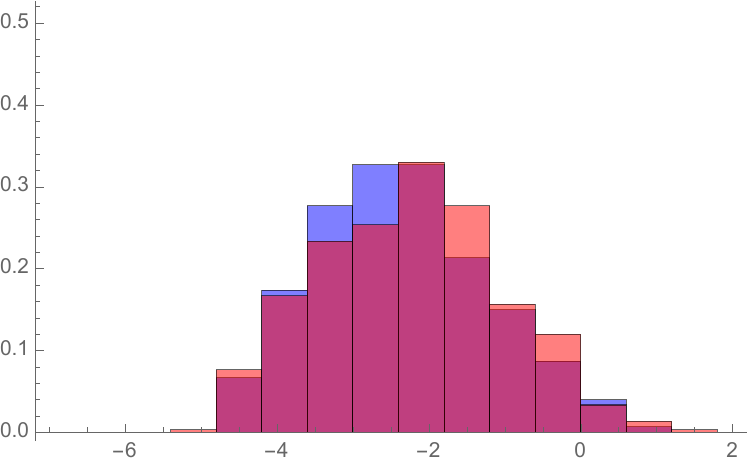}
&
\includegraphics[width=.22\textwidth]{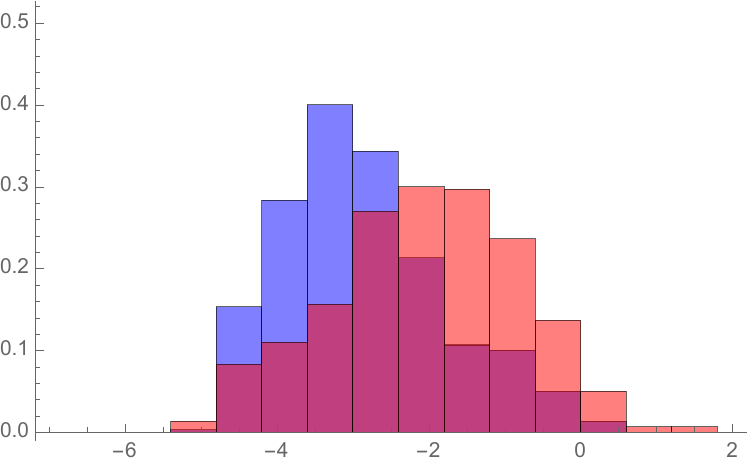}
&
\includegraphics[width=.22\textwidth]{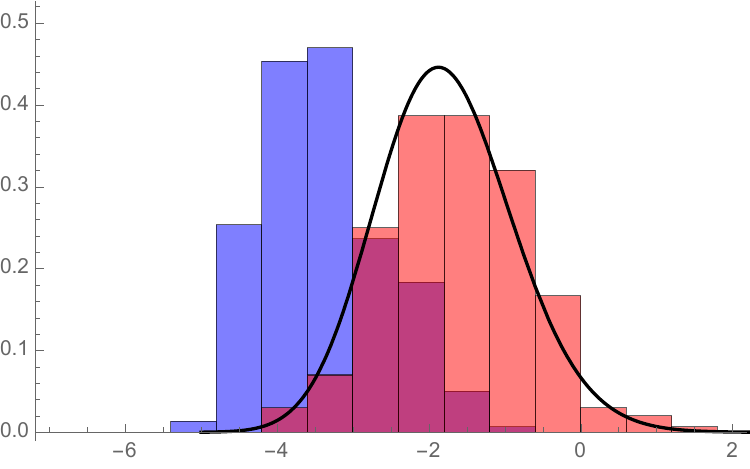} \\
$t=1$ & $t=0.5$ & $t=0.2$ & $t=0$
\end{tabular}
\end{tabular}
\caption{The numerically calculated fluctuations of the top portion of the arctic curve at $x=0$ were computed for several values of $t$. The red histogram shows the fluctuation for the red tililng and the blue histogram shows the fluctuation for the blue tiling. The black curve overlaid in the plots for $t=1$ and $t=0$ is the density of the Tracy-Widom $F_2$ distribution. As expected, for $t=1$ both arctic curves show Tracy-Widom fluctuations. From the bijection in Thm. \ref{thm:2tilingAC}, we also expect the arctic curve for the red tiling when $t=0$ to be Tracy-Widom. Moreover, when $t=0$, the top portion of the arctic curve for the blue tiling maps to the second highest Schr\"oder path in a 1-color Aztec diamond under the bijection in Thm. \ref{thm:2tilingAC}  and thus we expect its fluctuations to be the same as the fluctuations of the second largest eigenvalue of a GUE random matrix. Away from $t=0,1$, it is unclear what fluctuations to expect. Here we show the numerical results for $t=0.5$ and $t=0.2$, note the disappearance of Tracy-Widom type fluctuations for either color.}\label{fig:fluctuations}
\end{figure}

\begin{problem}
Study the LLT processes and understand the limi shapes and fluctuations given by these processes.  
\end{problem}

\appendix

\section{Appendix: Lozenge tilings}
\label{ap:a}

Here we consider the vertex model with only white rows. For a single color, if we start with the empty partition and end at the partition $(\underbrace{b,\ldots, b}_{a\text{ times}})$ after adding $a+c$ rows of white vertices, then the resulting vertex model configurations are in bijection with $k$ coupled lozenge tilings of an $a\times b \times c$ hexagon
\[
\resizebox{0.25\textwidth}{!}{
\begin{tikzpicture}[baseline=(current bounding box.center)]
\draw (0,0) grid (5,6);
\draw[ultra thick] (0.5,0)--(0.5,0.5);
\draw[ultra thick] (1.5,0)--(1.5,0.5);
\draw[ultra thick] (4.5,5.5)--(4.5,6);
\draw[ultra thick] (3.5,5.5)--(3.5,6);
\draw[thick,decorate,decoration={brace}] (1.9,-0.1)--(0.1,-0.1);
\draw[thick,decorate,decoration={brace}] (4.9,-0.1)--(2.1,-0.1);
\draw[thick,decorate,decoration={brace}] (-0.1,0.1)--(-0.1,5.9);
\node[below] at (1,-0.2) {$a$};
\node[below] at (3.5,-0.2) {$b$};
\node[] at (-0.7,3) {$a+c$};
\end{tikzpicture}
}
\rightarrow
\resizebox{0.25\textwidth}{!}{
\begin{tikzpicture}[baseline=(current bounding box.center)]
\begin{scope}[scale=1,yscale=0.866,xslant=0.5]
\draw[] (0,0)--(-2,2)--(-2,6)--(1,6)--(3,4)--(3,0)--(0,0);
\node[below,left] at (-1.0,1.0) {a};
\node[below] at (1.5,0.0) {b};
\node[above,left] at (-2.0,4.0) {c};
\foreach \i in {0,...,1}
\foreach \j in {-\i,...,3}
{
\draw[fill=lightgray] (\j,\i)--(\j,\i+1)--(\j-1,\i+1)--cycle;
}
\foreach \i in {2,...,3}
\foreach \j in {-1,...,3}
{
\draw[fill=lightgray] (\j,\i)--(\j,\i+1)--(\j-1,\i+1)--cycle;
}
\foreach \i in {4,...,5}{
\pgfmathparse{int(3-\i+3)}
\foreach \j in {-1,...,\pgfmathresult}
{
\draw[fill=lightgray] (\j,\i)--(\j,\i+1)--(\j-1,\i+1)--cycle;
}
}
\end{scope}
\end{tikzpicture}
}
\]
To do this we map paths to tiles by 
\[
\begin{aligned}
\resizebox{0.5cm}{!}{
\begin{tikzpicture}[baseline=(current bounding box.center)]
\draw (0,0) grid (1,2);
\draw[very thick] (0.5,0.5)--(0.5,1.5);
\end{tikzpicture}} & \mapsto &
\resizebox{0.5cm}{!}{
\begin{tikzpicture}[baseline=(current bounding box.center)]
\draw (0,0) grid (1,2);
\draw[dashed,fill=lightgray] (1,0)--(1,1)--(0,1)--cycle;
\draw[dashed] (1,1)--(0,2)--(0,1)--cycle;
\draw[very thick] (0.5,0.5)--(0.5,1.5);
\end{tikzpicture}} & \mapsto &
\resizebox{0.5cm}{!}{
\begin{tikzpicture}[baseline=(current bounding box.center)]
\draw[fill=lightgray] (1,0)--(1,1)--(0,1)--cycle;
\draw[] (1,1)--(0,2)--(0,1)--cycle;
\draw[very thick] (0.5,0.5)--(0.5,1.5);
\end{tikzpicture}} & \mapsto &
\resizebox{0.5cm}{!}{
\begin{tikzpicture}[baseline=(current bounding box.center)]
\begin{scope}[scale=1,yscale=0.866,xslant=0.5]
\draw[fill=lightgray] (1,0)--(1,1)--(0,1)--cycle;
\draw[] (1,1)--(0,2)--(0,1)--cycle;
\draw[very thick] (0.5,0.5)--(0.5,1.5);
\end{scope}
\end{tikzpicture}} \\
\resizebox{1cm}{!}{
\begin{tikzpicture}[baseline=(current bounding box.center)]
\draw (0,0) grid (2,1);
\draw[very thick] (0.5,0.5)--(1.5,0.5);
\end{tikzpicture}} & \mapsto &
\resizebox{1cm}{!}{
\begin{tikzpicture}[baseline=(current bounding box.center)]
\draw (0,0) grid (2,1);
\draw[dashed,fill=lightgray] (1,0)--(1,1)--(0,1)--cycle;
\draw[dashed] (2,0)--(1,1)--(1,0)--cycle;
\draw[very thick] (0.5,0.5)--(1.5,0.5);
\end{tikzpicture}} & \mapsto &
\resizebox{1cm}{!}{
\begin{tikzpicture}[baseline=(current bounding box.center)]
\draw[fill=lightgray] (1,0)--(1,1)--(0,1)--cycle;
\draw[] (2,0)--(1,1)--(1,0)--cycle;
\draw[very thick] (0.5,0.5)--(1.5,0.5);
\end{tikzpicture}} & \mapsto &
\resizebox{0.67cm}{!}{
\begin{tikzpicture}[baseline=(current bounding box.center)]
\begin{scope}[scale=1,yscale=0.866,xslant=0.5]
\draw[fill=lightgray] (1,0)--(1,1)--(0,1)--cycle;
\draw[] (2,0)--(1,1)--(1,0)--cycle;
\draw[very thick] (0.5,0.5)--(1.5,0.5);
\end{scope}
\end{tikzpicture}} \\
\resizebox{0.5cm}{!}{
\begin{tikzpicture}[baseline=(current bounding box.center)]
\draw (0,0) grid (1,1);
\end{tikzpicture}} & \mapsto &
\resizebox{0.5cm}{!}{
\begin{tikzpicture}[baseline=(current bounding box.center)]
\draw (0,0) grid (1,1);
\draw[dashed,fill=lightgray] (1,0)--(1,1)--(0,1)--cycle;
\draw[dashed] (0,0)--(0,1)--(1,0)--cycle;
\end{tikzpicture}} & \mapsto &
\resizebox{0.5cm}{!}{
\begin{tikzpicture}[baseline=(current bounding box.center)]
\draw[fill=lightgray] (1,0)--(1,1)--(0,1)--cycle;
\draw[] (0,0)--(0,1)--(1,0)--cycle;
\end{tikzpicture}} & \mapsto &
\resizebox{0.67cm}{!}{
\begin{tikzpicture}[baseline=(current bounding box.center)]
\begin{scope}[scale=1,yscale=0.866,xslant=0.5]
\draw[fill=lightgray] (1,0)--(1,1)--(0,1)--cycle;
\draw[] (0,0)--(0,1)--(1,0)--cycle;
\end{scope}
\end{tikzpicture}} 
\end{aligned}
\]
and then remove all frozen sections of paths. For example
\[
\resizebox{0.25\textwidth}{!}{
\begin{tikzpicture}[baseline=(current bounding box.center)]
\draw (0,0) grid (5,6);
\draw[ultra thick] (0.5,0)--(0.5,2.5)--(1.5,2.5)--(1.5,4.5)--(3.5,4.5)--(3.5,6);
\draw[ultra thick] (1.5,0)--(1.5,1.5)--(3.5,1.5)--(3.5,2.5)--(4.5,2.5)--(4.5,6);
\end{tikzpicture}
}
\rightarrow
\resizebox{0.25\textwidth}{!}{
\begin{tikzpicture}[baseline=(current bounding box.center)]
\begin{scope}[scale=1,yscale=0.866,xslant=0.5]
\draw[] (0,0)--(-2,2)--(-2,6)--(1,6)--(3,4)--(3,0)--(0,0);
\foreach \i in {0,...,1}
\foreach \j in {-\i,...,3}
{
\draw[fill=lightgray] (\j,\i)--(\j,\i+1)--(\j-1,\i+1)--cycle;
}
\foreach \i in {2,...,3}
\foreach \j in {-1,...,3}
{
\draw[fill=lightgray] (\j,\i)--(\j,\i+1)--(\j-1,\i+1)--cycle;
}
\foreach \i in {4,...,5}{
\pgfmathparse{int(3-\i+3)}
\foreach \j in {-1,...,\pgfmathresult}
{
\draw[fill=lightgray] (\j,\i)--(\j,\i+1)--(\j-1,\i+1)--cycle;
}
}
\draw[ultra thick, blue] (0,0)--(1,0)--(1,1)--(0,1)--(0,0);
\draw[ultra thick, blue] (1,0)--(2,0)--(2,1)--(1,1)--(1,0);
\draw[ultra thick, blue] (2,0)--(3,0)--(3,1)--(2,1)--(2,0);
\draw[ultra thick, blue] (2,1)--(3,1)--(3,2)--(2,2)--(2,1);
\draw[ultra thick, blue] (0,2)--(1,2)--(1,3)--(0,3)--(0,2);
\draw[ultra thick, blue] (0,3)--(1,3)--(1,4)--(0,4)--(0,3);
\draw[ultra thick, blue] (1,3)--(2,3)--(2,4)--(1,4)--(1,3);
\draw[ultra thick, blue] (-2,3)--(-1,3)--(-1,4)--(-2,4)--(-2,3);
\draw[ultra thick, blue] (-2,4)--(-1,4)--(-1,5)--(-2,5)--(-2,4);
\draw[ultra thick, blue] (-2,5)--(-1,5)--(-1,6)--(-2,6)--(-2,5);
\draw[ultra thick, blue] (-1,5)--(0,5)--(0,6)--(-1,6)--(-1,5);
\draw[ultra thick, blue] (0,5)--(1,5)--(1,6)--(0,6)--(0,5);
\draw[ultra thick, blue] (0,0)--(0,1)--(-1,2)--(-1,1)--(0,0);
\draw[ultra thick, blue] (-1,1)--(-1,2)--(-2,3)--(-2,2)--(-1,1);
\draw[ultra thick, blue] (2,1)--(2,2)--(1,3)--(1,2)--(2,1);
\draw[ultra thick, blue] (0,2)--(0,3)--(-1,4)--(-1,3)--(0,2);
\draw[ultra thick, blue] (3,2)--(3,3)--(2,4)--(2,3)--(3,2);
\draw[ultra thick, blue] (0,3)--(0,4)--(-1,5)--(-1,4)--(0,3);
\draw[ultra thick, blue] (3,3)--(3,4)--(2,5)--(2,4)--(3,3);
\draw[ultra thick, blue] (2,4)--(2,5)--(1,6)--(1,5)--(2,4);
\draw[ultra thick, blue] (0,1)--(1,1)--(0,2)--(-1,2)--(0,1);
\draw[ultra thick, blue] (0,4)--(1,4)--(0,5)--(-1,5)--(0,4);

\draw[ultra thick] (-1.5,1.5)--(-1.5,2.5)--(-0.5,2.5)--(-0.5,4.5)--(1.5,4.5)--(1.5,5.5);
\draw[ultra thick] (-0.5,0.5)--(-0.5,1.5)--(1.5,1.5)--(1.5,2.5)--(2.5,2.5)--(2.5,4.5);
\end{scope}
\end{tikzpicture}
}
\]
Now fix the number of colors $k$.  Map paths of each color to colored lozenges as described above. Note that in terms of the lozenges, we get a power of $t$ if 
\[
\resizebox{1cm}{!}{
\begin{tikzpicture}[baseline=(current bounding box.center)]
\begin{scope}[scale=1,yscale=0.866,xslant=0.5]
\draw[fill=lightgray] (1,0)--(1,1)--(0,1)--cycle;
\draw[] (2,0)--(1,1)--(1,0)--cycle;
\draw[very thick,blue] (1,0)--(0,1)--(1,1)--(2,0)--cycle;
\draw[very thick,red] (1.05,0.05)--(0.05,1.05)--(1.05,1.05)--(2.05,0.05)--cycle;
\end{scope}
\end{tikzpicture}
}
\text{ or }
\resizebox{1cm}{!}{
\begin{tikzpicture}[baseline=(current bounding box.center)]
\begin{scope}[scale=1,yscale=0.866,xslant=0.5]
\draw[fill=lightgray] (1,0)--(1,1)--(0,1)--cycle;
\draw[] (2,0)--(1,1)--(1,0)--cycle;
\draw[] (1,1)--(0,2)--(0,1)--cycle;
\draw[very thick,blue] (1,0)--(0,1)--(1,1)--(2,0)--cycle;
\draw[very thick,red] (1.05,0.05)--(0.05,1.05)--(0.05,2.05)--(1.05,1.05)--cycle;
\end{scope}
\end{tikzpicture}
}
\]
when blue is a smaller color than red.

\begin{remark} \label{rmk:lozengeLLT}
Under this bijection the partition function for $k$-tilings of the $a \times b \times c$ hexagon by lozenges is equal to the LLT polynomial $\mathcal{L}_{(\underbrace{\lambda,\ldots,\lambda}_{k \text{ times}})}(x_1,\ldots,x_{a+c};t)$ where $\lambda = (\underbrace{b,\ldots,b}_{a \text{ times}})$ and we take the specializations $x_i = 1$ for all $i=1,\ldots,a+c$.
\end{remark}

First, we consider a symmetry of the $k$-tilings.
\begin{prop}
Flipping a $k$-tilings of the $a\times b \times c$ hexagon across the vertical axis and reversing the order of the colors is a bijection between $k$-tilings of the $a\times b \times c$ hexagon and $k$-tilings of the $c\times b \times a$ hexagon such that a configuration with $j$ pairs of coupled lozenges maps to a configuration with $\binom{k}{2}(ab-bc)+j$ pairs of coupled lozenges.
\end{prop}
\begin{proof}
Label the types of lozenges
\[
\begin{tabular}{ccc}
  Type 1 & Type 2 & Type 3 \\
 \resizebox{0.5cm}{!}{
\begin{tikzpicture}[baseline=(current bounding box.center)]
\begin{scope}[scale=1,yscale=0.866,xslant=0.5]
\draw[fill=lightgray] (1,0)--(1,1)--(0,1)--cycle;
\draw[] (1,1)--(0,2)--(0,1)--cycle;
\end{scope}
\end{tikzpicture}} 
& 
\resizebox{0.67cm}{!}{
\begin{tikzpicture}[baseline=(current bounding box.center)]
\begin{scope}[scale=1,yscale=0.866,xslant=0.5]
\draw[fill=lightgray] (1,0)--(1,1)--(0,1)--cycle;
\draw[] (0,0)--(0,1)--(1,0)--cycle;
\end{scope}
\end{tikzpicture}} 
&
\resizebox{0.67cm}{!}{
\begin{tikzpicture}[baseline=(current bounding box.center)]
\begin{scope}[scale=1,yscale=0.866,xslant=0.5]
\draw[fill=lightgray] (1,0)--(1,1)--(0,1)--cycle;
\draw[] (2,0)--(1,1)--(1,0)--cycle;
\end{scope}
\end{tikzpicture}}
\end{tabular}
\]
Note that every tiling of an $a\times b \times c$ hexagon has $ac$ lozenges of type 1, $bc$ lozenges of type 2, and $ab$ lozenges of type 3. After flipping across the vertical axis, those of type 2 map to those of type 3 and vice-versa, while those of type 1 stay the same. Let $\phi$ be the map that flips a $k$-tiling about the vertical axis and reverses the order of the colors. Under $\phi$ lozenges of type 2 map to those of type 3 and vice-versa, while those of type 1 stay the same.

Consider any pair of colors $\alpha<\beta$. We'll draw color $\alpha$ as blue and color $\beta$ as red. Consider the 2-tiling $(T_\alpha,T_\beta)$. Recall the lozenges that give an interaction are
\[
\resizebox{1cm}{!}{
\begin{tikzpicture}[baseline=(current bounding box.center)]
\begin{scope}[scale=1,yscale=0.866,xslant=0.5]
\draw[fill=lightgray] (1,0)--(1,1)--(0,1)--cycle;
\draw[] (2,0)--(1,1)--(1,0)--cycle;
\draw[very thick,blue] (1,0)--(0,1)--(1,1)--(2,0)--cycle;
\draw[very thick,red] (1.05,0.05)--(0.05,1.05)--(1.05,1.05)--(2.05,0.05)--cycle;
\end{scope}
\end{tikzpicture}
}
\text{ or }
\resizebox{1cm}{!}{
\begin{tikzpicture}[baseline=(current bounding box.center)]
\begin{scope}[scale=1,yscale=0.866,xslant=0.5]
\draw[fill=lightgray] (1,0)--(1,1)--(0,1)--cycle;
\draw[] (2,0)--(1,1)--(1,0)--cycle;
\draw[] (1,1)--(0,2)--(0,1)--cycle;
\draw[very thick,blue] (1,0)--(0,1)--(1,1)--(2,0)--cycle;
\draw[very thick,red] (1.05,0.05)--(0.05,1.05)--(0.05,2.05)--(1.05,1.05)--cycle;
\end{scope}
\end{tikzpicture}
}.
\]
Note that 
\[
\#\left(
\resizebox{1cm}{!}{
\begin{tikzpicture}[baseline=(current bounding box.center)]
\begin{scope}[scale=1,yscale=0.866,xslant=0.5]
\draw[fill=lightgray] (1,0)--(1,1)--(0,1)--cycle;
\draw[] (2,0)--(1,1)--(1,0)--cycle;
\draw[very thick,blue] (1,0)--(0,1)--(1,1)--(2,0)--cycle;
\draw[very thick,red] (1.05,0.05)--(0.05,1.05)--(1.05,1.05)--(2.05,0.05)--cycle;
\end{scope}
\end{tikzpicture}
}
\right)
+
\#\left(
\resizebox{1cm}{!}{
\begin{tikzpicture}[baseline=(current bounding box.center)]
\begin{scope}[scale=1,yscale=0.866,xslant=0.5]
\draw[fill=lightgray] (1,0)--(1,1)--(0,1)--cycle;
\draw[] (2,0)--(1,1)--(1,0)--cycle;
\draw[] (1,1)--(0,2)--(0,1)--cycle;
\draw[very thick,blue] (1,0)--(0,1)--(1,1)--(2,0)--cycle;
\draw[very thick,red] (1.05,0.05)--(0.05,1.05)--(0.05,2.05)--(1.05,1.05)--cycle;
\end{scope}
\end{tikzpicture}
}
\right)
+
\#\left(
\resizebox{1.33cm}{!}{
\begin{tikzpicture}[baseline=(current bounding box.center)]
\begin{scope}[scale=1,yscale=0.866,xslant=0.5]
\draw[fill=lightgray] (1,0)--(1,1)--(0,1)--cycle;
\draw[] (2,0)--(1,1)--(1,0)--cycle;
\draw[] (1,0)--(0,1)--(0,0)--cycle;
\draw[very thick,blue] (1,0)--(0,1)--(1,1)--(2,0)--cycle;
\draw[very thick,red] (0.05,0.05)--(0.05,1.05)--(1.05,1.05)--(1.05,0.05)--cycle;
\end{scope}
\end{tikzpicture}
}
\right)
= ab
\]
since the number of blue lozenges of type 3 is $ab$. Rearranging we have
\begin{equation} \label{eq:lozcount1}
\#\left(
\resizebox{1cm}{!}{
\begin{tikzpicture}[baseline=(current bounding box.center)]
\begin{scope}[scale=1,yscale=0.866,xslant=0.5]
\draw[fill=lightgray] (1,0)--(1,1)--(0,1)--cycle;
\draw[] (2,0)--(1,1)--(1,0)--cycle;
\draw[very thick,blue] (1,0)--(0,1)--(1,1)--(2,0)--cycle;
\draw[very thick,red] (1.05,0.05)--(0.05,1.05)--(1.05,1.05)--(2.05,0.05)--cycle;
\end{scope}
\end{tikzpicture}
}
\right)
+
\#\left(
\resizebox{1cm}{!}{
\begin{tikzpicture}[baseline=(current bounding box.center)]
\begin{scope}[scale=1,yscale=0.866,xslant=0.5]
\draw[fill=lightgray] (1,0)--(1,1)--(0,1)--cycle;
\draw[] (2,0)--(1,1)--(1,0)--cycle;
\draw[] (1,1)--(0,2)--(0,1)--cycle;
\draw[very thick,blue] (1,0)--(0,1)--(1,1)--(2,0)--cycle;
\draw[very thick,red] (1.05,0.05)--(0.05,1.05)--(0.05,2.05)--(1.05,1.05)--cycle;
\end{scope}
\end{tikzpicture}
}
\right)
= ab-\#\left(
\resizebox{1.33cm}{!}{
\begin{tikzpicture}[baseline=(current bounding box.center)]
\begin{scope}[scale=1,yscale=0.866,xslant=0.5]
\draw[fill=lightgray] (1,0)--(1,1)--(0,1)--cycle;
\draw[] (2,0)--(1,1)--(1,0)--cycle;
\draw[] (1,0)--(0,1)--(0,0)--cycle;
\draw[very thick,blue] (1,0)--(0,1)--(1,1)--(2,0)--cycle;
\draw[very thick,red] (0.05,0.05)--(0.05,1.05)--(1.05,1.05)--(1.05,0.05)--cycle;
\end{scope}
\end{tikzpicture}
}
\right)
\end{equation}

Now the lozenges in $(T_\alpha,T_\beta)$ that will count as a coupled pair after applying $\phi$ are of the form
\[
\resizebox{1cm}{!}{
\begin{tikzpicture}[baseline=(current bounding box.center)]
\begin{scope}[scale=1,yscale=0.866,xslant=0.5]
\draw[fill=lightgray] (1,0)--(1,1)--(0,1)--cycle;
\draw[] (1,0)--(0,1)--(0,0)--cycle;
\draw[very thick,blue] (0,0)--(0,1)--(1,1)--(1,0)--cycle;
\draw[very thick,red] (0.05,0.05)--(0.05,1.05)--(1.05,1.05)--(1.05,0.05)--cycle;
\end{scope}
\end{tikzpicture}
}
\text{ or }
\resizebox{1cm}{!}{
\begin{tikzpicture}[baseline=(current bounding box.center)]
\begin{scope}[scale=1,yscale=0.866,xslant=0.5]
\draw[fill=lightgray] (1,0)--(1,1)--(0,1)--cycle;
\draw[] (1,1)--(0,2)--(0,1)--cycle;
\draw[] (1,0)--(0,1)--(0,0)--cycle;
\draw[very thick,red] (0,0)--(0,1)--(1,1)--(1,0)--cycle;
\draw[very thick,blue] (1.05,0.05)--(0.05,1.05)--(0.05,2.05)--(1.05,1.05)--cycle;
\end{scope}
\end{tikzpicture}
}.
\]
Similarly to the previous calculation, we have
\begin{equation}\label{eq:lozcount2}
\#\left(\resizebox{1cm}{!}{
\begin{tikzpicture}[baseline=(current bounding box.center)]
\begin{scope}[scale=1,yscale=0.866,xslant=0.5]
\draw[fill=lightgray] (1,0)--(1,1)--(0,1)--cycle;
\draw[] (1,0)--(0,1)--(0,0)--cycle;
\draw[very thick,blue] (0,0)--(0,1)--(1,1)--(1,0)--cycle;
\draw[very thick,red] (0.05,0.05)--(0.05,1.05)--(1.05,1.05)--(1.05,0.05)--cycle;
\end{scope}
\end{tikzpicture}
}
\right)
+
\#\left(\resizebox{1cm}{!}{
\begin{tikzpicture}[baseline=(current bounding box.center)]
\begin{scope}[scale=1,yscale=0.866,xslant=0.5]
\draw[fill=lightgray] (1,0)--(1,1)--(0,1)--cycle;
\draw[] (1,1)--(0,2)--(0,1)--cycle;
\draw[] (1,0)--(0,1)--(0,0)--cycle;
\draw[very thick,red] (0,0)--(0,1)--(1,1)--(1,0)--cycle;
\draw[very thick,blue] (1.05,0.05)--(0.05,1.05)--(0.05,2.05)--(1.05,1.05)--cycle;
\end{scope}
\end{tikzpicture}
}
\right) =
bc - \#\left(
\resizebox{1.33cm}{!}{
\begin{tikzpicture}[baseline=(current bounding box.center)]
\begin{scope}[scale=1,yscale=0.866,xslant=0.5]
\draw[fill=lightgray] (1,0)--(1,1)--(0,1)--cycle;
\draw[] (2,0)--(1,1)--(1,0)--cycle;
\draw[] (1,0)--(0,1)--(0,0)--cycle;
\draw[very thick,blue] (1,0)--(0,1)--(1,1)--(2,0)--cycle;
\draw[very thick,red] (0.05,0.05)--(0.05,1.05)--(1.05,1.05)--(1.05,0.05)--cycle;
\end{scope}
\end{tikzpicture}
}
\right).
\end{equation}
Subtracting equations (\ref{eq:lozcount2}) from (\ref{eq:lozcount1}), we see that the difference in the number of pairs of coupled lozenges is constant, in particular, it is $ab-bc$. Doing this for every pair of colors $\alpha<\beta$ gives the result. 
\end{proof}

Similar to the Aztec diamond, for special values of $t$ we have bijections with $1$-tilings.
\begin{prop}
When $t=0$ there is a bijection between $2$-tilings of the $a\times b \times c$ hexagon by lozenges and $1$-tilings of the $2a\times b \times (c-a)$ hexagon. (If $a>c$, then there are no configurations when $t=0$.)
\end{prop}
\begin{proof}
Look at the paths of the vertex model. Then a similar sliding argument as in Section \ref{t-0-section} works again.

More precisely, label the starting points of each color of path $1,\ldots,a$. Then shift the $i$-th red path over to the right $i$ columns, and then shift the $i$-th blue path over to the right by $i-1$ columns. The claim is that now the paths are non-intersecting. 

We can see this by first noting that when $t=0$ the $i$-th red path must be weakly to right of the $i$-th blue path (since they start at the same place and the blue paths can never cross the red path as it would give a power of $t$). Further, the blue paths also cannot travel horizontally with the red path. 

Next we see that the $i$-th red path is also strictly to the left of the $(i+1)$-th blue path. One can see this as the red path starts to the left and ends to the left of the blue path, so if the two paths ever share a face the blue path must cross the red path eventually, resulting in a power of $t$.

Thus, after the shifting, the $i$-th red path is strictly between the $i$-th and $(i+1)$-th blue paths. Clearly, this is reversible.
\end{proof}
See Figure \ref{fig:lozshift} for an example. A similar result holds for $k$-tilings.
\begin{figure}[ht]
    \centering
    \[
    \begin{aligned}
\resizebox{0.25\textwidth}{!}{
\begin{tikzpicture}[baseline=(current bounding box.center)]
\draw (0,0) grid (5,5);
\draw[ultra thick, blue] (0.4,0)--(0.4,3.6)--(1.4,3.6)--(1.4,4.6)--(3.4,4.6)--(3.4,5);
\draw[ultra thick, blue] (1.4,0)--(1.4,1.6)--(3.4,1.6)--(3.4,2.6)--(4.4,2.6)--(4.4,5);
\draw[ultra thick, red] (0.6,0)--(0.6,2.4)--(1.6,2.4)--(1.6,3.4)--(3.6,3.4)--(3.6,5);
\draw[ultra thick, red] (1.6,0)--(1.6,0.4)--(4.6,0.4)--(4.6,5);
\end{tikzpicture}
}
&
\rightarrow
\resizebox{0.34\textwidth}{!}{
\begin{tikzpicture}[baseline=(current bounding box.center)]
\draw (0,0) grid (7,5);
\draw[ultra thick,blue] (0.4,0)--(0.4,3.6)--(1.4,3.6)--(1.4,4.6)--(3.4,4.6)--(3.4,5);
\draw[ultra thick,blue] (2.4,0)--(2.4,1.6)--(4.4,1.6)--(4.4,2.6)--(5.4,2.6)--(5.4,5);
\draw[ultra thick,red] (1.6,0)--(1.6,2.4)--(2.6,2.4)--(2.6,3.4)--(4.6,3.4)--(4.6,5);
\draw[ultra thick,red] (3.6,0)--(3.6,0.4)--(6.6,0.4)--(6.6,5);
\end{tikzpicture}
}
\\
\resizebox{0.25\textwidth}{!}{
\begin{tikzpicture}[baseline=(current bounding box.center)]
\begin{scope}[scale=1,yscale=0.866,xslant=0.5]
\draw[] (0,0)--(-2,2)--(-2,5)--(1,5)--(3,3)--(3,0)--(0,0);
\foreach \i in {0,...,1}
\foreach \j in {-\i,...,3}
{
\draw[fill=lightgray] (\j,\i)--(\j,\i+1)--(\j-1,\i+1)--cycle;
}
\foreach \i in {2,...,2}
\foreach \j in {-1,...,3}
{
\draw[fill=lightgray] (\j,\i)--(\j,\i+1)--(\j-1,\i+1)--cycle;
}
\foreach \i in {3,...,4}{
\pgfmathparse{int(3-\i+2)}
\foreach \j in {-1,...,\pgfmathresult}
{
\draw[fill=lightgray] (\j,\i)--(\j,\i+1)--(\j-1,\i+1)--cycle;
}
}

\draw[ultra thick, blue] (0.4-2,1.6)--(0.4-2,3.6)--(1.4-2,3.6)--(1.4-2,4.6)--(3.4-2,4.6);
\draw[ultra thick, blue] (1.4-2,0.6)--(1.4-2,1.6)--(3.4-2,1.6)--(3.4-2,2.6)--(4.4-2,2.6)--(4.4-2,3.6);
\draw[ultra thick, red] (0.6-2,1.4)--(0.6-2,2.4)--(1.6-2,2.4)--(1.6-2,3.4)--(3.6-2,3.4)--(3.6-2,4.4);
\draw[ultra thick, red] (1.6-2,0.4)--(1.6-2,0.4)--(4.6-2,0.4)--(4.6-2,3.4);
\end{scope}
\end{tikzpicture}
}
&
\rightarrow
\resizebox{0.25\textwidth}{!}{
\begin{tikzpicture}[baseline=(current bounding box.center)]
\begin{scope}[scale=1,yscale=0.866,xslant=0.5]
\draw[] (0,0)--(-4,4)--(-4,5)--(-1,5)--(3,1)--(3,0)--(0,0);
\foreach \i in {0,...,0}
\foreach \j in {-\i,...,3}
{
\draw[fill=lightgray] (\j,\i)--(\j,\i+1)--(\j-1,\i+1)--cycle;
}
\foreach \i in {1,...,3}{
\pgfmathparse{int(3-\i)}
\foreach \j in {-\i,...,\pgfmathresult}
{
\draw[fill=lightgray] (\j,\i)--(\j,\i+1)--(\j-1,\i+1)--cycle;
}
}
\foreach \i in {4,...,4}{
\pgfmathparse{int(3-\i)}
\foreach \j in {-3,...,\pgfmathresult}
{
\draw[fill=lightgray] (\j,\i)--(\j,\i+1)--(\j-1,\i+1)--cycle;
}
}

\draw[ultra thick,blue] (0.4-4,3.6)--(0.4-4,3.6)--(1.4-4,3.6)--(1.4-4,4.6)--(3.4-4,4.6)--(3.4-4,4.6);
\draw[ultra thick,blue] (2.4-4,1.6)--(2.4-4,1.6)--(4.4-4,1.6)--(4.4-4,2.6)--(5.4-4,2.6)--(5.4-4,2.6);
\draw[ultra thick,red] (1.6-4,2.4)--(1.6-4,2.4)--(2.6-4,2.4)--(2.6-4,3.4)--(4.6-4,3.4)--(4.6-4,3.4);
\draw[ultra thick,red] (3.6-4,0.4)--(3.6-4,0.4)--(6.6-4,0.4)--(6.6-4,1.4);
\end{scope}
\end{tikzpicture}
}
    \end{aligned}
    \]
    \caption{Example of the bijection between the 2-tilings of the $2 \times 3 \times 3$ hexagon at $t=0$ and tilings of the $4 \times 3 \times 1$ hexagon. The top gives the bijection in terms of lattice path, while the bottom gives the lozenge tilings.}
    \label{fig:lozshift}
\end{figure}
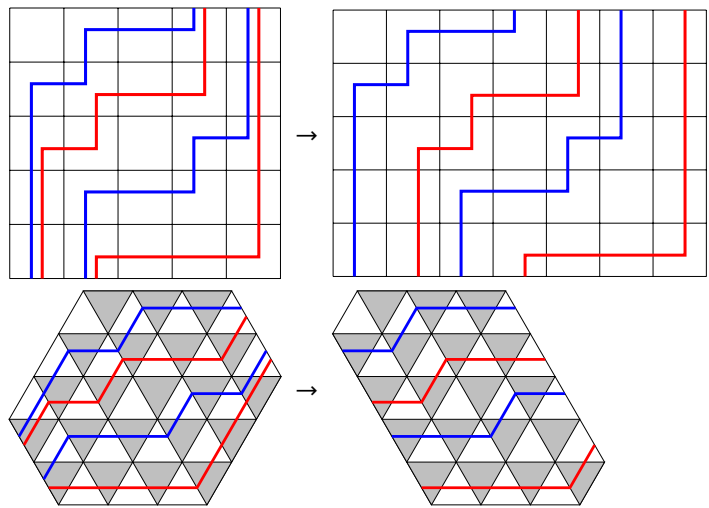

\begin{prop}
When $t=0$ there is a bijection between $k$-tilings of the $a\times b \times c$ hexagon by lozenges and $1$ tilings of the $ka\times b \times (c-(k-1)a)$ hexagon. (If $(k-1)a>c$, then there are no configurations when $t=0$.)
\end{prop}

The artic curve for lozenge tilings of a hexagon was computed by Cohn, Larsen, and Propp in 1998:
\begin{thm}\cite{CLP98,GorinBook}
For $aL \times bL \times cL$ hexagon, a uniformly random
tiling is with high probability asymptotically frozen outside the inscribed
ellipse as $L \rightarrow\infty$. In more detail, for each $(x, y)$ outside the ellipse, with
probability tending to 1 as $L \rightarrow\infty$, all the lozenges that we observe in a
finite neighborhood of $$(xL, yL)$$ are of the same type.
\end{thm}

From this we can calculate the arctic curve of the $k$-tilings of an hexagon when $t=0$.
\begin{thm}
When $t=0$, as $a\to\infty$ the arctic curves (for both colors) of the $2$-tilings of an $a\times 2a \times 3a$ hexagon are given by
\[
\begin{cases}
\left(\frac{6x-\sqrt{3}y+6}{3}\right)^2+y^2=3, & x\le-\frac{3}{2},\;\;  -\frac{\sqrt{3}}{2}\le y \le \frac{\sqrt{3}}{2} \\
x^2+y^2=3, & x\le0, \;\; \frac{\sqrt{3}}{2}\le y \le \sqrt{3} \\
\left(\frac{3x-\sqrt{3}y+3}{3}\right)^2+y^2=3, &  x\ge 0, \;\; \frac{\sqrt{3}}{2}\le y \le \sqrt{3} \\
\left(\frac{6x-\sqrt{3}y}{3}\right)^2+y^2=3, & x\ge\frac{1}{2}, \;\; -\frac{\sqrt{3}}{2}\le y \le \frac{\sqrt{3}}{2}  \\
(x+1)^2+y^2=3, & x\ge -1, \;\; -\sqrt{3}\le y \le -\frac{\sqrt{3}}{2} \\
\left(\frac{3x-\sqrt{3}y}{3}\right)^2+y^2=3, & x\le -1, \;\; -\sqrt{3}\le y \le -\frac{\sqrt{3}}{2}
\end{cases}
\]
(More generally, for $k$-tilings of an $a\times ka \times (2k-1)a$ hexagon the arctic curve can be worked out similarly.)
\end{thm}
\begin{proof}
As $a\to \infty$, we know that after appropriate rescaling the arctic curve for the $1$-tilings of the $2a\times 2a \times 2a$ hexagon is the circle $x^2+y^2=3$. We map it to the $2$-tiling case via our bijection (as in the case for the Aztec diamond).
\end{proof}

\begin{figure}[ht]
    \centering
    \resizebox{\textwidth}{!}{
    \begin{tabular}{ccc}
    \includegraphics[width=0.5\textwidth]{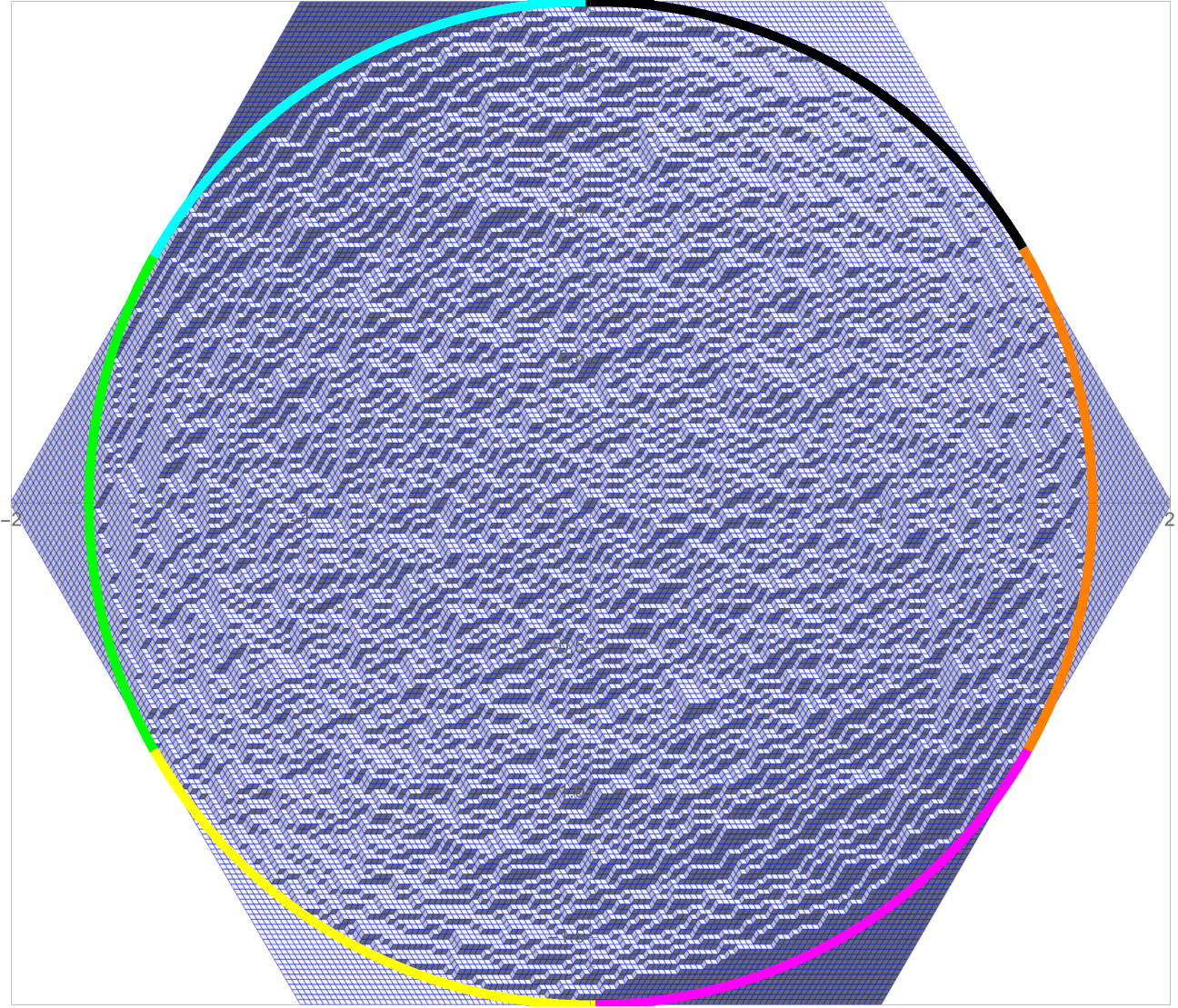} &
    \includegraphics[width=0.5\textwidth]{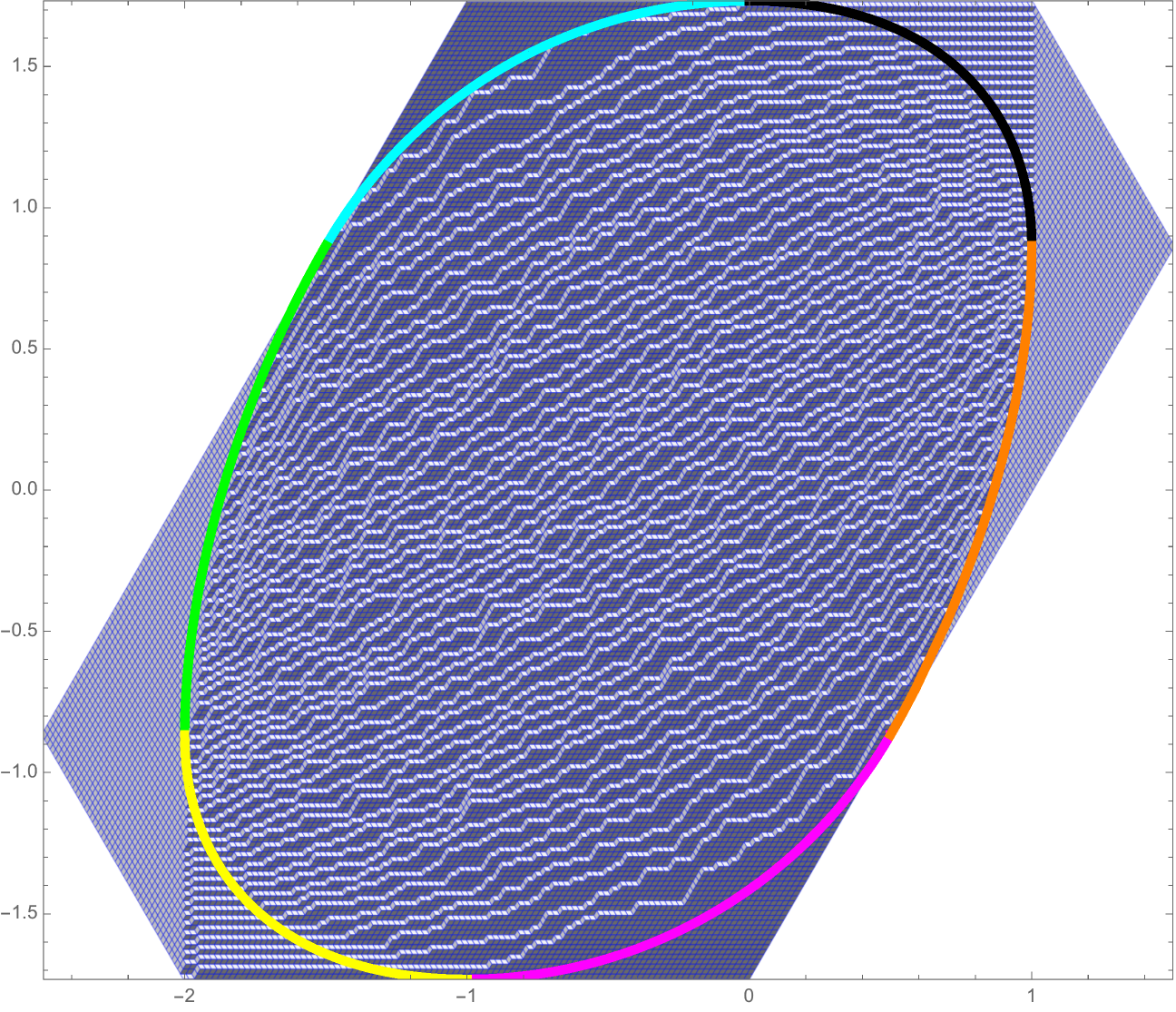} & \includegraphics[width=0.5\textwidth]{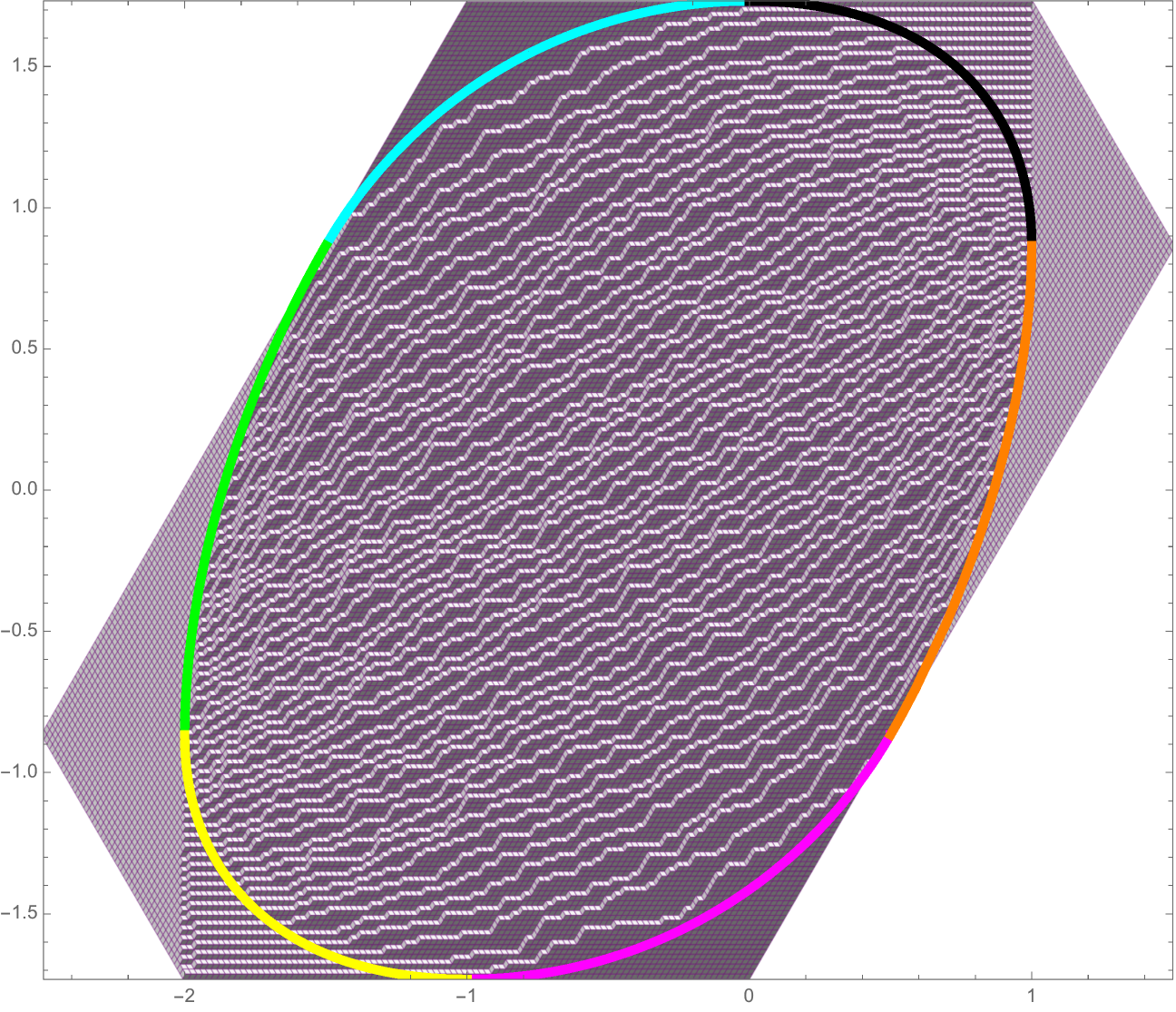}
    \end{tabular}
    }
    \caption{A simulation of a 1-tiling of an $100 \times 100 \times 100$ hexagon (left) and a simulation and computed arctic curve for $2$-tiling a $50 \times 100 \times 150$ hexagon (right). The colors of the arctic curve show which pieces map to each other under the bijection. }
    \label{fig:ArcticCurveLozenge}
\end{figure}

See Figure \ref{fig:ArcticCurveLozenge} for an example of the arctic curves for $2$ colors and $t=0$.

We can also work out the case when $t\to \infty$. Unlike the Aztec diamond, the mapping in this case takes a different form than that of $t=0$. One can show the following:

\begin{lem} \label{lem:loztinf}
For each color label the paths $1,\ldots, a$ with the path with the left-most starting point being path $1$ and then continuing to the right. Suppose $k=2$. Let $r_{i,j}^{(1)}$ be the row in which the $i$-th path of color $1$ goes right on its $j$-th step, and similarly for color $2$. Then
\[
r_{i,1}^{(1)} \le r_{i,1}^{(2)} \le r_{i,2}^{(1)} \le r_{i,2}^{(2)} \le \ldots \le r_{i,b}^{(1)} \le r_{i,b}^{(2)}.
\]
\end{lem}

From this it follows that
\begin{prop}
When $t\to \infty$ there is a bijection between 2-tilings of the $a\times b\times c$ hexagon by lozenges and 1-tilings of the $a\times 2b \times c$ hexagon.
\end{prop}
\begin{proof}
Let $r_{i,j}$ be the row in which the $i$-th path of the $1$-tiling goes right on its $j$-th step. Then the bijection is given by taking
\[
r_{i,2j-1} =r_{i,j}^{(1)} \text{ and } r_{i,2j} =r_{i,j}^{(2)}.
\]
The previous Lemma \ref{lem:loztinf} ensures this is a valid configuration of paths.
\end{proof}
More generally, a similar argument holds for $k$-tilings.
\begin{prop}
When $t\to \infty$ there is a bijection between $k$-tilings of the $a\times b\times c$ hexagon by lozenges and 1-tilings of the $a\times kb \times c$ hexagon.
\end{prop}

By reversing this bijection we can compute the arctic curve as $t\to \infty$.
\begin{thm}
When $t \rightarrow \infty$, as $a\to \infty$ the arctic curves (for both colors) of the 2-tiling of an $2a\times a \times 2a$ hexagon are given by
\[
\begin{cases}
x^2+y^2=3, & x\le-\frac{3}{2},\;\;  -\frac{\sqrt{3}}{2}\le y \le \frac{\sqrt{3}}{2} \\
\left(\frac{6x-\sqrt{3}y+6}{3}\right)^2+y^2=3, & x\le-\frac{1}{2}, \;\; \frac{\sqrt{3}}{2}\le y \le \sqrt{3} \\
\left(\frac{6x+\sqrt{3}y}{3}\right)^2+y^2=3, &  x\ge -\frac{1}{2}, \;\; \frac{\sqrt{3}}{2}\le y \le \sqrt{3} \\
(x+1)^2+y^2=3, & x\ge\frac{1}{2}, \;\; -\frac{\sqrt{3}}{2}\le y \le \frac{\sqrt{3}}{2}  \\
\left(\frac{6x-\sqrt{3}y}{3}\right)^2+y^2=3, & x\ge -\frac{1}{2}, \;\; -\sqrt{3}\le y \le -\frac{\sqrt{3}}{2} \\
\left(\frac{6x+\sqrt{3}y+6}{3}\right)^2+y^2=3, & x\le -\frac{1}{2}, \;\; -\sqrt{3}\le y \le -\frac{\sqrt{3}}{2}
\end{cases}
\]
(More generally, for $k$-tilings of an $ka\times a \times ka$ hexagon the arctic curves can be worked out similarly.)
\end{thm}
\begin{proof}
Rescale the $2a\times 2a \times 2a$ hexagon so that it has sides of length $2$ and it is centered at $(x,y)=(0,0)$. Now each lozenge has side length $1/a$.

Thinking of the tiling interchangeably as paths and lozenges, we have that the paths start on the SW side of the hexagon at 
\[
(x_i,y_i)=\left(-2+\frac{2i-1}{4a}, -\frac{\sqrt{3}(2i-1)}{4a}\right), \;\;\; i=1,2,\ldots,2a.
\]
Note that each path will take $2a$ horizontal steps, with each horizontal step moving the path $1/a$ to the right. For path $i$ the center of the $j$-th horizontal step will occur along the line $y=\sqrt{3}(x+2-(i+2j-2)/(2a))$. Reversing the bijection from the previous proposition, we see that in the 1-tiling  the $(2j-1)$-th horizontal step of path $i$ will map to the $j$-th horizontal step of path $i$ in the blue tiling, while the $2j$-th horizontal step of path $i$ will map to the $j$-th horizontal step of path $i$ in the red tiling, for $j=1,\ldots,a$. Under the bijection the $y$-coordinate of the path does not change. Geometrically this corresponds to the shifting the $(2j-1)$-th horizontal step of path $i$ in the 1-tiling to the right by $(j-1)/a$ to get the $j$-th horizontal step of path $i$ in the blue tiling. Similarly, we shift the $2j$-th horizontal step of path $i$ in the 1-tiling to the right by $j/a$ to get the red tiling. We can use this to see how to map different sections of the arctic curve.

For example, consider the first path which starts at $(x_1,y_1)$. The trajectory of this path gives the boundary of the upper frozen region of lozenges of type 2 and the disordered region. In the 1-tiling, this portion of the arctic curve is given by $x^2+y^2=3$, $x\le 0$, $\sqrt{3}{2}\le y\le \sqrt{3}$.

As stated above, to get the first blue path we shift $(2j-1)$-th horizontal step in the 1-tiling to the right by $(j-1)/a$. Since in the 1-tiling this horizontal step lies along the line $y=\sqrt{3}(x+2-(4j-1)/(2a))$, we have $(j-1)/a = \frac{1}{6} (6+3x-\sqrt{3}y) + O(1/a)$. Thus the map from the first path of the 1-tiling to the first path of the blue tiling is given by
\[
(x,y)\mapsto \left(x-\frac{1}{6} (6+3x-\sqrt{3}y),y\right) = \left(\frac{1}{6} (3x+\sqrt{3}y-6),y\right)
\]
up to terms that go to zero as $a\to\infty$. Inverting this we see that this portion of the arctic curve for the blue tiling is given by 
\[
\left(\frac{6x-\sqrt{3}y+6}{3}\right)^2+y^2=3.
\]
The analysis for the first red path works the same. This gives the second case in the statement of the Theorem.

The other portions of the arctic curve can be done similarly.

\end{proof}
See Figure \ref{fig:ArcticCurveLozenget100} for an example.

\begin{figure}[ht]
    \centering
    \resizebox{\textwidth}{!}{
    \begin{tabular}{ccc}
    \includegraphics[width=0.5\textwidth]{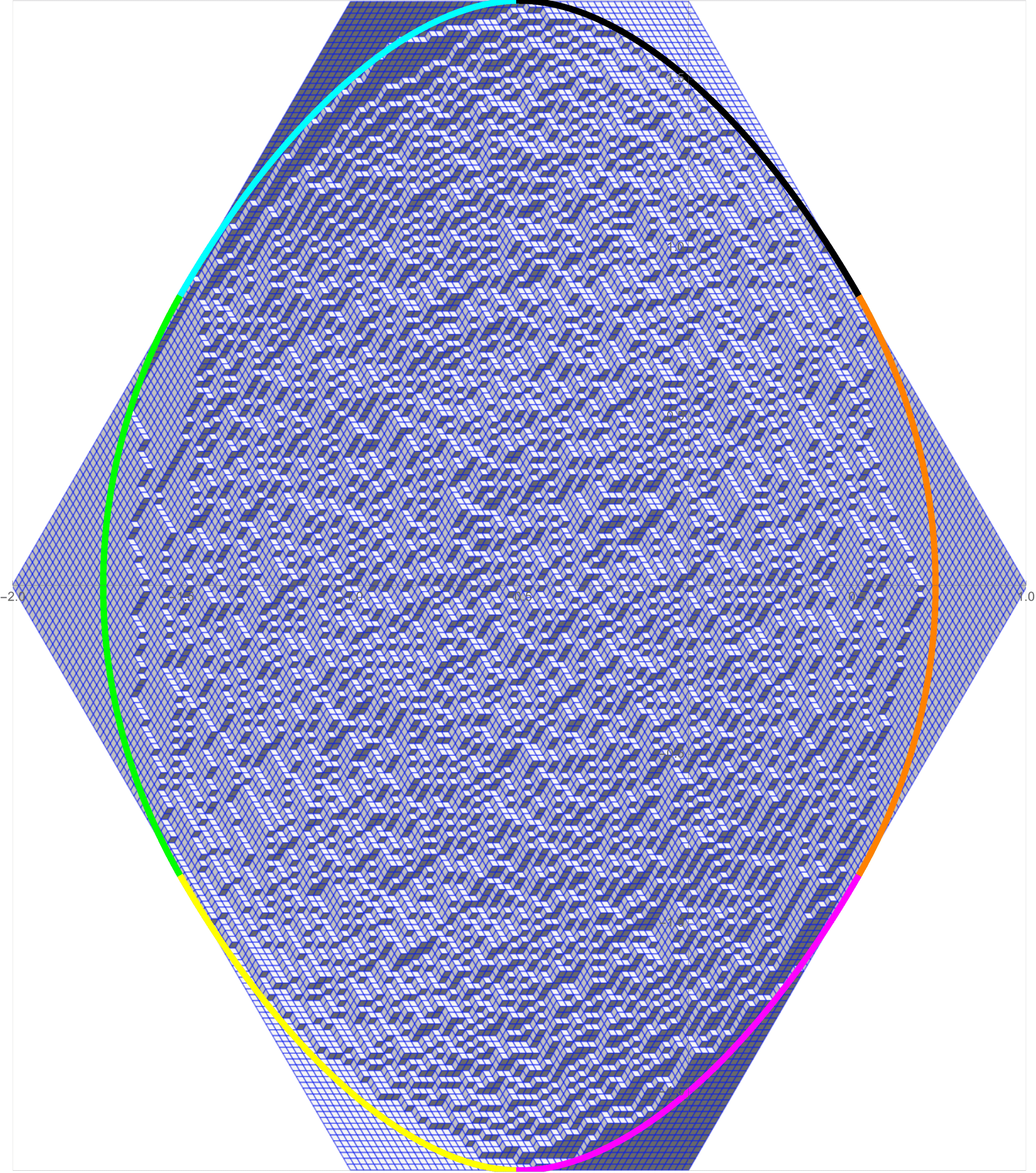} & \includegraphics[width=0.5\textwidth]{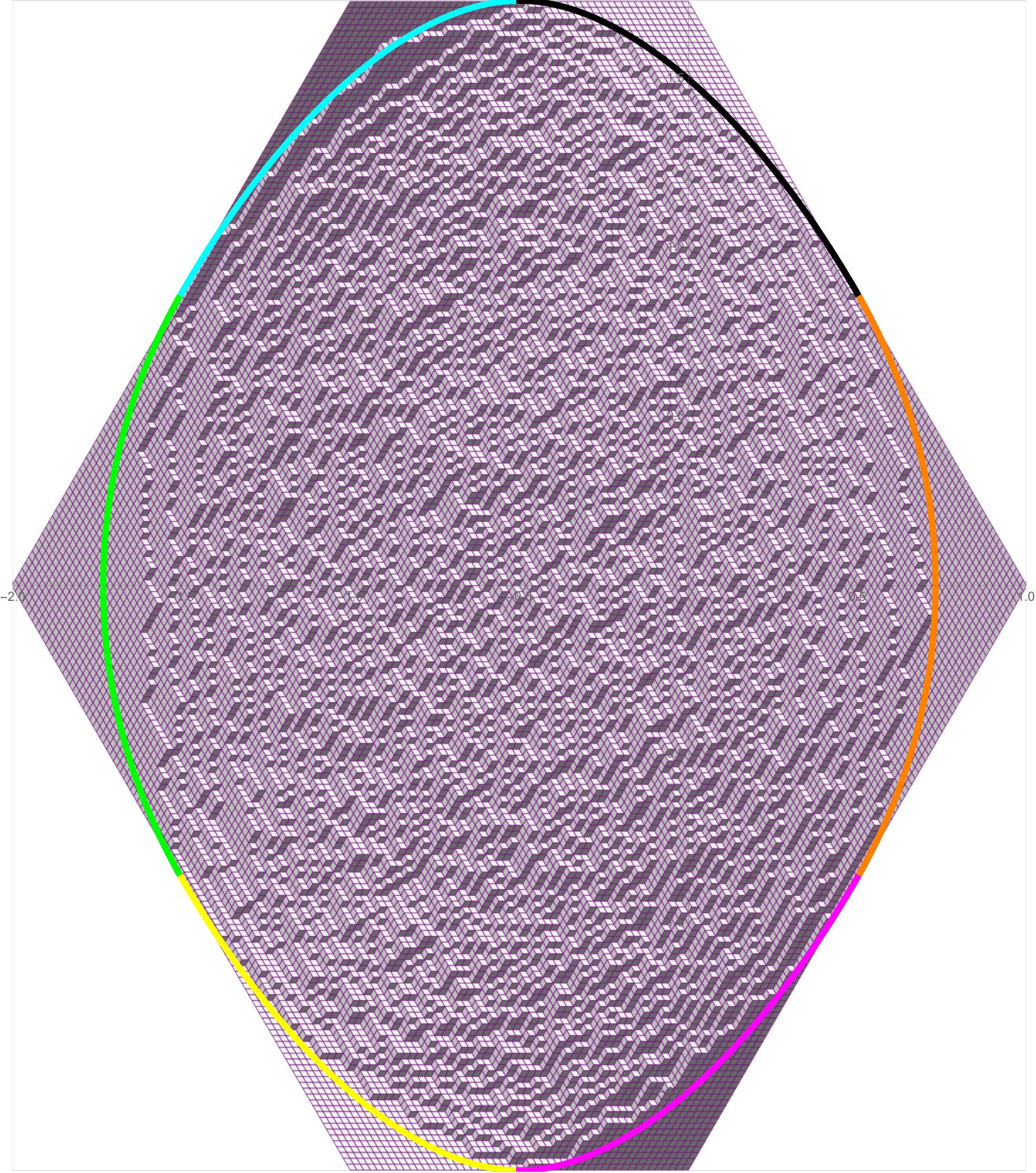}
    \end{tabular}
    }
    \caption{A simulation and computed arctic curve for $2$-tiling a $100 \times 50 \times 100$ hexagon for large $t$ ($t=100$).}
    \label{fig:ArcticCurveLozenget100}
\end{figure}

\begin{table}[ht]
\begin{center}
\begin{tabular}{|c|c|c|c|}\hline
a& b& c& Generating polynomial  \\ \hline \hline
1& 1& 1&  $3t + 1$ \\
1& 1& 2&  $6t + 3$ \\
1& 1& 3& $10t + 6$ \\
1& 2& 1&  $5t^2 + 3t + 1$ \\
1& 2& 2&  $15t^2 + 15t + 6$ \\
1& 2& 3& $35t^2+45t+20$ \\
2& 1& 1&  $3t(2t + 1)$ \\
2& 1& 2&  $20t^2 + 15t + 1$ \\
2& 1& 3& $50t^2 + 45t + 5$ \\
2& 2& 1&  $t^2(15t^2 + 15t + 6)$ \\
2& 2& 2&  $105t^4 + 175t^3 + 104t^2 + 15t + 1$ \\
2& 2& 3& $490t^4 + 1050t^3 + 770t^2 + 175t + 15$ \\
3& 1& 1&  $t^2(10t + 6)$ \\
3& 1& 2&  $5t(10t^2 + 9t + 1)$ \\
3& 1& 3& $175t^3 + 189t^2 + 35t + 1$ \\
3& 2& 1&  $t^4(35t^2 + 45t + 20)$ \\
3& 2& 2&  $t^2(490t^4 + 1050t^3 + 770t^2 + 175t + 15)$\\ 
3& 2& 3&  $4116t^6 + 11340t^5 + 10689t^4 + 3850t^3 + 594t^2 + 35t + 1$ \\ \hline
\end{tabular}\end{center}
\caption{Generating polynomial of 2-tilings of the $a\times b\times c$ hexagon}
\label{ex:tab}
\end{table}

\section{Appendix: Equivalence of algebraic and graphical definitions for the $L$ and $M$ vertices}
\label{ap:b}

Recall the algebraic definition of the $L$ matrix:
\[ L_x^{(k)}(\I,\J,\K,\L) = \textbf{1}_{\I+\J=\K+\L} \prod_{i=1}^k \textbf{1}_{I_i+J_i \neq 2} \cdot x^{|\L|} t^{\varphi(\L,\I+\J)}. \]
Due to the factor of $\textbf{1}_{\I+\J=\K+\L} \prod_{i=1}^k \textbf{1}_{I_i+J_i \neq 2}$, in order for the weight to be non-zero, we require $I_i+J_i=K_i+L_i$ and $I_i+J_i \neq 2$ for all $i \in [k]$.  In terms of our graphical interpretation, this means that each color must have one of the following five forms.
\[ \resizebox{0.5\textwidth}{!}{
\begin{tabular}{ccccc}
    \begin{tikzpicture}[baseline=(current bounding box.center)]\draw[thin] (0,0) rectangle (1,1);\end{tikzpicture} & \begin{tikzpicture}[baseline=(current bounding box.center)]\draw[thin] (0,0) rectangle (1,1);\draw[blue, thick] (0.5,0)--(0.5,0.5)--(1,0.5);\end{tikzpicture} & \begin{tikzpicture}[baseline=(current bounding box.center)] \draw[thin] (0,0) rectangle (1,1); \draw[blue, thick] (0,0.5)--(1,0.5); \end{tikzpicture} & \begin{tikzpicture}[baseline=(current bounding box.center)] \draw[thin] (0,0) rectangle (1,1); \draw[blue, thick] (0.5,0)--(0.5,1); \end{tikzpicture} & \begin{tikzpicture}[baseline=(current bounding box.center)] \draw[thin] (0,0) rectangle (1,1); \draw[blue, thick] (0,0.5)--(0.5,0.5)--(0.5,1); \end{tikzpicture} \\ A & B & C & D & E
\end{tabular}}
\]
Note that $L_i = 1$ if color $i$ has form B or C (i.e. color $i$ exits right) and 0 otherwise.  Also note that $I_i+J_i = 1$ if color $i$ has form B, C, D, or E (i.e. color $i$ is present) and 0 otherwise.  Assuming each color has one of these five forms, the weight is
\[ \begin{aligned}
x^{|\L|} t^{\varphi(\L,\I+\J)} &= x^{\sum_{i=1}^k L_i} t^{\sum_{i=1}^k \left ( L_i \sum_{j=i+1}^k (I_j+J_j) \right )} = \prod_{\substack{1 \leq i \leq k \\ L_i = 1}} xt^{\sum_{j=i+1}^k (I_j+J_j)} = \prod_{\substack{1 \leq i \leq k \\ \text{color $i$ exits right}}} xt^{\delta_i}
\end{aligned} \]
where $\delta_i$ is the number of colors greater than $i$ that are present.  It is easy to see that this matches the graphical definition of the $L$ matrix.

Now recall the algebraic definition of the $M$ matrix:
\[ \M_x^{(k)}(\I,\J,\K,\L) = x^kt^{\binom{k}{2}}L^{(k)}_{\bar x}(\I,\J,\K,\L) \]
where $\bar{x} = \frac{1}{xt^{k-1}}$.  In order for $\M_x^{(k)}(\I,\J,\K,\L)$ to be non-zero, we need $L^{(k)}_{\bar x}(\I,\J,\K,\L)$ to be non-zero, which requires each color to have form A, B, C, D, or E.  Assuming each color has one of these five forms, the weight is
\[ \begin{aligned}
x^kt^{\binom{k}{2}} \cdot L^{(k)}_{\bar x}(\I,\J,\K,\L) = x^k t^{\binom{k}{2}} \prod_{\substack{1 \leq i \leq k \\ \text{color $i$ exits right}}} \frac{1}{xt^{k-1}} t^{\delta_i}
\end{aligned} \]
which has the form $x^p t^q$ for some $p,q \in \mathbb{Z}$.  We see that
\[ p = k - \text{\# colors that exit right} = \text{\# colors that don't exit right} \]
and
\[ \begin{aligned}
q &= \binom{k}{2} + \sum_{\substack{1 \leq i \leq k \\ \text{color $i$ exits right}}} \delta_i - \sum_{\substack{1 \leq i \leq k \\ \text{color $i$ exits right}}} (k-1) \\
&= \sum_{1 \leq i < j \leq k} 1 + \sum_{\substack{1 \leq i < j \leq k \\ \text{color $i$ exits right} \\ \text{color $j$ is present}}} 1 - \left ( \sum_{\substack{1 \leq i < j \leq k \\ \text{color $i$ exits right}}} 1 + \sum_{\substack{1 \leq h < i \leq k \\ \text{color $i$ exits right}}} 1 \right ) \\
&= \sum_{1 \leq i < j \leq k} 1 + \sum_{\substack{1 \leq i < j \leq k \\ \text{color $i$ exits right} \\ \text{color $j$ is present}}} 1 - \left ( \sum_{\substack{1 \leq i < j \leq k \\ \text{color $i$ exits right}}} 1 + \sum_{\substack{1 \leq i < j \leq k \\ \text{color $j$ exits right}}} 1 \right ) \\
&= \sum_{1 \leq i < j \leq k} \left \{ \begin{array}{ll} 
0=1+0-(1+0) & \text{$i$ right, $j$ not present} \\ 
1=1+1-(1+0) & \text{$i$ right, $j$ present but not right} \\ 
0=1+1-(1+1) & \text{$i$ right, $j$ right} \\ 
1=1+0-(0+0) & \text{$i$ not right, $j$ not present} \\ 
1=1+0-(0+0) & \text{$i$ not right, $j$ present but not right} \\ 
0=1+0-(0+1) & \text{$i$ not right, $j$ right}  
\end{array} \right . \\
&= \sum_{\substack{1 \leq i < j \leq k \\ \text{color $i$ doesn't exit right}}} \alpha_i + \sum_{\substack{1 \leq i < j \leq k \\ \text{color $i$ exits right}}} \beta_i
\end{aligned} \]
where 
\[ \begin{aligned}
&\alpha_i = \text{\# colors $j>i$ that don't exit right}, \\
&\beta_i = \text{\# colors $j>i$ that are present but don't exit right} = \text{\# colors $j>i$ that exit top}.
\end{aligned} \]
Thus the weight is
\[ \begin{aligned}
x^p t^q = \prod_{\substack{1 \leq i < j \leq k \\ \text{color $i$ doesn't exit right}}} xt^{\alpha_i} \cdot \prod_{\substack{1 \leq i < j \leq k \\ \text{color $i$ exits right}}} t^{\beta_i}.
\end{aligned} \]
It is easy to see that this matches the graphical definition of the $\M$ matrix.

\bibliographystyle{abbrv}
\bibliography{Aztec}

\end{document}